\documentclass[12pt,a4paper]{amsart}
\usepackage{verbatim}
\usepackage{amsmath,amssymb,amsthm}             %数学公式
\usepackage{amssymb}             %数学公式
\usepackage{amsfonts}            %数学字体
\usepackage{mathrsfs}            %数学花体 
\usepackage{amsthm}
\usepackage{algorithm}  
\usepackage{algorithmicx}  
\usepackage{algpseudocode}  
\usepackage{mathtools}
\usepackage{commath}
\usepackage{physics}
\usepackage{bm}
\usepackage{graphicx}
\usepackage{geometry}
\usepackage{float}
\usepackage{listings}
\usepackage{subfigure}
\usepackage{multirow}
\usepackage{diagbox}
\usepackage{color}
\usepackage{bbm}
\usepackage[export]{adjustbox}
\usepackage{enumerate}
\usepackage{bbm}

\usepackage[colorlinks,linkcolor=blue,anchorcolor=blue,linktocpage=true,urlcolor=blue,citecolor=blue]{hyperref}

\title[]{Continuity and uniqueness of percolation critical parameters in Finitary Random Interlacements}

\numberwithin{equation}{section}

\author[]{Zhenhao Cai}
\author[]{Eviatar B. Procaccia}
\author[]{Yuan Zhang}
\address{(Zhenhao Cai) Peking University}
\email{caizhenhao@pku.edu.cn}
\address{(Eviatar B. Procaccia) Technion -- Israel Institute of Technology}
\email{eviatarp@technion.ac.il }
\address{(Yuan Zhang) Peking University}
\email{zhangyuan@math.pku.edu.cn}

\newtheorem{theorem}{Theorem}
\newtheorem{lemma}{Lemma}[section]
\newtheorem{definition}[lemma]{Definition}%[section]

\newtheorem{proposition}[lemma]{Proposition}%[section]
\newtheorem{corollary}{Corollary}

\newtheorem*{remark}{Remark}
\newcommand{\RNum}[1]{\uppercase\expandafter{\romannumeral #1\relax}}

\newcommand{\supp}{\text{support}}

\begin{document}
	\maketitle
	
\begin{abstract}
We prove that the critical percolation parameter for Finitary Random Interlacements (FRI) is continuous with respect to the path length parameter $T$. The proof uses a result which is interesting on its own right; equality of natural critical parameters for FRI percolation phase transition. 
\end{abstract}
	
	\tableofcontents

\section{Introduction}

%%%%%%%%%%%%%%%%%%%%%%%%%%%%%%%%

Finitary Random Interlacements (FRI) was introduced by Bowen \cite{bowen2019finitary} as a tool to answer a problem of Gaboriau and Lyons, who asked whether every non-amenable measured equivalence relation contains a non-amenable treeable subequivalence relation. FRI is a Poisson point process of geometrically killed random walk paths. There are two parameters controlling the model, an intensity parameter $u>0$, and the killing parameter $T>0$ controlling the length of the finite paths. In the same paper Bowen asks about the FRI on $\mathbb{Z}^d$ percolation properties as a function of $u$ and $T$. It was proved in \cite{procaccia2021percolation} that for any $u>0$ and large enough $T$ there is a unique infinite component and that for small enough $T$ all connected components are finite. In \cite{cai2021non} it was proved that the model is not monotonic as a function of $T$, making it harder to prove a sharp phase transition.  

This paper delves into the percolative properties of FRI for a fixed $T$ and changing intensity parameter $u$. This approach was considered in \cite{prevost2020percolation} in the context of the massive Gaussian free field, and in \cite{cai2021some} where the global existence of a critical intensity $u_*(T)$ was proved. Note that FRI can be considered as a massive random interlacements with killing measure.

In this paper we prove that that the critical parameter for percolation $u_*(T)$ is continuous as a function of $T$. To achieve that we first prove that various natural percolation critical parameters are all equal to each other. By proving that the critical intensity for percolation equals that of the so called local uniqueness phase (denoted by $\bar{u}(T)$ here), one immediately gets that results such as good chemical distance \cite{cai2020chemical} and isoperimery \cite{procaccia2016quenched} hold for all the super-ciritical phase.   

The proof of the unique critical parameter in this paper, follows the general scheme of \cite{duminil2020equality}. We try to avoid repeating arguments from \cite{duminil2020equality}, however since many of the details are different we do need to reprove some of their Lemmas. For instance their truncation of the Gaussian free field is replaced here with a restriction of the FRI Poisson point process to paths of length smaller than a prescribed parameter.

The FRI with parameters $u$ and $T$ naturally scales to Random Interlacements of level $u$ as $T\to\infty$ \cite[Appendix]{teixeira2013random}. Naturally the set of vertices not visited by the FRI scales as $T\to\infty$ to the so called vacant set of random interlacements. It would be interesting to extend the results of this paper to the case of the vacant set of Random Interlacements, and prove uniqueness of the percolation critical parameters in that case. Moreover the continuity proved in this paper suggests continuity of the chemical distance norm for FRI with $u>u_*(T)$, and its convergence to the chemical distance norm of random interlacements, in the limit as $T$ grows to infinity. It is expected that the ball in the chemical distance norm of random interlecements, scaled correctly, would converge to a Euclidean ball as the intensity parameter converges to zero.

	\subsection{Notations and Preliminaries}\label{notations} 
	Before presenting our results formally, we need to introduce some notations and useful results.

	\textbf{Graph $(\mathbb{Z}^d,\mathbb{L}^d)$ and two metrics:} We denote by $\mathbb{Z}^d$ the d-dimensional lattice. We also denote the $l^{\infty}$ and $l^1$ distances on $\mathbb{Z}^d$ by $|\cdot|$ and $|\cdot|_1$ respectively. Precisely, for any $x=(x^{(1)},...,x^{(d)})\in \mathbb{Z}^d$, $|x|:=\max_{1\le i\le d} |x^{(i)}|$ and $|x|_1:=\sum_{i=1}^d |x^{(i)}|$. The set of undirected edges on $\mathbb{Z}^d$ is denoted by $\mathbb{L}^d$ (i.e. $\mathbb{L}^d:=\left\lbrace e=\{x,y\}:x,y\in \mathbb{Z}^d\ s.t.\ |x-y|_1=1\right\rbrace $).
	
	\textbf{Definition of FRI:}  Let $W^{\left[0,\infty \right) }$ be the set containing all nearest-neighbor paths of finite lengths on $\mathbb{Z}^d$. For any $x\in \mathbb{Z}^d$ and $0<T<\infty$, we denote by $P_x^{(T)}$ the law of geometrically killed
	simple random walks on $\mathbb{Z}^d$ with starting point $x$ and killing rate $\frac{1}{T+1}$ at each step. Consider the measure $v^{(T)}\times \lambda^+$ on $W^{\left[ 0,\infty \right) }\times \left[ 0,\infty \right) $, where $v^{(T)}:=\sum_{x\in \mathbb{Z}^d}\frac{2d}{T+1}P_x^{(T)}$ and $\lambda^+$ is the Lebesgue measure on $\left[ 0,\infty \right) $. Since the set $W^{\left[0,\infty \right) }$ is countable and for each $\eta\in W^{\left[0,\infty \right) }$, $v^{(T)}(\eta)\le \frac{2d}{T+1} $, the measure $v^{(T)}\times \lambda^+$ is $\sigma$-finite. Referring to Definition 4.4 in \cite{drewitz2014introduction}, let $\mathcal{FI}^{T}$ be the Poisson point process on $W^{\left[0,\infty \right) }\times \left[0,\infty \right) $ with intensity measure $v^{(T)}\times \lambda^+$. We hereby present the first definition of FRI:
		\begin{definition}\label{def1}
			For $T\in (0,\infty)$, assume that $\mathcal{FI}^{T}:=\sum_{i\in \mathbb{N}}\delta_{(\eta_i,u_i)}$. Then for any $u>0$, finitary random interlacements with expected fiber length $T$ and level $u$ are defined as \begin{equation}
				\mathcal{FI}^{u,T}:=\sum_{i\in \mathbb{N}:u_i\le u} \delta_{\eta_i}. 
			\end{equation} 
		\end{definition}
		By Definition \ref{def1}, we note that $\mathcal{FI}^{u,T}$ is a point measure on $W^{\left[0,\infty \right) }$ and for any $u_1<u_2$, under the natural coupling one always has $\supp(\mathcal{FI}^{u_1,T})\subset \supp(\mathcal{FI}^{u_2,T}$).

		When we consider a single FRI $\mathcal{FI}^{u,T}$, an alternative definition of FRI introduced in \cite{procaccia2021percolation} is also useful:
	\begin{definition}\label{def2}
		Let $\left\lbrace N_x \right\rbrace_{x\in \mathbb{Z}^d}$ be a sequence of i.i.d. Poisson randon variables with parameter $\frac{2du}{T+1}$. For each $x\in \mathbb{Z}^d$, independently sample $N_x$ random walks by $P_x^{(T)}$. Then let $\mathcal{FI}^{u,T}$ be the point measure on $W^{\left[0,\infty \right) }$ consisting of all the trajectories above from every $x\in \mathbb{Z}^d$.  
	\end{definition}

	\textbf{Edge sets and sets of vertices}: When we write a set of vertices, we will use capital letters in the normal style such as $A,B,E$, etc. For edge sets, we use Calligraphic typeface such as $\mathcal{A}$, $\mathcal{B}$ and $\mathcal{E}$, or Greek capital letters such as $\Lambda$ and $\Sigma$. Especially, for $x\in \mathbb{Z}^d$ and $R\in \mathbb{N}^+$, we denote $B_x(R):=\left\lbrace y\in \mathbb{Z}^d:|x-y|\le R \right\rbrace $ and $\mathcal{B}_x(R):=\left\lbrace e=\{x,y\}\in \mathbb{L}^d:x,y\in B_x(R) \right\rbrace $. For an edge set $\mathcal{A}$, we write the set of all vertices covered by $\mathcal{A}$ as $V(\mathcal{A})$ (i.e. $V(\mathcal{A}):=\left\lbrace x\in \mathbb{Z}^d:there\ exists\ e\in \mathcal{A}\ s.t.\ x\in e \right\rbrace $).

	For any $\eta=(\eta(0),...,\eta(m))\in W^{\left[0,\infty \right) }$ ($m\ge 1$), $\eta$ can also be regarded as an edge set $\left\lbrace \{\eta(i),\eta(i+1)\} \right\rbrace_{i=1}^{m-1}\subset \mathbb{L}^d$. Hence, sometimes we no longer distinguish between these notations. In this way, a point measure $\omega=\sum_{j\in I}\delta_{\eta_j} $ on $W^{\left[0,\infty \right) }$ is equivalent to a unique edge set $\cup_{j\in I} \eta_j \subset \mathbb{L}^d$ and we similarly do not make  distinction between them. 
	
	Especially, for each nearest-neighbor path of finite length  $\eta=(\eta(0),...,\eta(m))$, the length of $\eta$ is defined as $|\eta|:=m$. 
	
	\
	
\textbf{Ordering relation in vertices, edges, edge sets and sets of vertices:} For any $x_1,x_2\in \mathbb{Z}^d$, we say $x_1$ is lexicographically-smaller than $x_2$ (denoted by $x_1\lhd x_2$) if there exists $1\le i\le d$ such that $x_1^{(i)}<x_2^{(i)}$ and for all $1\le j<i$, $x_1^{(j)}=x_2^{(j)}$. Furthermore, for $e_1=\{x_1,y_1\},e_2=\{x_2,y_2\}\in \mathbb{L}^d$ ($x_1\lhd y_1$, $x_2\lhd y_2$), we also say $e_1$ is lexicographically-smaller than $e_2$ (denoted by $e_1\lhd e_2$) if either of the following condition holds: (1) $x_1\lhd x_2$; (2) $x_1=x_2$ and $y_1\lhd y_2$.

		Based on the ordering relation given above, edge sets and sets of vertices can also be ordered as follows: for any $A=\{x_1,...,x_m\},B=(y_1,...,y_n)\subset \mathbb{Z}^d$ ($x_1\lhd ...\lhd x_m$, $y_1\lhd ...\lhd y_n$), we say $A$ is lexicographically-smaller than $B$, and denote it by $A\lhd B$, if $A\subset B$, or $A\not\subset B$ while $x_{i_0}\lhd y_{i_0}$, where $i_0:=\min\{i:x_i\neq y_i\}$. Similarly, for $\mathcal{A}=\{e_1,...,e_m\},\mathcal{B}=(e_1',...,e_n')\subset \mathbb{Z}^d$ ($e_1\lhd ...\lhd e_m$, $e_1'\lhd ...\lhd e_n'$), we also say $\mathcal{A}$ is lexicographically-smaller than $\mathcal{B}$ ($\mathcal{A}\lhd \mathcal{B}$) if $\mathcal{A}\subset \mathcal{B}$, or $\mathcal{A}\not\subset \mathcal{B}$ while $e_{j_0}\lhd e'_{j_0}$, where $j_0:=\min\{j:e_j\neq e'_j\}$.

	\
	
	\textbf{Several kinds of boundaries:} For each set of vertices $A$, we denote its internal boundary and outer boundary by $\partial  A$ and $\partial^{out} A$, where 
	\begin{equation}
		\partial A:=\left\lbrace x\in A:there\ exists\ y\in \mathbb{Z}^d\setminus A\ s.t.\ |x-y|_1=1 \right\rbrace,
	\end{equation}
	\begin{equation}
		\partial^{out} A:=\left\lbrace x\in \mathbb{Z}^d\setminus A:there\ exists\ y\in A\ s.t.\ |x-y|_1=1 \right\rbrace.
	\end{equation}

	Similarly, for an edge set $\mathcal{A}$, we write $\partial_e \mathcal{A}$ and $\partial^{out}_e \mathcal{A}$ for its internal edge boundary and outer edge boundary. Precisely, 
		\begin{equation}
			\partial_e \mathcal{A}:=\left\lbrace e=\{x,y\}\in \mathcal{A}:x,y\in \partial V(\mathcal{A}) \right\rbrace,
		\end{equation}
		\begin{equation}\label{4.4}
			\partial^{out}_e \mathcal{A}:=\left\lbrace e=\{x,y\}\in \mathbb{L}^d\setminus \mathcal{A}:x\in \partial V(\mathcal{A}),y\in \partial^{out} V(\mathcal{A}) \right\rbrace.
	\end{equation}

	\textbf{Distance and diameter:} Based on the metrics $|\cdot|$ and $|\cdot|_1$, one can define distances accordingly: for any $A,B\subset \mathbb{Z}^d$, $d(A,B):=\min_{x\in A,y\in B} |x-y|$ and $d_1(A,B):=\min_{x\in A,y\in B} |x-y|_1$; for edge sets $\mathcal{A},\mathcal{B}\subset \mathbb{L}^d$, $d(\mathcal{A},\mathcal{B}):=d(V(\mathcal{A}),V(\mathcal{B}))$ and $d_1(\mathcal{A},\mathcal{B}):=d_1(V(\mathcal{A}),V(\mathcal{B}))$.
	
	For each subset $E\subset \mathbb{Z}^d$, the diameter of $E$ is denoted by $\text{diam}(E):=\max_{x,y\in E}|x-y|$. Similarly, for $\mathcal{E}\subset \mathbb{L}^d$, $\text{diam}(\mathcal{E}):=\text{diam}(V(\mathcal{E}))$. 
	
	\
	
	\textbf{Connection between two sets:} For any $A,B\subset \mathbb{Z}^d$ and $\mathcal{E}\subset \mathbb{L}^d$, we say $A$ and $B$ are connected by $\mathcal{E}$ if there exists a path $\eta=(\eta(0),...,\eta(m))\in W^{\left[0,\infty \right) }$ such that $\eta\subset \mathcal{E}$, $\eta(0)\in A$ and $\eta(m)\in B$. We denote this event by $A\xleftrightarrow[]{\mathcal{E}} B$. We also say ``$A$ and $B$ are connected by $\mathcal{E}$ within a set of vertices $D$'' (denoted by $A\xleftrightarrow[D]{\mathcal{E}} B$) if $A\xleftrightarrow[]{\mathcal{E}} B$ and one of the connecting paths $\eta$ further satisfies that $V(\eta)\subset D$.

%	there exists a path $\eta=(\eta(0),...,\eta(m))\in W^{\left[0,\infty \right) }$ satisfying $\eta\subset \mathcal{E}$, $V(\eta)\subset D$, $\eta(0)\in A$ and $\eta(m)\in B$. 
	
	For any $A\subset \mathbb{Z}^d$ and $\mathcal{E}\subset \mathbb{L}^d$, we denote by $ A\xleftrightarrow[]{\mathcal{E}} \infty$ the event that there exist an infinite connected component (sometimes, connected components are also called clusters) in $\mathcal{E}$ intersecting $A$.

	For convenience, sometimes we also write $\mathcal{A} \xleftrightarrow[]{\mathcal{E}} \mathcal{B}:=V(\mathcal{A}) \xleftrightarrow[]{\mathcal{E}} V(\mathcal{B})$, $\mathcal{A} \xleftrightarrow[D]{\mathcal{E}} \mathcal{B}:=V(\mathcal{A}) \xleftrightarrow[D ]{\mathcal{E}} V(\mathcal{B})$ and $\mathcal{A}\xleftrightarrow[]{\mathcal{E}} \infty:=V(\mathcal{A})\xleftrightarrow[]{\mathcal{E}} \infty$ for $\mathcal{A},\mathcal{B},\mathcal{E}\subset \mathbb{L}^d$ and $D\subset \mathbb{Z}^d$. 
	
	\
	
	\textbf{Decomposition of FRI:} We introduce a restriction of $\mathcal{FI}^{u,T}=\sum_{i=1}^{\infty}\delta_{\eta_i}$: 
		
		For any $L\in \mathbb{N}^+$, let 
		\begin{equation}\label{2.1}
			\mathcal{FI}^{u,T}_L=\sum_{i=1}^{\infty}\delta_{\eta_i}\cdot \mathbbm{1}_{|\eta_i|\le L}.
		\end{equation} 
	be the collection of trajectories of length $\le L$.

%	We introduce a decomposition for $\mathcal{FI}^{u,T}=\sum_{i=1}^{\infty}\delta_{\eta_i}$: 
%	
%	For any $y\in \mathbb{Z}^d$ and $l \in \mathbb{N}^+$, let
%	\begin{equation}
%		\mathcal{FI}^{u,T}_{l,y}=\sum_{i=1}^{\infty}\delta_{\eta_i}\cdot\mathbbm{1}_{|\eta_i|=l,\eta_i(0)=y}.
%	\end{equation}
%	
%	%\note{This notation is never used} 
%	
%	
%	For edge set $\mathcal{A} \subset \mathbb{Z}^d$, we denote the collection of $\mathcal{FI}^{u,T}_{l,y}$ that may intersect $\mathcal{A}$ by $\textbf{Z}(\mathcal{A})$. To be precise, 
%	\begin{equation}
%		\textbf{Z}(\mathcal{A}):=\left\lbrace \mathcal{FI}^{u,T}_{l,y}:(y,l)\in \mathbb{Z}^d\times \mathbb{N}^+,d_{1}(\{y\},V(\mathcal{A}))\le l-1 \right\rbrace.  
%	\end{equation}
%	
%	We also introduce a restriction of FRI: For any $L\in \mathbb{N}^+$, let
%		\begin{equation}\label{2.1}
%			\mathcal{FI}^{u,T}_L=\sum_{y\in \mathbb{Z}^d}\sum_{1\le l\le L}\mathcal{FI}^{u,T}_{l,y}.
%	\end{equation} 

	\textbf{Critical values of FRI:} Similar to \cite{duminil2020equality}, we consider the following types of critical values: 
	
	For any fixed $d\ge 3$, \begin{itemize}
		\item $u_*(T):=\sup\left\lbrace u>0:P\left[0\xleftrightarrow[]{\mathcal{FI}^{u,T}}\infty \right]=0  \right\rbrace $, which is the most common critical value for percolation and is previously studied in \cite{cai2021non,cai2021some,procaccia2021percolation}. Especially, Theorem 4 in \cite{cai2021some} shows that $u_*(T)\in (0,\infty)$ for all $0<T<\infty$;

		\item $u_{**}(T):=\sup\left\lbrace u>0:\inf\limits_{R\in \mathbb{N}^+}P\left[B(R)\xleftrightarrow[]{\mathcal{FI}^{u,T}}\partial B(2R) \right]=0  \right\rbrace $;
		
		\item $\bar{u}(T):=\inf\left\lbrace u_0>0: \text{for all } u>u_0, \mathcal{FI}^{u,T}\text{ strongly percolates}\right\rbrace $, where ``$\mathcal{FI}^{u,T}$ strongly percolates'' means that there exists $c,\rho>0$ such that for any integer $R>0$,
		\begin{equation}\label{exi}
			\begin{split}
				P\left[\text{Exist}(R,\mathcal{FI}^{u,T})\right]
				:=P[&\text{there exists a cluster in }\mathcal{B}(R)\cap \mathcal{FI}^{u,T}\text{ with} \\
				&\text{diameter at least } R/5]\ge 1-e^{-cR^\rho},
			\end{split}
		\end{equation}
		
		\begin{equation}\label{2.2}
			\begin{split}
				P\left[\text{Unique}(R,\mathcal{FI}^{u,T})\right]
				:=P[&\text{all clusters in } \mathcal{B}(R)\cap \mathcal{FI}^{u,T}\text{ with diameter at least} \\
				&\frac{R}{10}\text{ are connected by } \mathcal{B}(2R)\cap \mathcal{FI}^{u,T}]\ge 1-e^{-cR^\rho};
			\end{split}
		\end{equation}

		\item $\widetilde{u}(T)=\inf\left\lbrace u>0 :\inf\limits_{R\in \mathbb{N}}R^dP\left[B(\mu(R))\stackrel{\mathcal{FI}^{u,T}}{\nleftrightarrow} \partial B(R) \right]=0 \right\rbrace $, where $\mu(R):=\left\lfloor e^{(\log(R))^{\frac{1}{3}}}\right\rfloor $. 
		
	\end{itemize}
	
	\section{Statements of Results}
	Our main result is the continuity of the critical parameter with respect to the killing parameter $T$.
	\begin{theorem}\label{continuity}
		For any $d\ge 3$, $u_*(T)$ is a continuous function w.r.t. $T$ on $(0,\infty)$.
	\end{theorem}
	The main tool we use is a theorem which is parallel to Theorem 1.1 of \cite{duminil2020equality}. We prove that all critical values introduced above are actually equivalent to each other.
	
	\begin{theorem}\label{equality}
		For any $d\ge 3$ and $T>0$, 
		\begin{equation}
			u_*(T)=u_{**}(T)=\bar{u}(T)=\widetilde{u}(T). 
		\end{equation}
	\end{theorem}
	
%	Based on Theorem \ref{equality}, we can get continuity of the function $u_*(T)$. 
%	
%	\begin{corollary}\label{continuity}
%		For any $d\ge 3$, $u_*(T)$ is a continuous function w.r.t. $T$ on $(0,\infty)$.
%	\end{corollary}
%	
	
	Referring to \cite{aizenman1987sharpness,duminil2019sharp,duminil2016new,grimmett1999percolation}, for a model with phase transition, if in the subcritical regime, the cluster size distribution decays exponentially, then this model is considered to have sharp phase transition. By Theorem \ref{equality}, we can show that the phase transition of FRI is sharp.

	\begin{corollary}\label{sharpness}
		For any $d\ge 3$ and $(u,T)\in (0,\infty)^2$, if $u>u_*(T)$, $\mathcal{FI}^{u,T}$ strongly percolates; if $u<u_*(T)$, there exists $C(u,T)>0$ and $\delta(u,T)>0$ such that for any integer $R>0$, 
		\begin{equation}
			P^{u,T}\left[0\xleftrightarrow[]{\mathcal{FI}^{u,T}} \partial B(R) \right]\le Ce^{-\delta R}. 
		\end{equation}
	\end{corollary}
	
	Theorem 2 of \cite{cai2020chemical} shows that for any $d\ge 3$ and $u>0$, there exists a constant $T_3(u,d)>0$ such that for any $T>T_3$, $\mathcal{FI}^{u,T}$ strongly percolates. Indeed, Corollary \ref{sharpness} implies the following stronger results and gives a more precise expression for the constant $T_3$ in Theorem 2, \cite{cai2020chemical}. 
	
	For a fixed $d\ge 3$, we introduce the following critical value: for $u>0$,
	\begin{equation}\label{T*+}
		T_*^+(u):=\inf\left\lbrace T_0\in \mathbb{R}^+:\text{for  all } T>T_0,u>u_*(T) \right\rbrace. 
	\end{equation}
	
	\begin{corollary}\label{coro3}
		For any $u>0$, $T_*^+(u)\in (0,\infty)$ and if $T>T_*^+(u)$, $\mathcal{FI}^{u,T}$ strongly percolates.
	\end{corollary} 
	
	Indeed, Corollary \ref{coro3} can improve the constant $T_3$ in Theorem 2 of \cite{cai2020chemical} to  $T_*^+(u)$. Furthermore, contants $T_2$ in Corollary 2.1 of \cite{cai2020chemical} and $T_5$ in Corollary 2.2 of \cite{cai2020chemical} can also be improved to $T_*^+(u)$ by using the arguements in Section 7 of \cite{vcerny2012internal} and Theorem 1.1 of \cite{procaccia2016quenched}.

	\section{Proof of Theorem \ref{continuity} }
	
	In this section, we show the proof of Theorem \ref{continuity} assuming Theorem \ref{equality}.

	For Theorem \ref{continuity}, it is sufficient to prove the following four results: for any $T_0>0$ and $\epsilon\in (0,0.5u_*(T_0))$, there exists $\Delta T_1(T_0,\epsilon),\Delta T_2(T_0,\epsilon),\Delta T_3(T_0,\epsilon),\Delta T_4(T_0,\epsilon)\in (0,T_0)$ such that 
	\begin{enumerate}
		\item for any $T\in [T_0,T_0+\Delta T_1]$, $\mathcal{FI}^{u_*(T_0)+\epsilon,T}$ percolates;
		
		\item for any $T\in [T_0-\Delta T_2,T_0]$, $\mathcal{FI}^{u_*(T_0)-\epsilon,T}$ does not percolate;
		
		\item for any $T\in [T_0,T_0+\Delta T_3]$, $\mathcal{FI}^{u_*(T_0)-\epsilon,T}$ does not percolate;
		
		\item for any $T\in [T_0-\Delta T_4,T_0]$, $\mathcal{FI}^{u_*(T_0)+\epsilon,T}$ percolates.
		
	\end{enumerate}
	
	Among these four parts, Part 1 and Part 2 follow from a coupling construction of FRI's. Part 3 is mainly based on a decoupling type renormalization argument in \cite{rodriguez2013phase,sznitman2012decoupling} and Part 4 uses arguments from \cite{procaccia2016quenched,rath2013effect}.

	Before proving, we need to introduce some notations and results:
	\begin{itemize}
		\item For $K_0,k_0\in \mathbb{N^+}$, let $K_n=K_0\cdot k_0^n$ and $\mathbb{K}_n=K_n\cdot\mathbb{Z}^d$. 
		
		\item For integr $n\ge 0$, let $\mathbb{I}_n=\{n\}\times \mathbb{K}_n$ ($n$ stands for ``level'' in the decoupling argument); for $(n,x)\in \mathbb{I}_n$, $n\ge 1$, define that 
		\begin{equation}
			\mathbb{H}_1(n,x)=\{(n-1,y)\in \mathbb{I}_{n-1}:B_y(K_{n-1})\subset B_x(K_{n}),B_y(K_{n-1})\cap \partial B_x(K_{n})\neq \emptyset \},
		\end{equation}
		\begin{equation}
			\mathbb{H}_2(n,x)=\left\lbrace (n-1,y)\in \mathbb{I}_{n-1}:B_y(K_{n-1})\cap \partial B_x(\lfloor 1.5K_{n-1}\rfloor )\neq \emptyset \right\rbrace,
		\end{equation}
		\begin{equation}
			\begin{split}
				\Upsilon_{n,x}=\{&\mathcal{T}\subset \bigcup_{k=0}^n\mathbb{I}_k:\mathcal{T}\cap \mathbb{I}_n=(n,x);\text{for all } 0<k\le n,(k,y)\in\mathcal{T}\cap \mathbb{I}_k, (k,y)\\
				&\text{ has two descendants } (k-1,y_{i}(k,y))\in \mathbb{H}_{i}(k,y),i=1,2\text{ such that}\\
				&\mathcal{T}\cap\mathbb{I}_{k-1}=\bigcup_{(k,y)\in \mathcal{T}\cap\mathbb{I}_{k}}\{(k-1,y_{1}(k,y)),(k-1,y_{2}(k,y)) \}  \}. 
			\end{split}
		\end{equation}
		Like (2.8) of \cite{rodriguez2013phase}, there exists constant $C(d)>0$ such that for any $n\ge 1$,
		\begin{equation}\label{7.4}
			|\Upsilon_{n,x}|\le \left(C\cdot k_0^{2(d-1)}\right)^{2^n}.  
		\end{equation}

		\item For $x\in \mathbb{K}_n$, define event $A_{n,x}=\{B_x(K_n) \xleftrightarrow[]{\mathcal{FI}^{u,T}}\partial B_x(2K_n)  \}$. Similar to (2.14) in \cite{rodriguez2013phase}, 
		\begin{equation}\label{7.5}
			A_{n,x}\subset \bigcup_{\mathcal{T}\in \Upsilon_{n,x}}A_{\mathcal{T}},
		\end{equation}
		where $A_{\mathcal{T}}:=\bigcap_{(0,y)\in \mathcal{T}\cap \mathbb{I}_0}A_{0,y}$. See Figure \ref{figure4} for an
			illustration of this renormalization scheme. 
		
		\begin{figure}[h]
			\centering
			\includegraphics[width=0.5\textwidth]{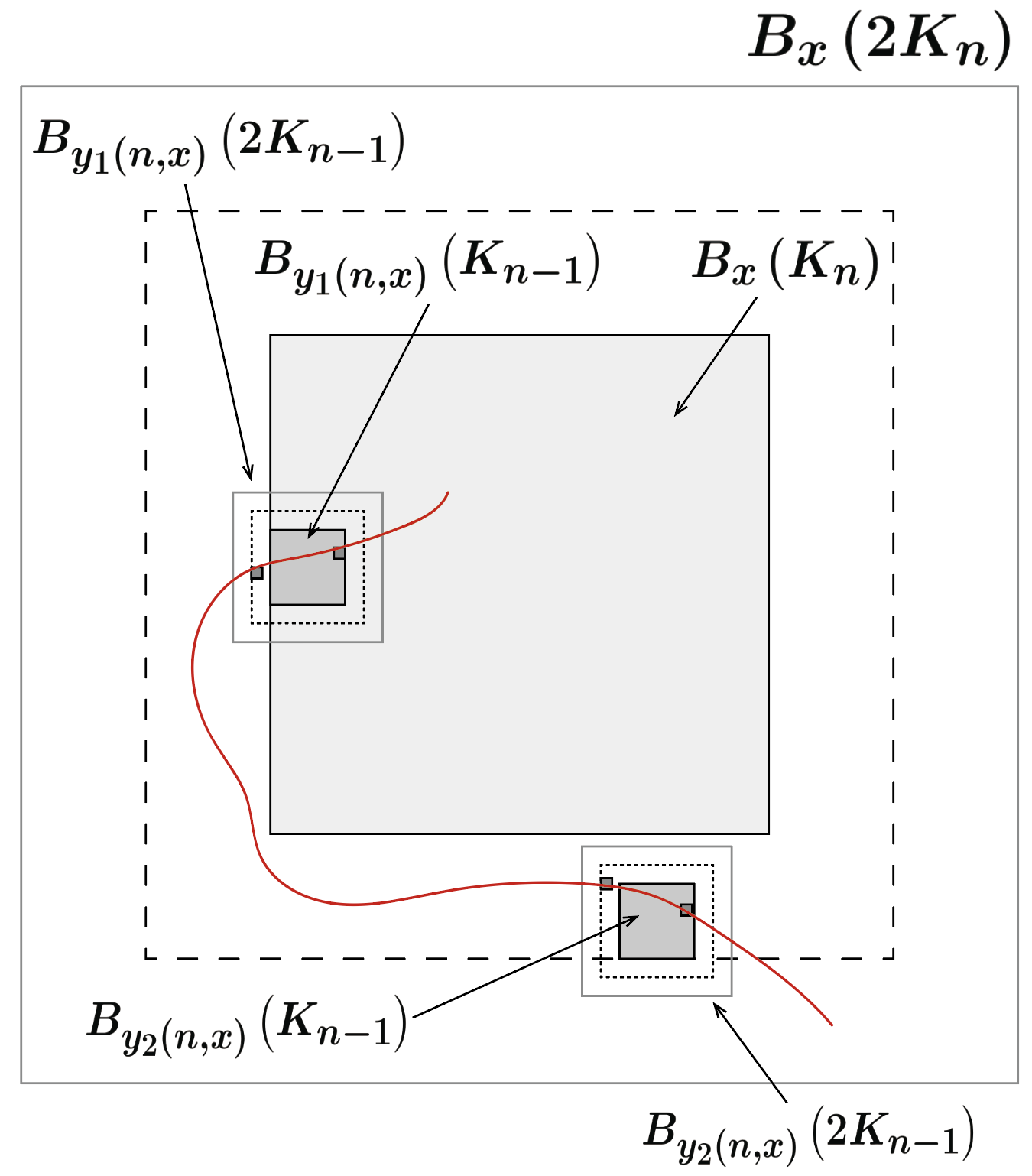}
			\caption{An
				illustration of the renormalization scheme.}
			\label{figure4}
		\end{figure}

		\item In fact, we can couple two geometric random variables by $\{U_i\}_{i=1}^\infty\overset{i.i.d.}{\sim} U[0,1]$: for any $0<p_1<p_2<1$, let $N_i=\min\{n\in \mathbb{N}^+:U_n\ge p_i\}$, $i=1,2$. Then for $i\in \{1,2\}$, $N_i\sim Geo(p_i)$. Note that for all $m\in \mathbb{N}^+$,  
		\begin{equation}\label{7.6}
			P[N_1=N_2|N_2=m]=\frac{	P[N_1=N_2=m]}{P[N_2=m]}=\left(\frac{p_1}{p_2}\right)^{m-1}. 
		\end{equation}

		\item For any $T'>T>0$ and $u>0$, there are two steps to construct $\mathcal{FI}^{u,T}$ from $\mathcal{FI}^{u,T'}$: \begin{enumerate}
			\item For each $\eta=(\eta(0),...,\eta(m)) \in \mathcal{FI}^{u,T'}$, sample the geometric random variable $N^\eta_1$ condition on $N^\eta_2=m$ under the coupling above ($p_1=\frac{T}{T+1},p_2=\frac{T'}{T'+1}$); then let $\widetilde{\eta}=(\eta(0),...,\eta(N^\eta_1))$ and $\widetilde{\mathcal{FI}^{u,T'}}=\sum_{\eta\in \mathcal{FI}^{u,T'}}\delta_{\widetilde{\eta}}$;
			
			\item Sample an independent copy of $\mathcal{FI}^{\frac{u(T'-T)}{T'+1},T}$ and then add it to $\widetilde{\mathcal{FI}^{u,T'}}$.
		\end{enumerate}
		With a slight abuse of notation, we still denote the probability measure of this coupling by $P$. Under this construction, we have: for any $R\in \mathbb{N}^+$, 
		\begin{equation}\label{7.7}
			\begin{split}
				&P\left[\mathcal{FI}^{u,T}\cap \mathcal{B}(R)\neq\mathcal{FI}^{u,T'}\cap \mathcal{B}(R)\right]\\
				\le & P\left[ \mathcal{FI}^{\frac{u(T'-T)}{T'+1},T}\cap  \mathcal{B}(R)\neq\emptyset\right] +P\left[there\ exists\ \eta\in \mathcal{FI}^{u,T'}\ s.t.\ \eta\cap \mathcal{B}(R)=\emptyset,N_1^\eta\neq N_2^\eta\right]. 
			\end{split}
		\end{equation}
		By (\ref{7.6}), (\ref{7.7}) and definition of FRI, under the given coupling between FRI's, we have 
		\begin{equation}\label{7.8}
			\lim\limits_{\Delta T\to 0+}\sup_{T:|T-T_0|\le \Delta T} P\left[\mathcal{FI}^{u,T}\cap \mathcal{B}(R)=\mathcal{FI}^{u,T_0 }\cap\mathcal{B}(R)\right]=1
		\end{equation}
	\end{itemize}

	Now we are ready to prove Theorem \ref{continuity}.

	\begin{proof}[Proof of Theorem \ref{continuity}]
		We separatly show the existence of $\Delta T_1$, $\Delta T_2$, $\Delta T_3$ and $\Delta T_4$. Recall that $T_0>0$ and $\epsilon\in (0,0.5u_*(T_0))$. 
		
		For $\Delta T_1$: Arbitrarily take $\delta\in (0,0.5\epsilon)$, then choose $\Delta T_1>0$ such that $\frac{2d(u_*(T_0)+\epsilon)}{T_0+\Delta T_1+1}= \frac{2d(u_*(T_0)+\delta)}{T_0+1}$. For any $T\in [T_0,T_0+\Delta T_1]$, since $\frac{2d(u_*(T_0)+\epsilon)}{T+1}\ge  \frac{2d(u_*(T_0)+\delta)}{T_0+1}$ and $T\ge T_0$, $\mathcal{FI}^{u_*(T_0)+\epsilon,T}$ stochastically dominates $\mathcal{FI}^{u_*(T_0)+\delta,T_0}$. By definition of $u_*(\cdot)$, we know that $\mathcal{FI}^{u_*(T_0)+\delta,T_0}$ percolates. Hence, we also have $\mathcal{FI}^{u_*(T_0)+\epsilon,T}$ percolates.

		For $\Delta T_2$: Similarly, we take $\delta\in (0,0.5\epsilon)$ arbitrarily and then choose $\Delta T_2>0$ such that $\frac{2d(u_*(T_0)-\epsilon)}{T_0-\Delta T_2+1}= \frac{2d(u_*(T_0)-\delta)}{T_0+1}$. For any $T\in [T_0-\Delta T_2,T_0]$, since $\frac{2d(u_*(T_0)-\epsilon)}{T+1}\le  \frac{2d(u_*(T_0)-\delta)}{T_0+1}$ and $T\le T_0$, we have  $\mathcal{FI}^{u_*(T_0)-\delta,T_0}$ stochastically dominates $\mathcal{FI}^{u_*(T_0)-\epsilon,T}$. Noting that $\mathcal{FI}^{u_*(T_0)-\delta,T_0}$ does not percolate, we know that  $\mathcal{FI}^{u_*(T_0)-\epsilon,T}$ does not percolate either.

		For $\Delta T_3$: For any $n\ge 1$ and $\mathcal{T}\in \Upsilon_{n,0}$, parallel to Equation (5.15) of \cite{cai2021some}, take $k_0>2T_0$, then there exists $c>0$ such that for all $T\in [T_0,2T_0]$, 
		\begin{equation}\label{7.9}
			P(A_{\mathcal{T}})\le \left(P^{u,T}(A_{0,0})+2e^{-cK_0}\right)^{2^n}. 
		\end{equation}
		Combine (\ref{7.4}), (\ref{7.5})and (\ref{7.9}), 
		\begin{equation}\label{7.10}
			P(A_{n,0})\le \left[Ck_0^{2(d-1)}\left( P(A_{0,0})+2e^{-cK_0}\right)  \right]^{2^n}. 
		\end{equation}
		By Theorem \ref{equality}, we have $u_*=u_{**}$. Hence, there exists a large enough integer $K_0$ such that \begin{equation}\label{7.11}
			Ck_0^{2(d-1)}\left( P^{u_*(T_0)-\epsilon,T_0}(A_{0,0})+2e^{-cK_0}\right)<\frac{1}{3}. 
		\end{equation}
		By (\ref{7.8}), there exists $\Delta T_3\in (0,T_0)$ such that for all $T\in [T_0,T_0+\Delta T_3]$, 
		\begin{equation}\label{7.12}
			Ck_0^{2(d-1)}\cdot P\left[\mathcal{FI}^{u_*(T_0)-\epsilon,T}\cap \mathcal{B}(2R)=\mathcal{FI}^{u_*(T_0)-\epsilon,T_0 }\cap\mathcal{B}(2R)\right] <\frac{1}{3}. 
		\end{equation}
		Combining (\ref{7.10}), (\ref{7.11}) and (\ref{7.12}), for any $T\in [T_0,T_0+\Delta T_3]$, we have 
		\begin{equation}
			P\left[B_x(K_n) \xleftrightarrow[]{\mathcal{FI}^{u_*(T_0)-\epsilon,T}}\partial B_x(2K_n)   \right] \le \left(\frac{2}{3} \right)^{n}.
		\end{equation}
		Therefore, $\mathcal{FI}^{u_*(T_0)-\epsilon,T}$ does not percolate for all $T\in [T_0,T_0+\Delta T_3]$.

		For $\Delta T_4$: In this part, we adapt a block construction argument introduced in \cite{procaccia2016quenched,rath2013effect}.

		We say that a box $B_x(K_0)$ is good if both of the following events, denoted by $G_1$ and $G_2$ occur (otherwise, call $B_x(K_0)$ a bad box): 
		\begin{enumerate}
			\item Recall notations $\text{Exist}(\cdot,\cdot)$ and  $\text{Unique}(\cdot,\cdot)$ in (\ref{exi}) and (\ref{2.2}). Define \begin{equation}
				G_1(x):= \text{Exist}(K_0,-x+\mathcal{FI}^{u_{*}(T_0)+\epsilon,T_0})\cap \text{Unique}(K_0,-x+\mathcal{FI}^{u_{*}(T_0)+\epsilon,T_0}),
			\end{equation}
			where $-x+\mathcal{FI}^{u_{*}(T_0)}:=\left\lbrace \{y_1-x,y_2-x\} :e=\{y_1,y_2\}\in \mathcal{FI}^{u_{*}(T_0)+\epsilon,T_0} \right\rbrace $.
			
			\item Under the coupling between $\mathcal{FI}^{u_{*}(T_0)+\epsilon,T_0}$ and $\mathcal{FI}^{u_{*}(T_0)+\epsilon,T}$, define 
			\begin{equation}
				G_2(x,T):=\left\lbrace \mathcal{B}_x(2K_0)\cap \mathcal{FI}^{u_{*}(T_0)+\epsilon,T_0}=\mathcal{B}_x(2K_0)\cap \mathcal{FI}^{u_{*}(T_0)+\epsilon,T} \right\rbrace
			\end{equation}
			
		\end{enumerate}
		Let $F_{x}:=G_1(x)\cap G_2(x,T)$. By Theorem \ref{equality} and (\ref{7.8}), we have 
		\begin{equation}\label{limdT0+}
			\lim\limits_{\Delta T\to 0+}\sup_{T\in [T_0-\Delta T,T_0]} P\left[F_x \right]=1.  
		\end{equation}
		
		For $x,y\in \mathbb{K}_0$, we say $x$ and $y$ are $*$-neighbors if $|x-y|=K_0$ and call a path $(x_0,x_1...x_m)$ in $\mathbb{L}_0$ as $*$-path if for all $0\le i\le m-1$, $x_i$ and $x_{i+1}$ are $*$-neighbors. For integers $N\ge M\ge 1$, let $H^*(x,M,N)$ be the event that $B_x(M)$ and $\partial B_x(N)$ are connected by a $*$-path with centers of bad boxes. Let $\widetilde{A}_{n,x}=H^*(x,K_n,2K_n)$. Similar to (\ref{7.10}), we have: for any $n\ge 1$,   \begin{equation}\label{PuTAn0}
			P(\widetilde{A}_{n,0})\le \left[Ck_0^{2(d-1)}\left( P(\widetilde{A}_{0,0})+2e^{-cK_0}\right)  \right]^{2^n}.
		\end{equation}
		Combining (\ref{limdT0+}) and the fact that $P^{u,T}(\widetilde{A}_{0,0})\le P^{u,T}(\exists x\in K_0\times B(1), B_x(K_0)\ are\ bad)\le 3^d(1-P[F_x])$, there exists $K_0\in \mathbb{N}^+,\Delta T_4>0$ such that $Ck_0^{2(d-1)}\left( P^{u_*(T_0)+\epsilon,T_0}(\widetilde{A}_{0,0})+2e^{-cK_0}\right)<\frac{1}{2}$. By (\ref{limdT0+}) and (\ref{PuTAn0}), we have: for any $T\in [T_0-\Delta T_4,T_0+\Delta T_4]$, 
		\begin{equation}\label{PuTAn02}
			P(\widetilde{A}_{n,0})\le 2^{-2^n}. 
		\end{equation}

		For any $T\in [T_0-\Delta T_4,T_0]$, $F_{x}\subset  \text{Exist}(K_0,-x+\mathcal{FI}^{u_{*}(T)+\epsilon,T})\cap \text{Unique}(K_0,-x+\mathcal{FI}^{u_{*}(T_0)+\epsilon,T}) $. Hence, if there exists an infinite nearest-neighbor path in $(K_0\cdot\mathbb{Z}^2)\times \{0\}^{d-2}$ containing $0\in \mathbb{Z}^d$ with centers of good boxes, then there also exists an infinite cluster in $\mathcal{FI}^{u_{*}(T)+\epsilon,T}$ within the slab $(\mathbb{Z}^2)\times [-2K_0,2K_0]^{d-2}$. 
		
		By dual graph arguement on $\mathbb{Z}^2$, if the infinite cluster in $(K_0\mathbb{Z}^d)\times \{0\}^{d-2}$ with centers of good boxes does not exist, then there must exist a $*$-connected circuit in $(K_0\mathbb{Z}^d)\times \{0\}^{d-2}$ with centers of bad boxes surrounding all boxes $[-K_m,K_m]^2\times \{0\}^{d-2}$, $m\in \mathbb{N}^+$. Therefore, by (\ref{PuTAn02}), for any $T\in [T_0-\Delta T_4,T_0+\Delta T_4]$ and integer $m_0\ge 1$, we have  
		\begin{equation}\label{pthere}
			\begin{split}
				&P\left[\text{there exists an infinite cluster of } \mathcal{FI}^{u_{*}(T)+\epsilon,T} \text{within the slab}\  (\mathbb{Z}^2)\times [-2K_0,2K_0]^{d-2}\right]\\
				\ge&1- P\bigg[\text{for all}\ m\in \mathbb{N}^+,\ \text{there exists a *-connected circuit with centeres of bad  boxes}\\ 
				&\ \ \ \ \ \ \ \ \ \text{surrounding the slab}\ [-K_m,K_m]^2\times \{0\}^{d-2}\bigg] \\
				\ge &1- P\bigg[\text{there exists a *-connected circuit with centeres of bad  boxes surrounding}\\ 
				&\ \ \ \ \ \ \ \ \ \text{the slab}\ [-K_{m_0},K_{m_0}]^2\times \{0\}^{d-2}\bigg] \\
				\ge &1- \sum_{n\ge m_0}\sum_{x\in (K_0\mathbb{Z}^d)\times \{0\}^{d-1}:K_n\le |x|<K_{n+1}}P\left[\widetilde{A}_{n,0}\right]\\
				\ge &1-\sum_{n\ge m_0}k_0^{n+1}2^{-2^n}.
			\end{split}
		\end{equation}
	Take a sufficiently large integer $m_0$ in (\ref{pthere}) such that $\sum_{n\ge m_0}k_0^{n+1}2^{-2^n}<1$. Then we have that for any $T\in [T_0-\Delta T_4,T_0]$, $\mathcal{FI}^{u_{*}(T)+\epsilon,T}$ percolates.
		
%		\begin{equation}
%			\begin{split}
%				&	P\left[\text{there does not exist an infinite cluster in } \mathcal{FI}^{u_{*}(T)+\epsilon,T} \right] \\
%				\le &\inf_{m\in \mathbb{N}^+}P[\text{there exists a *-connected circuit with centeres of bad  boxes}\\ 
%				&\ \ \ \ \ \ \ surrounding\ [-L_m,L_m]^2\times \{0\}^{d-2} ] \\
%				\le &\inf_{m\in \mathbb{N}^+}\sum_{n\ge m}\sum_{x\in (L_0\mathbb{Z}^d)\times \{0\}^{d-1}:L_n\le |x|<L_{n+1}}P\left[\widetilde{A}_{n,0}\right]\\
%				\le & \inf_{m\in \mathbb{N}^+}\sum_{n\ge m}l_0^{n+1}2^{-2^n}=0. 
%			\end{split}
%		\end{equation}
		
		In conclusion, we take $\Delta T=\min\{\Delta T_1,\Delta T_2,\Delta T_3,\Delta T_4\}$, then for any $T\in [T_0-\Delta T,T_0+\Delta T]$, we have $|u_{*}(T)-u_{*}(T_0)|\le \epsilon$.
	\end{proof}
	
	\begin{remark}
		In fact, the apporach in this section can also be adapted to prove the continuity of critical value of massive Gaussian free field level set.

		Briefly, massive Gaussian free field $\{\varphi_x^\theta\}_{x\in \mathbb{Z}^d}$, $\theta\in (0,1)$ is a Gaussian random field satisfying that $E\left(\varphi_x^\theta\varphi_y^\theta\right)=g_\theta(x,y)$, where $g_\theta(x,y)$ is the Green function produced by geometrically killed simple random walks on $\mathbb{Z}^d$ with killing rate $\theta$ at each step (see Section 1.5 of \cite{werner2020lecture} or \cite{rodriguez20170} for more details). For any $h\in \mathbb{R}$, the level set at level $h$ is $E_\theta^{\ge h}:=\{x\in \mathbb{Z}^d:\varphi_x^\theta\ge h\}$. Let $h_{*}^\theta$, $h_{**}^\theta$, $\bar{h}^\theta$ and $\widetilde{h}^\theta$ be critical values corresponding to $h_{*}$, $h_{**}$, $\bar{h}$ and $\widetilde{h}$ in \cite{duminil2020equality}. 
		
		As mentioned in the last paragraph of Section 1.2 in \cite{duminil2020equality}, the techniques developed in \cite{duminil2020equality} is ready to be applied to prove $h_{*}^\theta=h_{**}^\theta=\bar{h}^\theta=\widetilde{h}^\theta$ for all $\theta\in (0,1)$. Based on this equivalence, parallel to the proof of Theorem \ref{continuity}, one can use the remonalization scheme in \cite{rodriguez2013phase}, the block construction approach in \cite{rath2013effect} and a coupling arguement between massive Gaussian free fields with different killing rates $\theta$ to show that $h_{*}^\theta$ is continuous w.r.t. $\theta$ on $(0,1)$. 	
	\end{remark}

	\section{Proof outline of Theorem \ref{equality}}
	
	In this paper, we follow the strategy in \cite{duminil2020equality} to prove Theorem \ref{equality}. For coherence, we still give a general scheme for the proof in this section  . 
	
	Theorem \ref{equality} is divided into four sub-questions (\textbf{SQ}): for any $d\ge 3$ and $T>0$, 
	\begin{enumerate}
		\item[\textbf{SQ1}:]  $u_*\ge u_{**}$; 
		
		\item[\textbf{SQ2}:]  $\bar{u}\ge  u_*$;
		
		\item[\textbf{SQ3}:]  $\widetilde{u} \ge \bar{u}$;
		
		\item[\textbf{SQ4}:]  $u_{**}\ge \widetilde{u}$.
		
	\end{enumerate}

	Among these four sub-questions, proofs of \textbf{SQ1} and \textbf{SQ2} are elementary. So we just put them at the end of this section for reader's confirmation.

	Motivated by \cite{duminil2020equality}, we introduce a restricted model $\mathcal{FI}_L^{u,T}$ for $L\in \mathbb{N}^+$ (recall the definition of $\mathcal{FI}_L^{u,T}$ in (\ref{2.1})). In fact, $\mathcal{FI}_L^{u,T}$ share a series of good properties (such as ``finite-range dependence'') and play an important role in proofs of \textbf{SQ3} and \textbf{SQ4}. Parallel to definitions of $u_{*}$, $u_{**}$, $\bar{u}$ and $\widetilde{u}$, one can also define critical values $u_{*}^L$, $u_{**}^L$, $\bar{u}^L$ and $\widetilde{u}^L$ for the restricted model $\mathcal{FI}_L^{u,T}$.
 
%Roughly speaking, $\mathcal{FI}_L^{u,T}$ is the point measure containing all trajectories in $\mathcal{FI}_L^{u,T}$ whose lengths are less than or equal to $L$. 	
	
	To prove \textbf{SQ3}, we first build up a ``bridging lemma'' for $\mathcal{FI}^{u,T}$, similar to Lemma 3.5 in \cite{duminil2020equality}. This ``bridging lemma'' shows that for any subsets $\mathcal{S}_1,\mathcal{S}_2\subset \mathbb{L}^d$, if they are both large and close to each other in a given box, then no matter how one fixes the configuration of $\mathcal{FI}^{u,T}$ ($u>u_{**}$) on $\mathcal{S}_1\cup\mathcal{S}_2$ and outside the given box, there is still a non-negligible probability for $\mathcal{FI}^{u,T}$ to connect $\mathcal{S}_1$ and $\mathcal{S}_2$. Using the ``bridging lemma'' and a series of combinatorial arguments, one can show that when $u>\widetilde{u}$, a weak version of event $\text{Exist}(R,\mathcal{FI}^{u,T})\cap \text{Unique}(R,\mathcal{FI}^{u,T})$ happens with high probability. Using another bridging lemma for finite-dependent models and restriction arguements based on $\mathcal{FI}_L^{u,T}$, we ulteriorly prove that there exists a sufficiently large and thick cluster of $\mathcal{FI}^{u,T}$ ($u>\widetilde{u}$) in an arbitrarily given box except in a stretched-exponentially small probability (w.r.t. the size of box). With the help of this key cluster, the desired ``strongly percolating'' property is ultimately proved to hold for all $u>\widetilde{u}$. Details of this part can be found in Section \ref{SQ3}. Indeed, the proof of \textbf{SQ3} is parallel to Proposition 1.5 in \cite{duminil2020equality}. Although level-sets of Gaussian free field (GFF) and FRI share some similarities in structure, they are two different types of objects (GFF is a random field but FRI is a Possion point process). So we can not directly cite the main results in \cite{duminil2020equality}. For completeness of this article, we still need to write the proof down in detail.

	\textbf{SQ4} is proved by contradiction. The first step is to show that for all $L\in \mathbb{N}^+$, $u_{*}^L=u_{**}^L=\widetilde{u}^L$. We adapt the approaches in \cite{duminil2019sharp} and \cite{grimmett1990supercritical} to accomplish this step. In the second step, we aim to find a uniform upper bound for the increment term $P[B(R) \xleftrightarrow[]{\mathcal{FI}^{u,T}} \partial B(2R) ]-P[B(R) \xleftrightarrow[]{\mathcal{FI}_L^{u+\epsilon,T}} \partial B(2R) ]$. To get the bound, Russo's formulas for FRI and ``bridging lemma'' are used to prove a key inequality between two partial derivatives of $\theta(t,u,R)$, the probability that $B(R)$ and $\partial B(2R)$ are connected by $\chi_t$ (where $\chi_t$ is a continuous version of the restricted model $\mathcal{FI}_L^{u,T}$). Assuming there exists $u_0,\epsilon>0$ such that $u_0\in (u_{**}+2\epsilon,\widetilde{u}-2\epsilon)$, it is immediate that $u_{**}<u<u+\epsilon<\widetilde{u}^L=u_{**}^L$ for all $L\in \mathbb{N}^+$, and finally causes contradiction with the bound mentioned above. See Section \ref{SQ4} for technical details. 
	
%	the crossing probability w.r.t. annulus $B(2R)\setminus B(R)$ (see Section \ref{estimate_truncation} for precise definition).

	We hereby conclude this section with the proofs of \textbf{SQ1} and \textbf{SQ2}:

	\begin{proof}[Proof of \textbf{SQ1}]
		Recall the definition of $u_{**}$ in Section \ref{notations}. For any $u<u_{**}$,
		\begin{equation}\label{3.1}
			\inf\limits_{R\in \mathbb{N}}P\left[B(R)\xleftrightarrow[]{\mathcal{FI}^{u,T}}\partial B(2R) \right]=0.
		\end{equation}
		
		Note that for any $R>0$, $P\left[0\xleftrightarrow[]{\mathcal{FI}^{u,T}}\partial B(2R) \right]\le P\left[B(R)\xleftrightarrow[]{\mathcal{FI}^{u,T}}\partial B(2R) \right]$. By (\ref{3.1}), 
		\begin{equation}
			P\left[0\xleftrightarrow[]{\mathcal{FI}^{u,T}}\infty \right]=\inf\limits_{R\in \mathbb{N}}P\left[0\xleftrightarrow[]{\mathcal{FI}^{u,T}}\partial B(2R) \right]\le \inf\limits_{R\in \mathbb{N}}P\left[B(R)\xleftrightarrow[]{\mathcal{FI}^{u,T}}\partial B(2R) \right]=0,
		\end{equation}
		which implies that $u\le u_*$. In conclusion, we have $u_{**}\le u_*$.
	\end{proof}

	\begin{proof}[Proof of \textbf{SQ2}]
		
		For any $u>\bar{u}$, recalling (\ref{exi}) and (\ref{2.2}), 
		\begin{equation}
			\begin{split}
				\sum_{R\in \mathbb{N}^+}P\left[\text{Exist}(R,\mathcal{FI}^{u,T})^c\cup \text{Unique}(R,\mathcal{FI}^{u,T})^c\right]
				\le  \sum_{R\in \mathbb{N}^+} 2e^{-cR^\rho}<\infty. 
			\end{split}
		\end{equation}
		
		By Borel-Cantelli Lemma, with probability one, the following events may not happen i.o.: $\left\lbrace \text{Exist}(R,\mathcal{FI}^{u,T})^c\cup \text{Unique}(R,\mathcal{FI}^{u,T})^c \right\rbrace_{R\in \mathbb{N}^+} $. Therefore, with probability $1$, there exists $R_0(\omega)$ such that for all $R\ge R_0$, the event $\text{Exist}(R,\mathcal{FI}^{u,T})\cap \text{Unique}(R,\mathcal{FI}^{u,T})$ happens.

		For each $R\ge R_0$, denote by $\mathfrak{C}_{R}$ the collection of all connected clusters contained in $\mathcal{FI}^{u,T}\cap \mathcal{B}(R)$ with diameter at least $R/5$. For any $\mathcal{C}_1\in \mathfrak{C}_{R}$ and $\mathcal{C}_2\in \mathfrak{C}_{R+1}$, since $\text{diam}(\mathcal{C}_1)\ge \frac{R}{5}\ge \frac{R+1}{10}$, $\text{diam}(\mathcal{C}_2)\ge \frac{R+1}{5}\ge \frac{R+1}{10}$ and the event $\text{Unique}(R+1,\mathcal{FI}^{u,T})$ occurs, we have that $\mathcal{C}_1$ and $\mathcal{C}_2$ are connected by $\mathcal{FI}^{u,T}\cap \mathcal{B}(2(R+1))$. Therefore, $\cup_{R\ge R_0}\mathfrak{C}_{R}$ is connected, which implies that a.s. there exists an infinite cluster in $\mathcal{FI}^{u,T}$. In conclusion, $\bar{u}\ge u_*$. 
	\end{proof}

	\section{Proof of SQ3}\label{SQ3}
	
%	\note{move to general discussion} The proof of \textbf{SQ3} is parallel to Proposition 1.5 in \cite{duminil2020equality}. Indeed, although level-sets of Gaussian free field (GFF) and FRI share some similarities in structure, they are two different types of objects (GFF is a random field but FRI is a Possion point process). So we can not directly cite the main results in \cite{duminil2020equality}. For completeness of this article, we still need to write the proof down in detail. The proof is divided into four parts and presented in the following subsections.
	
	Unless stated otherwise, we fix $d\ge 3$ and $T\in (0,\infty)$ in the rest of this section.
	
	\subsection{Bridging Lemmas}\label{bridginglemmas}
	
	In the first part, we cite some notations and concepts about ``bridge'' from Section 2 of \cite{duminil2020equality} along with two useful ``bridging lemmas''.
	
	Here are some basic notations we need: 
	\begin{itemize}
		\item For integers $L_0\ge 100,l_0\ge 1000$ and $n\ge 0$, let $L_n:=L_0l_0^n$ and $\mathbb{L}_n:=(2L_n+1)\mathbb{Z}^d$;  
		
		\item For integers $\kappa\ge 100$ and $n\ge 0$, let $\Lambda_n:=\mathcal{B}(10\kappa L_n)$, $\underline{\Lambda}_n:=\mathcal{B}(8\kappa L_n)$ and $\Sigma_n:=\mathcal{B}(9\kappa L_n)\setminus \mathcal{B}(\lfloor 8.5\kappa L_n\rfloor)$;
		
		\item For integer $m\ge 0$ and vertex $x\in \mathbb{Z}^d$, we call the box $\mathcal{B}_x(L_m)$ an $m$-edge-box centered at $x$ and $B_x(L_m)$ an $m$-vertex-box centered at $x$. 	
	\end{itemize}
	
	For $\Lambda_n$, say that the edge set $\mathcal{S}\subset \Lambda_n$ is $0$-admissible if $\mathcal{S}$ is a conneted subset of $\underline{\Lambda}_n$ such that $\text{diam}(\mathcal{S})\ge \kappa L_n$. For any susbets $\mathcal{S}_1,\mathcal{S}_2\subset \Lambda_n$, define that $\mathfrak{B}^0$, a finite collection of some $0$-vertex-boxes is a $0$-bridge between $\mathcal{S}_1$ and $\mathcal{S}_2$ in $\Lambda_n$ if $\cup_{B\in \mathfrak{B}^0}B$ is a connceted subset of $V(\Lambda_n)$ intersecting both $V(\mathcal{S}_1)$ and $V(\mathcal{S}_2)$.

	\begin{lemma}(\cite[Corollary 2.6]{duminil2020equality} and \cite[Theorem 4.1]{drewitz2014chemical})\label{0-bridgelemma}
		Suppose there exists some family of events $\left\lbrace F_{0,x}:x\in \mathbb{L}_0 \right\rbrace $ satisfying:
		\begin{enumerate}
			\item For any finite subset  $U\subset \mathbb{Z}^d$ such that $|y-z|\ge \frac{\kappa}{2}$ for all $y,z\in U$, the events $\{F_{0,x}:x\in (2L_0+1)U\}$ are independent;
			
			\item For all $x\in \mathbb{L}_0$, $P[F_{0,x}^c]\le c_1$, where $c_1\in (0,1)$ is a constant.  
			
		\end{enumerate}
		Say a $0$-bridge $\mathfrak{B}^0$ is good if for any $B_x(L_0)\in \mathfrak{B}^0$, the event $F_{0,x}$ happens. Then for each $\kappa \ge 100,l_0\ge C_1(\kappa)$, there exists $c_1(\kappa,l_0)\in (0,1)$ such that for all $L_0\ge C_2(\kappa,l_0)$ and $n\ge 0$,  
		\begin{equation}\label{Gn0}
			P\left[\mathscr{G}_n^0\right]:=P\left[\bigcap\limits_{0\text{-admissible}\ \mathcal{S}_1,\mathcal{S}_2\subset \Lambda_n}\left\lbrace \exists \text{a good}\ 0\text{-bridge between}\ \mathcal{S}_1\ \text{and}\ \mathcal{S}_2\ \text{in}\ \Lambda_n \right\rbrace  \right] \ge 1-2^{-2^n}.  
		\end{equation}

	\end{lemma}

	\begin{definition}\label{defbridge}(Definition 2.1, \cite{duminil2020equality})
		For any $\mathcal{S}_1,\mathcal{S}_2\subset \Lambda_n$, say a finite collection of vertex-boxes of various levels, $\mathfrak{B}$ is a bridge between $\mathcal{S}_1$ and $\mathcal{S}_2$ inside $\Lambda_n$ if $\mathfrak{B}$ satisfies:
		\begin{enumerate}
			\item $\bigcup_{B\in \mathfrak{B}}B$ is connected and for each box $B\in \mathfrak{B}$, $B\subset V(\Sigma_n)$;
			
			\item There exists two $0$-vertex-boxes $B_1,B_2\in \mathfrak{B}$ such that $B_1\cap V(\mathcal{S}_1)\neq \emptyset$, $B_2\cap V(\mathcal{S}_2)\neq \emptyset$ and for all $B\in \mathfrak{B}\setminus \{B_1,B_2\}$, $B\cap V(\mathcal{S}_1\cup \mathcal{S}_2)=\emptyset$; 
			
			\item For each m-vertex-box $B\in \mathfrak{B}$ with $1\le m\le n$, $d(B,V(\mathcal{S}_1\cup \mathcal{S}_2))\ge \kappa L_m$; 
			
			\item For any $0\le m\le n$, the number of m-vertex-boxes in $\mathfrak{B}$ is smaller than $2K$, where $K\ge 100$ is a constant.
		\end{enumerate}
	\end{definition}

	For an edge set $\mathcal{S}\subset \Lambda_n$, say that $\mathcal{S}$ is admissible if each connected component in $\mathcal{S}$ intersects $\partial V(\Lambda_n)$ and there exists one of them intersecting $\partial V(\underline{\Lambda}_n)$. 
	
	We need another bridging lemma, which can be seen as the counterpart of Lemma 3.5, \cite{duminil2020equality}. 
	
	\begin{lemma}\label{bridgelemma}
		For any $\epsilon>0$ and $L_0\ge C_3(\epsilon)$, there exists constant $C_4(\epsilon,L_0)>0$ such that for all $u\ge u_{**}+\epsilon$, integer $n\ge 1$, admissible edge sets $\mathcal{S}_1,\mathcal{S}_2$ in $\Lambda_n$, and event $D$ measurable w.r.t. $\mathcal{F}:=\sigma(\sum_{(a_i,\eta_i)\in \mathcal{FI}^{T}}\delta_{(a_i,\eta_i)}\cdot\mathbbm{1}_{0<a_i\le u,\eta_i\cap (\mathcal{S}_1\cup\mathcal{S}_2\cup \Lambda_n^c)\neq \emptyset })$, we have
			\begin{equation}\label{bridgeequality}
				P\left[\mathcal{S}_1\xleftrightarrow[]{\mathcal{FI}^{u,T}\cap \Lambda_n}\mathcal{S}_2 \bigg| D\right]\ge e^{-C_4(\log(L_n))^2}. 
		\end{equation}
	\end{lemma}

	Before proving Lemma \ref{bridgelemma}, we need the following estimate parallel to Lemma 3.3 in \cite{duminil2020equality}. 
	
	\begin{lemma}\label{lemma3}
		For any $\epsilon>0$, there exist constants $c_2(\epsilon),C_5>0$ such that for all $u\ge u_{**}+\epsilon $, integer $L\ge 1$ and $x,y\in B(L)$, we have 
			\begin{equation}\label{5.3}
				P\left[x \xleftrightarrow[B(2L)]{\mathcal{FI}_L^{u,T}} y \right] \ge c_2L^{-C_5}.
		\end{equation}
	\end{lemma} 
	
	\begin{proof}[Proof of Lemma \ref{lemma3}] %For any $u>u_{**}$, let $\delta=\frac{u-u_{**}}{2}$. 
		Let $u_0=u_{**}+\epsilon$. There exists $c(\epsilon)>0$ such that for all $ n\ge 1$, 
		\begin{equation}
			P\left[B(n) \xleftrightarrow[]{\mathcal{FI}^{u_0,T}} \partial B(2n) \right]\ge c.  
		\end{equation}
Since $c\le P\left[B(n) \xleftrightarrow[]{\mathcal{FI}^{u_0,T}} \partial B(2n) \right]\le \sum_{y\in \partial B(n)}P\left[y \xleftrightarrow[]{\mathcal{FI}^{u_0},T} \partial B_y(n)\right]$, we have 
		\begin{equation}
			P\left[0 \xleftrightarrow[]{\mathcal{FI}^{u_0},T} \partial B(n)\right]\ge c'(\epsilon)n^{-(d-1)}. 
		\end{equation}
		Hence, there exists $y^*\in \partial B(n)$ such that \begin{equation}
			P\left[0 \xleftrightarrow[B(n)]{\mathcal{FI}^{u_0,T}} y^*\right]\ge c''(\epsilon)n^{-2(d-1)}.
		\end{equation}
For $1\le i\le d$, denote by $x_i$ the vertex $x\in \mathbb{Z}^d$ such that $x^{(i)}=1$ and for any $j\neq i$, $x^{(j)}=0$. By FKG inequality (Proposition \ref{FKG} in Appendix \ref{appendixFKG}), we have  
		\begin{equation}\label{p0b2n}
			P[0 \xleftrightarrow[B(2n)]{\mathcal{FI}^{u_0,T}} 2nx_i]\ge \left( P[0 \xleftrightarrow[B(n)]{\mathcal{FI}^{u_0,T}} y^*]\right) ^2\ge (c'')^2n^{-4(d-1)}. 
		\end{equation}
		
		For $x,y\in B(L)$, let $z_0=x$, $z_1=(y^{(1)},x^{(2)},...,x^{(d)})$, $z_2=(y^{(1)},y^{(2)},x^{(3)},...,x^{(d)})$,..., and $z_d=y$. If $x_i-y_i$ is even for all $1\le i\le d$, applying (\ref{p0b2n}) and FKG inequality,
		\begin{equation}\label{PxyB2L}
			\begin{split}
				P\left[x\xleftrightarrow[B(2L)]{\mathcal{FI}^{u_0,T}}y\right]\ge& \prod_{i=0}^{d-1}P\left[z_i \xleftrightarrow[B(2L)]{\mathcal{FI}^{u_0,T}} z_{i+1}\right]
				\ge (c'')^{2d}\prod_{i=1}^{d} \left| \frac{x^{(i)}-y^{(i)}}{2} \right| ^{-4(d-1)}\ge (c'')^{2d}L^{-4d(d-1)}. 
			\end{split}
		\end{equation}
If there exists $1\le i\le d$ such that $x_i-y_i$ is odd, we arbitrarily select a vertex $y'\in B_y(1)\cap B(L)$ such that for all $1\le i\le d$, $x_i-y_i'$ is even. By FKG inequality and (\ref{PxyB2L}), \begin{equation}\label{Pxtoy}
			P[x\xleftrightarrow[B(2L)]{\mathcal{FI}^{u_0,T}}y]\ge P[x\xleftrightarrow[B(2L)]{\mathcal{FI}^{u_0,T}}y']\cdot P[y'\xleftrightarrow[B(2L)]{\mathcal{FI}^{u_0,T}}y]\ge c'''(\epsilon)L^{-4d(d-1)}. 
		\end{equation}
		
		For each $e=\{w,w'\}\in \mathbb{L}^d$, we have 
		\begin{equation}\label{PeFI}
			\begin{split}
				&P[e\in \mathcal{FI}^{u_0,T},e\notin \mathcal{FI}_L^{u_0,T}]\\
				\le&\sum_{x\in \mathbb{Z}^d\setminus B_w(L)}P[there\ exists\ path\ \eta\in \mathcal{FI}^{u_0,T}\ such\ that\ \eta(0)=x \ and\ e\in \eta ]\\
				&+\sum_{x\in B_w(L)}P[there\ exists\ path\ \eta\in \mathcal{FI}^{u_0,T}\ such\ that\ \eta(0)=x \ and\ |\eta|\ge L+1 ]\\
				\le &\sum_{x\in \mathbb{Z}^d\setminus B_w(L)} \left(1-\exp(-\frac{2du_0}{T+1}\left(\frac{T}{T+1} \right)^{|x-w|} ) \right)\\
				&+\sum_{x\in B_w(L)}\left(1-\exp(-\frac{2du_0}{T+1}\left(\frac{T}{T+1} \right)^{L+1} ) \right)\\
				\le &\sum_{m=L}^{\infty} C(m+1)^{d-1}e^{-\hat{c}m}+\sum_{m=0}^{L-1} C(m+1)^{d-1}e^{-\hat{c}(L+1)}\le C'(\epsilon)e^{-\widetilde{c}(\epsilon)L}
			\end{split}
		\end{equation}

		By (\ref{PxyB2L}), (\ref{Pxtoy}) and (\ref{PeFI}), we have: for any $u\ge u_{**}+\epsilon$ and $L\ge C''(\epsilon)$, 
		\begin{equation}\label{411}
			\begin{split}
				P[x\xleftrightarrow[B(2L)]{\mathcal{FI}_L^{u,T}}y] \ge& P[x\xleftrightarrow[B(2L)]{\mathcal{FI}_L^{u_0,T}}y]\\
				\ge &P[x\xleftrightarrow[B(2L)]{\mathcal{FI}^{u_0,T}}y]-\sum_{e\in \mathcal{B}(2L)}P[e\in \mathcal{FI}^{u-\epsilon,T},e\notin \mathcal{FI}_L^{u-\epsilon,T}]\\
				\ge &c'''(\epsilon)L^{-4d(d-1)}-\sum_{e\in \mathcal{B}(2L)}C'(\epsilon)e^{-\widetilde{c}(\epsilon)L}\ge 0.5c'''(\epsilon)L^{-4d(d-1)}.
			\end{split}
		\end{equation}
Thus we only need to choose a sufficiently small $c_2\in (0,0.5c'''(\epsilon))$ such that  $P[x\xleftrightarrow[B(2L)]{\mathcal{FI}_L^{u_0,T}}y] \ge c_2L^{-4d(d-1)}$ holds for all $L< C''(\epsilon)$, and then (\ref{5.3}) follows.
	\end{proof}
	
	Now we are able to show the proof of Lemma \ref{bridgelemma}. See Figure \ref{figure1} for an illustration of the proof of Lemma \ref{bridgelemma}.

	\begin{proof}[Proof of Lemma \ref{bridgelemma}]
		
		By Lemma 2.4 of \cite{duminil2020equality}, the following fact holds: for any $\kappa>0$, $l_0\ge C(\kappa)$ and $K\ge C'(\kappa,l_0)$, then for all $L_0\ge 100$, $n\ge 0$ and admissible $\mathcal{S}_1,\mathcal{S}_2\subset \Lambda_n$, there exists a bridge $\mathfrak{B}$ between them inside $\Lambda_n$.

			\begin{figure}[h]
			\centering
			\includegraphics[width=0.5\textwidth]{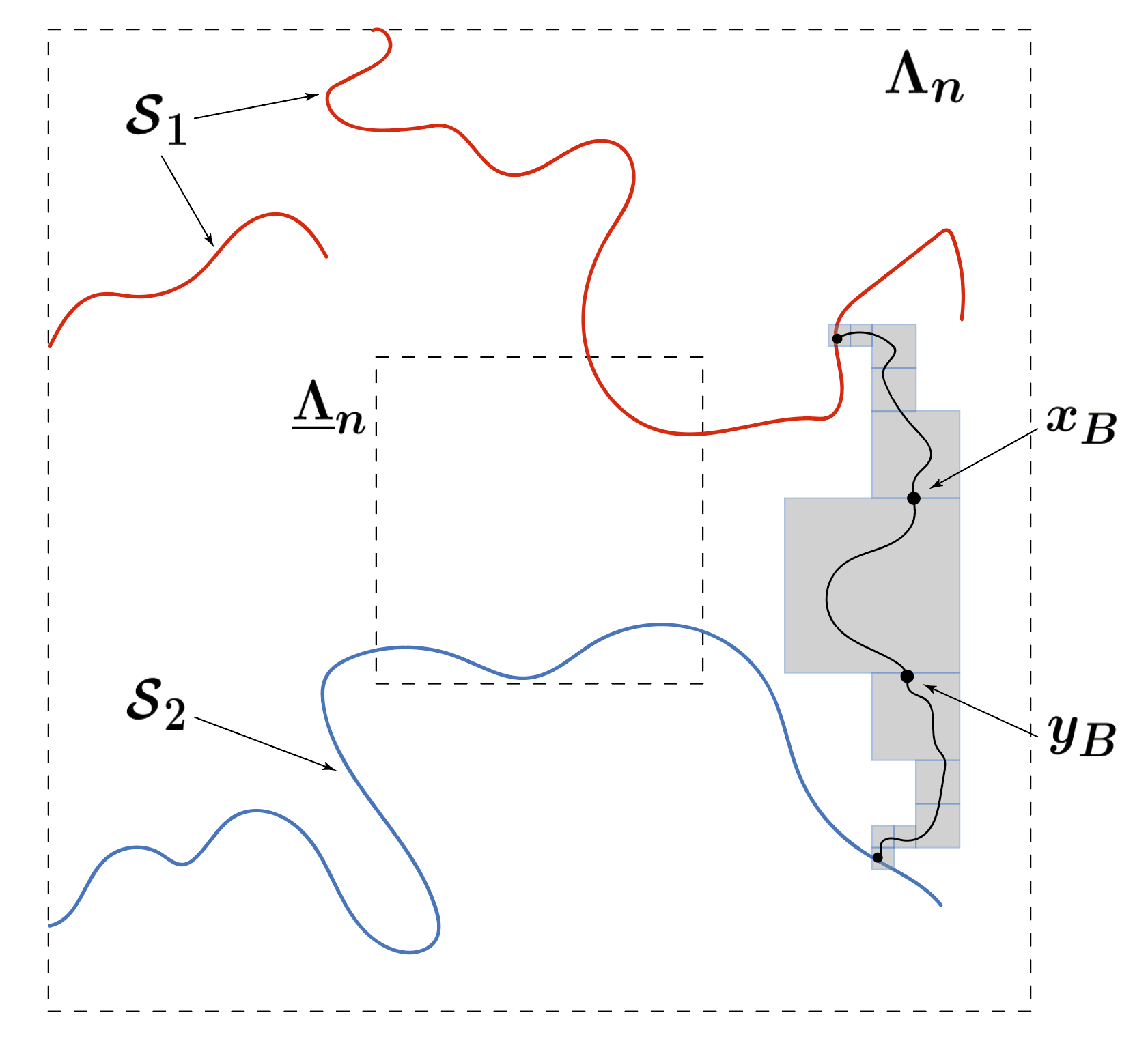}
			\caption{An illustration of the proof of Lemma \ref{bridgelemma}.}
			\label{figure1}
		\end{figure}

		Consider an arbitrary bridge $\mathfrak{B}$ between $\mathcal{S}_1$ and $\mathcal{S}_2$ in $\Lambda_n$. Take a sequence of boxes in $\mathfrak{B}$, $(B_0,B_1,...,B_m)$ such that: $B_0$ and $B_m$ are the unique boxes intersecting $V(\mathcal{S}_1)$ and $V(\mathcal{S}_2)$ respectively; for $0\le i\le m-1$, $B_i\cap B_{i+1}\neq \emptyset$. Then for any $1\le i\le m$, we arbitrarily choose a sequence of vertices $z_i\in B_{i-1}\cap B_{i}$, and $z_0\in V(\mathcal{S}_1)\cap B_0$, $z_{m+1}\in V(\mathcal{S}_2)\cap B_m$. For any $0\le i\le m$, let $x_{B_i}:=z_i$ and $y_{B_i}:=z_{i+1}$. For each $B_i=B_{y_i}(L_{n_i})$, we also denote that $\mathcal{B}_i=\mathcal{B}_{y_i}(L_{n_i})$.

		For $0$-vertex-box $B=B_x(L_0)\in \mathfrak{B}$, let $\mathcal{B}:=\mathcal{B}_x(L_0)$ and define the event
	\begin{equation}
				A_B:=\left\lbrace x_B \xleftrightarrow[]{\mathcal{FI}_{1}^{u,T}\cap \left(\mathcal{B}\setminus \partial_e \mathcal{B} \right) } y_B \right\rbrace.
		\end{equation}
For $m\ge 1$ and $m$-vertex-box $B=B_x(L_m)\in \mathfrak{B}$, we define
		\begin{equation}
			A_B:=\left\lbrace x_B \xleftrightarrow[B_x(2L_m)]{\mathcal{FI}_{L_m}^{u,T}} y_B \right\rbrace. 
		\end{equation}
		
		Since $A_B$ is independent to $\mathcal{F}$ and $A_{B_0}\cap A_{B_m}$ for all $B\in \mathfrak{B}\setminus \{B_0,B_m\}$, we have
			\begin{equation}\label{4.18}
				\begin{split}
					P\left[\{\mathcal{S}_1\xleftrightarrow[]{\mathcal{FI}^{u,T}\cap \Lambda_n} \mathcal{S}_2\}\cap D   \right] \ge &P\left[D \cap \bigcap_{i=0}^m A_{B_i} \right] \\
					=& E\left[\mathbbm{1}_{D}\cdot P\left[\bigcap_{i=0}^m A_{B_i} \bigg|D \right]  \right]\\
					=&E\left[\mathbbm{1}_{D}\cdot P\left[A_{B_0}\cap A_{B_m}\big|D \right]  \right]P\left[\bigcap_{i=1}^{m-1} A_{B_i}  \right].
				\end{split}
		\end{equation}
For $i\in \{0,m\}$, we arbitrarily choose a nearest-neighbor path $\eta_i$ from $x_i$ to $x_{i+1}$ within $\mathcal{B}_i\setminus (\mathcal{S}_1 \cup \mathcal{S}_2\cup \partial_e \mathcal{B}_i)$. Noting that events $\left\lbrace e\in \mathcal{FI}_1^{u,T}:e\in \eta_0\cup \eta_m \right\rbrace$ are independent to each other and independent to $\mathcal{F}$, we have
		\begin{equation}\label{ab1ab2}
			\begin{split}
				&P\left[A_{B_0}\cap A_{B_m}\big|D \right]\\
				\ge &P\left[\bigcap_{e\in \eta_0\cup \eta_m}\{e\in \mathcal{FI}_1^{u,T}\}\bigg|D\right]\\
				\ge & c(\epsilon)^{|\mathcal{B}_0|+|\mathcal{B}_m|}:= [c'(\epsilon,L_0)]^2.
			\end{split}
		\end{equation}
	By FKG inequality and (\ref{ab1ab2}),
		\begin{equation}\label{4.19}
				\begin{split}
					&E\left[\mathbbm{1}_{D}\cdot P\left[A_{B_0}\cap A_{B_m}\big|D \right]  \right]\cdot P\left[\bigcap_{i=1}^{m-1} A_{B_i}  \right]\\
					\ge &(c')^2P[D]\cdot\prod_{i=1}^{m-1}P\left[A_B\right]. 
				\end{split}
		\end{equation}

		For any $0\le m\le n$, assume that there are exactly $k_m$ $m$-vertex-boxes in $\mathfrak{B}$. 
			
			By Lemma \ref{lemma3}, $\sum_{1\le m\le n}k_m\log(L_m)\le 2K\sum_{1\le m\le n}\log(L_m)\le C''(\log(L_n))^2$ and for any $0$-vertex box $B$, $P[A_{B}]\ge c'$, there exists a contant $C_3$ such that for any $L_0\ge C_3$,
		\begin{equation}\label{4.20}
				\begin{split}
					(c')^2\prod_{i=1}^{m-1}P\left[A_B\right]\ge& (c')^{k_0+2}\prod_{1\le m\le n}\left(c_2 (L_m)^{-C_5} \right)^{k_m}\\
					=&\exp((k_0+2)\log(c')+\sum_{1\le m\le n}k_m\left(\log(c_2)-C_5\log(L_m) \right))\\
					\ge &\exp(-2C_5C''(\log(L_n))^2).
				\end{split}
		\end{equation}
	For $C_4:=2C_5C''$, one may combine (\ref{4.18}), (\ref{4.19}) and (\ref{4.20}),
	\begin{equation}
			P\left[\{\mathcal{S}_1\xleftrightarrow[]{\mathcal{FI}^{u,T}\cap \Lambda_n} \mathcal{S}_2\}\cap D\right] \ge \exp(-C_4(\log(L_n))^2)P[D], 
		\end{equation}
		from which (\ref{bridgeequality}) follows.   	
	\end{proof}

	\subsection{From Connection to a Weak Version of Strong Percolation }
	
	For $\beta>\alpha>0$ and $N\in \mathbb{N}^+$, the event $\xi(N,\alpha,\beta)$ is defined as:
	\begin{equation}\label{eventxi}
		\begin{split}
			\xi(N,\alpha,\beta):=&\left\lbrace B(N)\xleftrightarrow[]{\mathcal{FI}^{\alpha,T}}\partial B(6N)  \right\rbrace \cap \bigg\{ \text{all clusters in}\ \mathcal{FI}^{\alpha,T}\cap \mathcal{B}(4N)\\
			& \ \text{crossing}\ B(4N)\setminus B(2N)\ \text{are connected in}\  \mathcal{FI}^{\beta,T}\cap \mathcal{B}(4N) \bigg\}, 
		\end{split}
	\end{equation}
	where ``crossing $B(4N)\setminus B(2N)$'' means intersecting both  $\partial B(4N)$ and $\partial B(2N+1)$.

	Parallel to Proposition 3 in \cite{duminil2020equality}, we have the following result:
	
	\begin{proposition}\label{prop2}
		For any $\epsilon>0$, \begin{equation}\label{equationprop2}
			\limsup\limits_{N\to \infty} \inf_{u\ge \widetilde{u}+2\epsilon}P\left[\xi(N,u,u+\epsilon) \right]=1. 
		\end{equation}
	\end{proposition}
	
	Before proving Proposition \ref{prop2}, we first introduce the following notations:
	\begin{enumerate}
		\item $\mathfrak{C}=\left\lbrace \mathcal{C}\subset \mathcal{B}(4N):\mathcal{C}\ is\ a\ cluster\ in\ \mathcal{FI}^{u,T}\cap \mathcal{B}(4N)\ intersecting\ \partial B(4N) \right\rbrace $;

		\item For $\mathcal{C},\mathcal{C}'\in \mathfrak{C}$ and edge set $\mathcal{W}$ such that $\mathcal{FI}^{u,T}\subset \mathcal{W}$, we write that $\mathcal{C}\sim_{\mathcal{W}} \mathcal{C}'$ if $\mathcal{C}\xleftrightarrow[]{\mathcal{W}} \mathcal{C}'$. Note that if $\mathcal{C}\sim_{\mathcal{W}} \mathcal{C}'$, then $\mathcal{C}$ and $\mathcal{C}'$ are contained in the same cluster of $\mathcal{W}$. Therefore, ``$\cdot \sim_{\mathcal{W}}\cdot$'' forms an equivalence relation.

		\item For any $\widetilde{\mathfrak{C}}\subset \mathfrak{C}$, $\widetilde{\mathfrak{C}}/ \sim_{\mathcal{W}}$ is a partition of $\widetilde{\mathfrak{C}}$ such that in each equivalence class, all clusters are connected by $\mathcal{W}$, but clusters in different classes are not; 
		
		\item For $0\le i\le 2\lceil \sqrt{N} \rceil$, let $\mathcal{V}_i:=\mathcal{B}(\lfloor 4N-i\sqrt{N}\rfloor)$;

		\item For $0\le i\le \lceil \sqrt{N} \rceil$ and $\mathcal{W}$ such that  $\mathcal{FI}^{u,T}\subset \mathcal{W}$, $\mathfrak{U}_i(\mathcal{W}):=\{\mathcal{C}\in \mathfrak{C}:\mathcal{C}\cap \mathcal{V}_{2i}\neq \emptyset\}/ \sim_{\mathcal{W}}$; the number of classes is written as $U_i(\mathcal{W}):=|\mathfrak{U}_i(\mathcal{W})|$. 
			Note that $U_i(\mathcal{W})$ is decreasing w.r.t. $\mathcal{W}$ and $i$ (i.e., if $\mathcal{W}\subset \mathcal{W}'$, $i\le i'$, then $U_i(\mathcal{W})\ge U_{i'}(\mathcal{W}')$).

		\item For any $\widetilde{\mathfrak{C}}\subset \mathfrak{C}$, we denote the support of this subset by $\cup\widetilde{\mathfrak{C}}:=\bigcup_{\mathcal{C}\in \widetilde{\mathfrak{C}}}\mathcal{C}$; 
		
		\item For $0\le i\le \lceil \sqrt{N} \rceil$, define that \begin{equation}
			\mathcal{W}_i=\left( \mathcal{V}_{2i}\cap\mathcal{FI}^{u,T}\right)  \cup \left( \mathcal{B}(4N)\setminus \mathcal{V}_{2i}\cap\mathcal{FI}^{u+\epsilon,T}\right). 
		\end{equation}
		Note that we have $\mathcal{W}_0\subset \mathcal{W}_1\subset ...\subset \mathcal{W}_{\lfloor \sqrt{N} \rfloor}$.

%		 $\mathcal{W}_i\subset \mathbb{L}^d\cap \mathcal{B}(4N)$ as follows: for each $e\in \mathcal{V}_{2i}$, let $e\in \mathcal{W}_i$ if $e\in \mathcal{FI}^{u,T}$; for each $e\in \mathcal{B}(4N)\setminus \mathcal{V}_{2i}$, let $e\in \mathcal{W}_i$ if $e\in \mathcal{FI}^{u+\epsilon,T}$.  
		
		\item For $0\le i\le \lceil \sqrt{N} \rceil$, let $\mathfrak{U}_i:=\mathfrak{U}_i(\mathcal{W}_i)$ and $U_i:=|\mathfrak{U}_i|$. Note that $U_i$ is decreasing w.r.t. $i$. An example of $\mathfrak{U}_i$ can be found in Figure \ref{figure2}.

		\begin{figure}[h]
			\centering
			\includegraphics[width=0.5\textwidth]{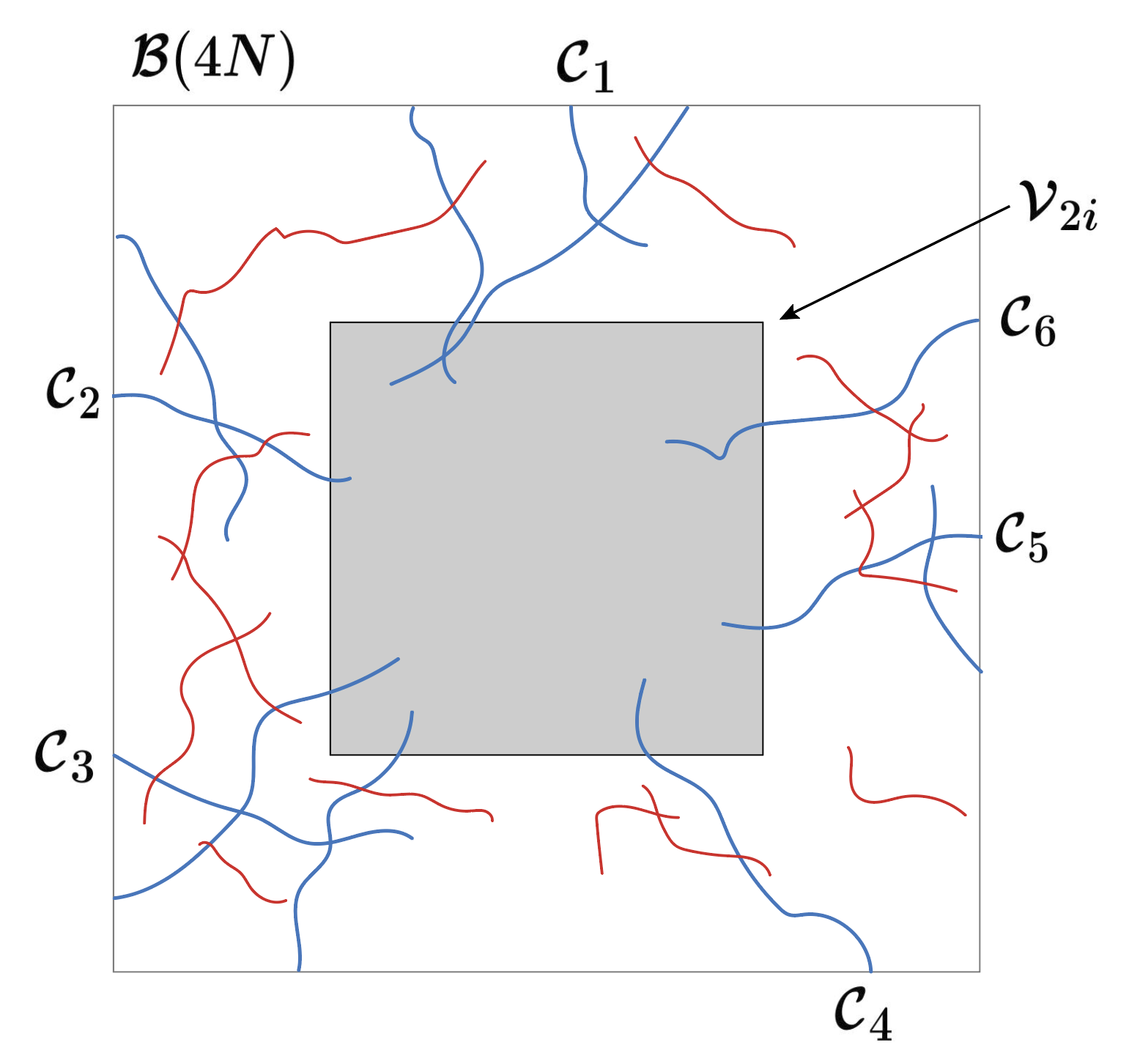}
			\caption{An example of $\mathfrak{U}_i$: in this example, there are six clusters in $\mathfrak{C}$, called $\mathcal{C}_1,\mathcal{C}_2,...,\mathcal{C}_6$ (colored in blue). The red lines represent the edge set $\mathcal{W}_{j}\setminus \mathcal{FI}^{u,T}$. As shown in the picture, clusters $\mathcal{C}_1$, $\mathcal{C}_2$ and $\mathcal{C}_3$ are connected by $\mathcal{W}_{j}$, so they are contained in the same class of $\mathfrak{U}_i$. In the same way, there are also two other classes in $\mathfrak{U}_i$: $\left\lbrace \mathcal{C}_4\right\rbrace $ and $ \left\lbrace \mathcal{C}_5,\mathcal{C}_6\right\rbrace $.       }
			\label{figure2}
		\end{figure}

		\item Recall the notation $\mu(\cdot)$ below (\ref{2.2}). Let $n_0=\min \{n:L_n\ge \mu(10N)\}$. Define the event \begin{equation}\label{3.3}
				A:=\bigcap_{x\in \mathbb{Z}^d:B_x(L_{n_0})\subset B(4N)}\left\lbrace B_x(L_{n_0}) \xleftrightarrow[]{\mathcal{FI}^{u,T}} \partial B(6N) \right\rbrace. 
		\end{equation}
	\end{enumerate}
	For the proof of Proposition \ref{prop2}, we need the following lemma similar to Lemma 4.2 in \cite{duminil2020equality}.

	\begin{lemma}\label{lemmaU}
		For any $\epsilon>0$, there exists constant $c_3(\epsilon)>0$ such that for all $u\ge \widetilde{u}+2\epsilon$, integers $N\ge 1$, $a\in [4,N^{1/4}]$ and $0\le i\le \lceil \sqrt{N} \rceil-a$, \begin{equation}\label{equationU}
			P\left[A\cap \{U_{i+a}>1\vee \frac{2U_i}{a}\} \right]\le \exp(-c_3N^{1/4}). 
		\end{equation}
	\end{lemma}
	
	Next we prove Proposition \ref{prop2} given Lemma \ref{lemmaU}. 
	
	\begin{proof}[Proof of Proposition \ref{prop2}]
		Fix $\epsilon>0$ and arbitrarily take $u\ge \widetilde{u}+2\epsilon $. Note that \begin{equation}\label{AcyB}
			A^c\subset \bigcup_{y\in B(4N)}\left\lbrace B_y(L_{n_0}) \overset{\mathcal{FI}^{u,T}}{\nleftrightarrow} \partial B(6N) \right\rbrace  \subset \bigcup_{y\in B(4N)}\left\lbrace B_y(\mu(10N)) \overset{\mathcal{FI}^{u,T}}{\nleftrightarrow} \partial B_y(10N) \right\rbrace .
		\end{equation}
		By (\ref{AcyB}), we have \begin{equation}\label{Ac}
			P\left[A^c\right]\le (8N+1)^d\cdot P\left[B(\mu(10N)) \overset{\mathcal{FI}^{u,T}}{\nleftrightarrow} \partial B(10N)  \right]. 
		\end{equation}
%	{\color{blue}Recall the definition of $\widetilde{u}$ in Section \ref{notations}. By (\ref{Ac}), for all $u\ge \widetilde{u}+2\epsilon$, 
%	\begin{equation}
%		\limsup\limits_{N\to \infty} P\left[A^c\right]\le \limsup\limits_{N\to \infty}(8N+1)^d\cdot P\left[B(\mu(10N)) \overset{\mathcal{FI}^{\widetilde{u}+2\epsilon,T}}{\nleftrightarrow} \partial B(10N)  \right]=0.
%	\end{equation}}

%		which implies \begin{equation}
%			P[A]\ge 1-(8N+1)^d\cdot P\left[B(\mu(10N)) \overset{\mathcal{FI}^{u,T}}{\nleftrightarrow} \partial B(10N)  \right]. 
%		\end{equation}
		
		Since $A\cap \{U_{\lfloor \sqrt{N} \rfloor}=1\}\subset \xi(N,u,u+\epsilon)$, we have
		\begin{equation}\label{xi}
			\begin{split}
				P\left[ \xi(N,u,u+\epsilon)^c \right]\le &P\left[A^c \right]+P\left[A\cap \{U_{\lfloor \sqrt{N} \rfloor}>1\} \right].
			\end{split}
		\end{equation}
		
	Note that there exists constant $C>0$ such that \begin{equation}
			|\mathfrak{C}|\le |\partial B(4N)|\le CN^{d-1}. 
		\end{equation}

	For an integer $a\in \left[4,N^{\frac{1}{4}} \right] $, arbitrarily select an integer $M$ satifying $aM\le \lceil \sqrt{N} \rceil$ and $\left(\frac{2}{a} \right)^MCN^{d-1}<1$. If the event $\bigcap_{0\le k<M}\left\lbrace U_{(k+1)a}\le 1\vee \frac{2U_{ka}}{a} \right\rbrace $ happens, then either
		\begin{equation}
			U_{\lceil \sqrt{N} \rceil}\le U_{Ma} \le \frac{2}{a}U_{(M-1)a}\le ... \le \left(\frac{2}{a} \right)^MU_0=\left(\frac{2}{a} \right)^M|\mathfrak{C}|\le \left(\frac{2}{a} \right)^MCN^{d-1}<1,
		\end{equation}
		or for some $0\le k<M$, $U_{(k+1)a}\le 1$ and thus $U_{\lceil \sqrt{N} \rceil}\le U_{(k+1)a}\le 1 $. 
		
		Therefore, we have 
		\begin{equation}
			\bigcap_{0\le k<M}\left\lbrace U_{(k+1)a}\le 1\vee \frac{2U_{ka}}{a} \right\rbrace  \subset \left\lbrace U_{\lceil \sqrt{N} \rceil}\le 1 \right\rbrace,
		\end{equation}
		which implies that 
		\begin{equation}\label{Acap}
			A\cap \left\lbrace U_{\lceil \sqrt{N} \rceil}> 1 \right\rbrace \subset \bigcup_{0\le k<M}   \left\lbrace A\cap \left\lbrace U_{(k+1)a}> 1\vee \frac{2U_{ka}}{a} \right\rbrace\right\rbrace . 
		\end{equation}
		
		By Lemma \ref{lemmaU} and (\ref{Acap}), 
		\begin{equation}\label{4.16}
			P\left[A\cap \left\lbrace U_{\lceil \sqrt{N} \rceil}> 1 \right\rbrace \right]\le M\cdot P\left[A\cap \left\lbrace U_{(k+1)a}> 1\vee \frac{2U_{ka}}{a} \right\rbrace \right]\le M\cdot \exp(-c_3N^{1/4}).  
		\end{equation}
		
	Recall the definition of $\widetilde{u}$ in Section \ref{notations}. Combining (\ref{Ac}),  (\ref{xi}), (\ref{4.16}) and the fact that $\liminf\limits_{N\to \infty} (a_N+b_N)\le \limsup\limits_{N\to \infty}a_N+\liminf\limits_{N\to \infty}b_N$, we conclude Proposition \ref{prop2}:
		\begin{equation}\label{limsupNto}
			\begin{split}
				&\limsup\limits_{N\to \infty} \inf_{u\ge \widetilde{u}+2\epsilon}P\left[ \xi(N,u,u+\epsilon) \right]\\
				\ge &\limsup\limits_{N\to \infty} \left\lbrace 1-\sup_{u\ge \widetilde{u}+2\epsilon}P\left[A\cap \left\lbrace U_{\lceil \sqrt{N} \rceil}> 1 \right\rbrace \right]-\sup_{u\ge \widetilde{u}+2\epsilon}P\left[A^c\right]      \right\rbrace\\
				\ge &1-\limsup_{N\to \infty} \left\lbrace M\cdot \exp(-c_3N^{1/4})\right\rbrace - \liminf_{N\to \infty}\left\lbrace (8N+1)^d\cdot P\left[B(\mu(10N)) \overset{\mathcal{FI}^{\widetilde{u}+2\epsilon,T}}{\nleftrightarrow} \partial B(10N)  \right]\right\rbrace \\
				=&1.
			\end{split}
		\end{equation} 
\end{proof}

	Thus, in order to prove Proposition \ref{prop2}, it is sufficient to prove Lemma \ref{lemmaU}, where some additional notations need to be specified in advance:
	\begin{enumerate}
		\item For $0\le i \le \lceil \sqrt{N} \rceil$, $k\in \{0,\frac{1}{2}\}$ and edge set $\mathcal{W}$ containing $\mathcal{FI}^{u,T}$, let 
		\begin{equation}
			\mathfrak{U}_{i+k,i+1}(\mathcal{W}):=\left\lbrace  class\ \mathscr{C}\in \mathfrak{U}_i(\mathcal{W}):\left( \cup\mathscr{C}\right)\cap \mathcal{V}_{2(i+1)}=\emptyset,\left( \cup\mathscr{C}\right) \cap \mathcal{V}_{2(i+k)}\neq \emptyset \right\rbrace,
		\end{equation}
		\begin{equation}
			U_{i,i+1}(\mathcal{W}):=|\mathfrak{U}_{i,i+1}(\mathcal{W})|.
		\end{equation}
		
		\item Consider the following mapping $\psi_i: \mathfrak{U}_{i}(\mathcal{W})\setminus \mathfrak{U}_{i,i+1}(\mathcal{W})\to \mathfrak{U}_{i+1}(\mathcal{W})$, 
		\begin{equation}\label{psii}
			\psi_i(\mathscr{C})=\mathscr{C}\setminus \{\mathcal{C}\in \mathscr{C}:\mathcal{C}\cap \mathcal{V}_{2(i+1)}=\emptyset\}.
		\end{equation}
		\begin{remark}
			Note that $\psi_i$ is a bijection and thus
			\begin{equation}\label{Uiomega}
				U_i(\mathcal{W})=U_{i+1}(\mathcal{W})+U_{i,i+1}(\mathcal{W}). 
			\end{equation}
			By (\ref{Uiomega}), one has  
			\begin{equation}\label{mapping2}
				U_{i+a}(\mathcal{W}_i)+\sum_{j:i\le j<i+a}U_{j,j+1}(\mathcal{W}_i)=U_i(\mathcal{W}_i)=U_i. 
			\end{equation}
		\end{remark}
	
		\item For $0\le j< \lceil \sqrt{N} \rceil$, define a set of equivalence classes $\widetilde{\mathfrak{U}}_j$ and an edge set $\mathcal{A}_j$ as follows: if $\mathfrak{U}_{j+\frac{1}{2},j+1}(\mathcal{W}_j)=\emptyset$, let $\mathcal{A}_j=\mathcal{V}_{2j+1}\setminus \mathcal{V}_{2j+2}$ and $\widetilde{\mathfrak{U}}_j=\mathfrak{U}_j(\mathcal{W}_j)\setminus \mathfrak{U}_{j,j+1}(\mathcal{W}_j)$; otherwise, define $\mathcal{A}_j=\mathcal{V}_{2j}\setminus \mathcal{V}_{2j+1}$ and  $\widetilde{\mathfrak{U}}_j:=\mathfrak{U}_j(\mathcal{W}_j)\setminus \mathfrak{U}_{j,j+1}(\mathcal{W}_j)\cup \{\widetilde{\mathscr{C}}\}$, where  $\widetilde{\mathscr{C}}:=\{\mathcal{C}:there\ exists\ 
			\mathscr{C}\in \mathfrak{U}_{j,j+1}(\mathcal{W}_j)\ such\ that\ \mathcal{C}\in \mathscr{C}\}$.
		
	\end{enumerate}
	
%	The following lemma on $\widetilde{\mathfrak{U}}_j$ is parallel to that of Lemma 4.2 in \cite{duminil2020equality}. We hereby include the proof for completeness.
	
	We hereby cite the following combinatorial lemma on $\widetilde{\mathfrak{U}}_j$, which will play an important role in proving Lemma \ref{lemmaU}. The proof of this lemma can be totally found in the proof of Lemma 4.2 in \cite{duminil2020equality}.

	\begin{lemma}\label{tildeU}
	For integers $0\le j\le \lceil \sqrt{N}\rceil$ and $4\le a\le \lfloor N^{\frac{1}{4}}\rfloor$, given that the event $E_j:=A\cap \left\lbrace U_{j+1}>1\vee \left(\frac{U_j}{a}+U_{j,j+1}(\mathcal{W}_j) \right)\right\rbrace $ happens, $\widetilde{\mathfrak{U}}_j$ has following properties with probability one:
		\begin{enumerate}
			\item for each $\mathscr{C}\in \widetilde{\mathfrak{U}}_j$, $\cup\mathscr{C}$ crosses $\mathcal{A}_j$ and $\bigcup_{\mathscr{C}\in \widetilde{\mathfrak{U}}_j}\left( \cup\mathscr{C}\right) $ intersects all the boxes with radius $L_{n_0}$ contained in $V(\mathcal{A}_j)$; 
			
			\item there exists a non-trivial partition $ \widetilde{\mathfrak{U}}_j=\widetilde{\mathfrak{U}}_j^1\cup \widetilde{\mathfrak{U}}_j^2$ such that $1\le |\widetilde{\mathfrak{U}}_j^1|\le a$ and $\widetilde{\mathcal{C}}_j^1\overset{\mathcal{W}_{j+1}}{\not\leftrightarrow}\widetilde{\mathcal{C}}_j^2$, where $\widetilde{\mathcal{C}}_j^i=\bigcup_{\mathscr{C}\in \widetilde{\mathfrak{U}}_j^i}\left( \cup\mathscr{C}\right)$, $i\in \{1,2\}$.  
		\end{enumerate}
	\end{lemma}

	With Lemma \ref{tildeU}, we are ready to conclude the proof of Lemma \ref{lemmaU} and Proposition \ref{prop2}. See Figure \ref{figure3} for an illustration of the main step in the proof of Lemma \ref{lemmaU}.  
	
	\begin{proof}[Proof of Lemma \ref{lemmaU}]
		By (\ref{mapping2}), there must exist an integer $j\in \left[i,i+a \right) $ such that $U_{j,j+1}(\mathcal{W}_j)\le \frac{U_i}{a} $. When the event in (\ref{equationU}) occurs (we denote it by $E$), we have
		\begin{equation}
			U_{j}\ge U_{i+a}> \frac{2U_i}{a}=\frac{U_i}{a}+\frac{U_i}{a}\ge \frac{U_j}{a}+U_{j,j+1}(\mathcal{W}_j),  
		\end{equation}
		which implies that \begin{equation}
			E\subset \bigcup_{0\le j<\lceil \sqrt{N} \rceil} E_j. 
		\end{equation}
		
		\begin{figure}[h]
			\centering
			\includegraphics[width=0.5\textwidth]{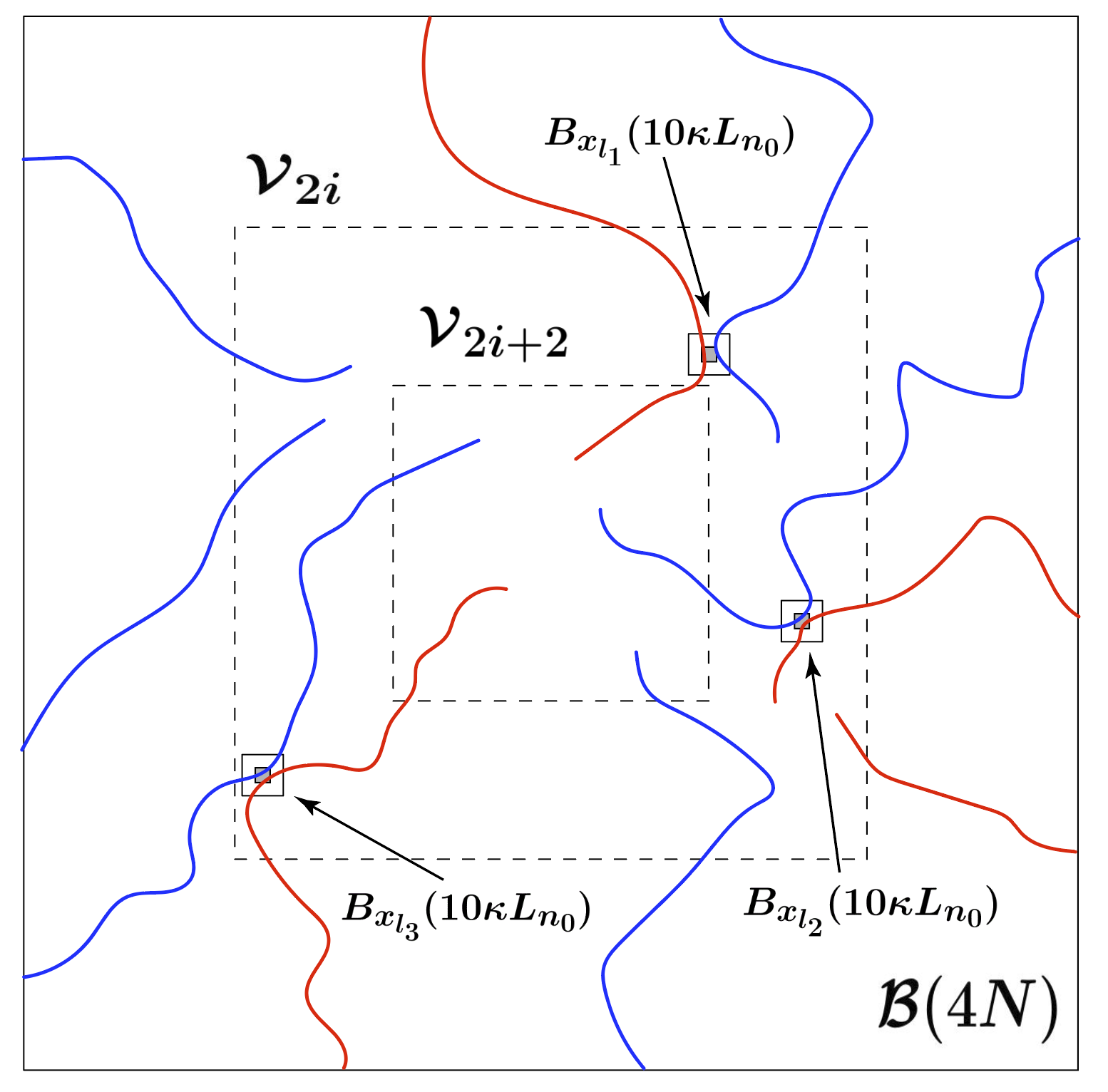}
			\caption{An illustration for the proof of Lemma \ref{lemmaU}: in this picture, the blue lines and red lines represent $\widetilde{\mathcal{C}}_j^{1}$ and $\widetilde{\mathcal{C}}_j^{2}$ respectively. By definitions, we know that $\widetilde{\mathcal{C}}_j^{1}$ and $\widetilde{\mathcal{C}}_j^{2}$ are not connected by $\mathcal{W}_{j+1}$. Meanwhile, they  get close to each other in each box $B_{x_l}(10\kappa L_{n_0})$. Recall that Lemma \ref{bridgelemma} shows there is still an considerable probability that  $\widetilde{\mathcal{C}}_j^{1}$ and $\widetilde{\mathcal{C}}_j^{2}$ are connected within $B_{x_l}(10\kappa L_{n_0})$, which will give an upper bound for the probability of event $E_j$. }
			\label{figure3}
		\end{figure}

		For any $0\le j\le \lceil \sqrt{N} \rceil$, given the event $E_j$, for any $\widetilde{\mathfrak{U}}_j^0$, which is a possible realization of $\widetilde{\mathfrak{U}}_j^0$, by Lemma \ref{tildeU} there exists a partition $\widetilde{\mathfrak{U}}_j^0=\widetilde{\mathfrak{U}}^{1,0}_j\cup \widetilde{\mathfrak{U}}_j^{2,0}$ such that $|\widetilde{\mathfrak{U}}^{1,0}_j|\le a$ and $\widetilde{\mathcal{C}}_j^{1,0}\overset{\mathcal{W}_{j+1}}{\not\leftrightarrow}\widetilde{\mathcal{C}}_j^{2,0}$, where $\widetilde{\mathcal{C}}_j^{i,0}=\bigcup_{\mathscr{C}\in \widetilde{\mathfrak{U}}_j^{i,0}}\left(\cup\mathscr{C}\right)$, $i\in \{1,2\}$.

		Recall the definition of $n_0$ above (\ref{3.3}) and the event $E_j$ in Lemma \ref{tildeU}. We claim that given the event $E_j$, there exist $k=\lfloor \frac{\sqrt{N}}{100\kappa L_{n_0}} \rfloor$ disjoint boxes $B_{x_1}(10\kappa L_{n_0})$,...,$B_{x_k}(10\kappa L_{n_0})$ contained in $V(\mathcal{A}_j)$ with $x_i\in \mathbb{L}_{n_0}$ such that for all $1\le l\le k$, $\widetilde{\mathcal{C}}_j^{1,0}$ and $\widetilde{\mathcal{C}}_j^{2,0}$ both intersect $B_{x_l}(8\kappa L_{n_0})$. In fact, for $1\le l\le k$,  consider $T_l:=\left\lbrace z\in \mathbb{Z}^d:|z|=b_j+50\kappa L_{n_0}+(l-1)80\kappa L{n_0} \right\rbrace$, where $b_j=\lfloor 4N-(2j+1)\sqrt{N} \rfloor$ if $\mathcal{A}_j=\mathcal{V}_{2j}\setminus \mathcal{V}_{2(j+1)}$ and $b_j=\lfloor 4N-(2j+2)\sqrt{N} \rfloor$ if $\mathcal{A}_j=\mathcal{V}_{2j+1}\setminus \mathcal{V}_{2(j+1)}$. Note that for any $l_1\neq l_2$ and $y_1\in T_{l_1},y_2\in T_{l_2}$, the boxes $B_{y_1}(20\kappa L_{n_0})$ and $B_{y_2}(20\kappa L_{n_0})$ are both contained in $V(\mathcal{A}_j)$ and disjoint from each other. For any $x\in T_l$, color it in black if $\widetilde{\mathcal{C}}_1^0$ intersects $B_x(L_{n_0})$ and in red if $\widetilde{\mathcal{C}}_2^0$ intersects $B_x(L_{n_0})$. Note that for each $x\in T_l$, $x$ must be colored for at least once (but possibly twice). Since both $\widetilde{\mathcal{C}}_j^{1,0}$ and $\widetilde{\mathcal{C}}_j^{2,0}$ cross $T_l$, there must exist $x'_l,x''_l\in T_l$ such that $|x'_l-x''_l|_{\infty}=1$, and that $x'_l$, $x''_l$ are colored in black (or both black and red) and in red (or both black and red) respectively. Let $x_l$ be the closest vertex in $\mathbb{L}_{n_0}$ to $\frac{x'+x''}{2}$ (if there are more than one such vertices, let $x_l$ be the lexicographically-smallest one of them). Then $\{B_{x_l}(10\kappa L_{n_0})\}_{1\le l\le k}$ are disjoint, and for all $1\le l\le k$, both $\widetilde{\mathcal{C}}_j^{1,0}$ and $\widetilde{\mathcal{C}}_j^{2,0}$ intersect $B_{x_l}(8\kappa L_{n_0})$. Hence, \begin{equation}\label{4.37}
				\begin{split}
					P[E_j]
					\le \sum_{\widetilde{\mathfrak{U}}_j^0}P\left[ \widetilde{\mathfrak{U}}_j=\widetilde{\mathfrak{U}}_j^0\right]\cdot
					\sum_{\widetilde{\mathfrak{U}}_j^{1,0}\subset \widetilde{\mathfrak{U}}_j^0:|\widetilde{\mathfrak{U}}_j^{1,0}|\le a}P\left[\bigcap\limits_{1\le l\le k}\{\widetilde{\mathcal{C}}_j^{1,0}\xleftrightarrow[B_{x_l}(10\kappa L_{n_0})]{\mathcal{FI}^{u+\epsilon,T}}  \widetilde{\mathcal{C}}_j^{2,0}\}^c\bigg|\widetilde{\mathfrak{U}}_j=\widetilde{\mathfrak{U}}_j^0 \right]. 
				\end{split}
			\end{equation}

			For $i\in \{1,2\}$, let $\mathcal{S}_i^l:=\mathcal{B}_{x_l}(10\kappa L_{n_0})\cap \left(\widetilde{\mathcal{C}}_j^{i,0}\cup \partial_e^{out}\widetilde{\mathcal{C}}_j^{i,0} \right)$. When the event $\mathcal{S}_1^l\xleftrightarrow[B_{x_l}(10\kappa L_{n_0})]{\mathcal{FI}^{u+\epsilon,T}}\mathcal{S}_2^l$ occurs, there exists edge set $\{e^l_1,e^l_2\}$ such that $e^l_i\in \partial_e^{out}\widetilde{\mathcal{C}}_j^{i,0}\cap \mathcal{B}_{x_l}(10\kappa L_{n_0})$ for $i\in \{1,2\}$ and $e_1^l\xleftrightarrow[B_{x_l}(10\kappa L_{n_0})]{\mathcal{FI}^{u+\epsilon,T}}e_2^l$. We denote by $\{e^{l,0}_1,e^{l,0}_2\}$ the lexicographically-smallest one among all edge sets $\{e^l_1,e^l_2\}$. Note that $\mathcal{A}_j\cap \left(\mathcal{B}(4N)\setminus \mathcal{V}_{2j}\right)=\emptyset$ and that the event $\widetilde{\mathfrak{U}}_j=\widetilde{\mathfrak{U}}_j^0$ only depends on 
			\begin{equation}\label{measurable}
				\sum_{(a_i,\eta_i)\in \mathcal{FI}^{T}}\delta_{\eta}\cdot\left(\mathbbm{1}_{0<a_i\le u,\eta_i\cap \mathcal{B}(4N)\neq \emptyset}+\mathbbm{1}_{u<a_i\le u+\epsilon,\eta_i\cap (\mathcal{B}(4N)\setminus \mathcal{V}_{2j})\neq \emptyset }\right),
			\end{equation}
			which is independent to $\underline{\mathcal{FI}}^{\epsilon,T}_1\cap \mathcal{A}_j$, where $\underline{\mathcal{FI}}^{\epsilon,T}_1:=\sum_{(a_i,\eta_i)\in \mathcal{FI}^{T}}\delta _{\eta}\cdot\mathbbm{1}_{u<a_i\le u+\epsilon,|\eta_i|=1}$.

			%the event $\bigcap_{i=1,2} \{e^{l,0}_i\in \underline{\mathcal{FI}}^{\epsilon,T}_1\}$ (  $\underline{\mathcal{FI}}^{\epsilon,T}_1:=\sum_{(a_i,\eta_i)\in \mathcal{FI}}\mathbbm{1}_{\eta}*\mathbbm{1}_{u<a_i\le u+\epsilon,|\eta_i|=1}$). 

			By $\bigcap\limits_{i\in \{1,2\}} \{e^{l,0}_i\in \underline{\mathcal{FI}}^{\epsilon,T}_1\}\cap \left\lbrace \mathcal{S}_1^l\xleftrightarrow[B_{x_l}(10\kappa L_{n_0})]{\mathcal{FI}^{u+\epsilon,T}}\mathcal{S}_2^l\right\rbrace \subset \left\lbrace \widetilde{\mathcal{C}}_j^{1,0}\xleftrightarrow[B_{x_l}(10\kappa L_{n_0})]{\mathcal{FI}^{u+\epsilon,T}}\widetilde{\mathcal{C}}_j^{2,0}\right\rbrace$, for any $1\le l\le k$, 
			\begin{equation}\label{448}
				\begin{split}
					&P\left[ \{\widetilde{\mathcal{C}}_j^{1,0}\xleftrightarrow[B_{x_l}(10\kappa L_{n_0})]{\mathcal{FI}^{u+\epsilon,T}}  \widetilde{\mathcal{C}}_j^{2,0}\}^c\bigg|\widetilde{\mathfrak{U}}_j=\widetilde{\mathfrak{U}}_j^0,\bigcap_{1\le m<l}\{\widetilde{\mathcal{C}}_j^{1,0}\xleftrightarrow[B_{x_m}(10\kappa L_{n_0})]{\mathcal{FI}^{u+\epsilon,T}}  \widetilde{\mathcal{C}}_j^{2,0}\}^c\right]\\
					\le &P\left[ \{\mathcal{S}_1^l\xleftrightarrow[B_{x_l}(10\kappa L_{n_0})]{\mathcal{FI}^{u+\epsilon,T}}  \mathcal{S}_2^l\}^c\bigg|\widetilde{\mathfrak{U}}_j=\widetilde{\mathfrak{U}}_j^0,\bigcap_{1\le m<l}\{\widetilde{\mathcal{C}}_j^{1,0}\xleftrightarrow[B_{x_m}(10\kappa L_{n_0})]{\mathcal{FI}^{u+\epsilon,T}}  \widetilde{\mathcal{C}}_j^{2,0}\}^c\right]\\
					&+P\left[ \mathcal{S}_1^l\xleftrightarrow[B_{x_l}(10\kappa L_{n_0})]{\mathcal{FI}^{u+\epsilon,T}}  \mathcal{S}_2^l\bigg|\widetilde{\mathfrak{U}}_j=\widetilde{\mathfrak{U}}_j^0,\bigcap_{1\le m<l}\{\widetilde{\mathcal{C}}_j^{1,0}\xleftrightarrow[B_{x_m}(10\kappa L_{n_0})]{\mathcal{FI}^{u+\epsilon,T}}  \widetilde{\mathcal{C}}_j^{2,0}\}^c\right]\\
					&\ \ \ \cdot P\left[\bigcup\limits_{i=1,2} \{e^{l,0}_i\notin \underline{\mathcal{FI}}^{\epsilon,T}_1\}\bigg|\mathcal{S}_1^l\xleftrightarrow[B_{x_l}(10\kappa L_{n_0})]{\mathcal{FI}^{u+\epsilon,T}}  \mathcal{S}_2^l,\widetilde{\mathfrak{U}}_j=\widetilde{\mathfrak{U}}_j^0,\bigcap_{1\le m<l}\{\widetilde{\mathcal{C}}_j^{1,0}\xleftrightarrow[B_{x_m}(10\kappa L_{n_0})]{\mathcal{FI}^{u+\epsilon,T}}  \widetilde{\mathcal{C}}_j^{2,0}\}^c\right]\\
					=&1-c(\epsilon)\cdot P\left[ \mathcal{S}_1^l\xleftrightarrow[B_{x_l}(10\kappa L_{n_0})]{\mathcal{FI}^{u+\epsilon,T}}  \mathcal{S}_2^l\bigg|\widetilde{\mathfrak{U}}_j=\widetilde{\mathfrak{U}}_j^0,\bigcap_{1\le m<l}\{\widetilde{\mathcal{C}}_j^{1,0}\xleftrightarrow[B_{x_m}(10\kappa L_{n_0})]{\mathcal{FI}^{u+\epsilon,T}}  \widetilde{\mathcal{C}}_j^{2,0}\}^c\right]. 
				\end{split}
			\end{equation}

			We claim that for any $1\le j\le \lceil \sqrt{N}\rceil$ and $1\le l\le k$, the event $\widetilde{\mathfrak{U}}_j=\widetilde{\mathfrak{U}}_j^0$ is measurable w.r.t. the following $\sigma$-field: 
			\begin{equation}
				\mathcal{F}_l:=\sigma\left( \sum_{(a_i,\eta_i)\in \mathcal{FI}^{T}}\delta_{(a_i,\eta_i)}\cdot \mathbbm{1}_{0<a_i\le u+\epsilon,\eta\cap (\mathcal{S}_1^l\cup\mathcal{S}_2^l\cup \mathcal{B}^c_{x_l}(10\kappa L_{n_0}))\cap \mathcal{B}(4N)\neq \emptyset }\right).
			\end{equation}
		Since $\widetilde{\mathfrak{U}}_j=\widetilde{\mathfrak{U}}_j^0$ is measurable w.r.t. $\sum_{(a_i,\eta_i)\in \mathcal{FI}}\delta_{(a_i,\eta_i)}\cdot \mathbbm{1}_{0<a_i\le u+\epsilon,\eta\cap \mathcal{B}(4N)\neq \emptyset }$, it is sufficient to confirm that $\widetilde{\mathfrak{U}}_j=\widetilde{\mathfrak{U}}_j^0$ is independent to $\sum_{(a_i,\eta_i)\in \mathcal{FI}}\delta_{(a_i,\eta_i)}\cdot \mathbbm{1}_{0<a_i\le u+\epsilon,\eta\subset \mathcal{B}_{x_l}(10\kappa L_{n_0})\setminus (\mathcal{S}_1^l\cup\mathcal{S}_2^l)  }$. For each given $\widetilde{\mathfrak{U}}_j^0$, we enumerate the clusters in $\cup_{\mathscr{C}\in \widetilde{\mathfrak{U}}_j^0} \mathscr{C}$ by $ \left\lbrace \mathcal{C}_1,...,\mathcal{C}_n \right\rbrace $. To ensure the event $\widetilde{\mathfrak{U}}_j=\widetilde{\mathfrak{U}}_j^0$ happens, it is sufficient to have that for any $1\le i\le n$, $\left( \mathcal{C}_i\cup \partial_e^{out}\mathcal{C}\right)\cap \mathcal{FI}^{u,T}=\mathcal{C}_i$ and that $W_j$, $W_{j+1}$ group $\left\lbrace \mathcal{C}_1,...,\mathcal{C}_n \right\rbrace$ following the way given in $\widetilde{\mathfrak{U}}_j^0$. Since each path in $\mathcal{FI}^{u+\epsilon,T}$ that totally contained in $\mathcal{B}_{x_l}(10\kappa L_{n_0})\setminus (\mathcal{S}_1^l\cup\mathcal{S}_2^l)$ does not intersect any $\mathcal{C}_i$, it will not have any impact on $\left( \mathcal{C}_i\cup \partial_e^{out}\mathcal{C}\right)\cap \mathcal{FI}^{u,T}$ and the way how $\left\lbrace \mathcal{C}_1,...,\mathcal{C}_n \right\rbrace$ are grouped. In conclusion, $\widetilde{\mathfrak{U}}_j=\widetilde{\mathfrak{U}}_j^0$ is measurable w.r.t. $\mathcal{F}_l$.

%		{\color{red}	In fact, if $\mathfrak{U}_{j+\frac{1}{2},j+1}(\mathcal{W}_j)\neq \emptyset$, noting that $\bigcup_{\mathscr{C}\in \widetilde{\mathfrak{U}}_j}\mathscr{C}=\bigcup_{\mathscr{C}\in \mathfrak{U}_j(\mathcal{W}_j)}\mathscr{C}$ and $\mathcal{W}_j\in \mathcal{F}_l$ (since $\mathcal{B}(4N)\setminus \mathcal{V}_{2j}\subset\mathcal{B}^c_{x_l}(10\kappa L_{n_0})$), we immediately have $\{\widetilde{\mathfrak{U}}_j=\widetilde{\mathfrak{U}}_j^0\}\in \mathcal{F}_l$.
%			
%			
%			
%			
%			 Otherwise (i.e. $\mathfrak{U}_{j+\frac{1}{2},j+1}(\mathcal{W}_j)= \emptyset$), every class $\mathscr{C}\in \mathfrak{U}_{j,j+1}(\mathcal{W}_j)$ does not intersect $\mathcal{V}_{2j+1}$. Since $\mathcal{B}_{x_l}(10\kappa L_{n_0})\subset \mathcal{A}_j=\mathcal{V}_{2j+1}\setminus\mathcal{V}_{2j+2}$, we have $\widetilde{\mathfrak{U}}_j\cap \mathcal{B}_{x_l}(10\kappa L_{n_0})=\mathfrak{U}_j(\mathcal{W}_j)\cap \mathcal{B}_{x_l}(10\kappa L_{n_0})$ and thus $\{\widetilde{\mathfrak{U}}_j=\widetilde{\mathfrak{U}}_j^0\}\in \mathcal{F}_l$ also holds. 
%		}
			
			For each $1\le l\le k$, noting that $\widetilde{\mathfrak{U}}_j=\widetilde{\mathfrak{U}}_j^0$, $\bigcap_{1\le m<l}\{\widetilde{\mathcal{C}}_j^{1,0}\xleftrightarrow[B_{x_m}(10\kappa L_{n_0})]{\mathcal{FI}^{u+\epsilon,T}}  \widetilde{\mathcal{C}}_j^{2,0}\}^c$ are both measurable w.r.t. $\mathcal{F}_l$, by Lemma \ref{bridgelemma} we have
			\begin{equation}\label{4.39}
				\begin{split}
					P\left[ \mathcal{S}_1^l\xleftrightarrow[B_{x_l}(10\kappa L_{n_0})]{\mathcal{FI}^{u+\epsilon,T}}  \mathcal{S}_2^l\bigg|\widetilde{\mathfrak{U}}_j=\widetilde{\mathfrak{U}}_j^0,\bigcap_{1\le m<l}\{\widetilde{\mathcal{C}}_j^{1,0}\xleftrightarrow[B_{x_m}(10\kappa L_{n_0})]{\mathcal{FI}^{u+\epsilon,T}}  \widetilde{\mathcal{C}}_j^{2,0}\}^c\right]
					\ge e^{-C_4(\log(L_{n_0}))^2}.
				\end{split}
		\end{equation}

	Since $|\widetilde{\mathfrak{U}}_j^0|\le CN^{d-1}$, there are no more than $(CN^{d-1})^a$ subsets of $\widetilde{\mathfrak{U}}_j^0$, whose cardinalities are not larger than $a$. By (\ref{4.37}), (\ref{448}) and (\ref{4.39}), we have
			\begin{equation}
				\begin{split}
					P[E_j]\le&\sum_{\widetilde{\mathfrak{U}}_j^0}P\left[ \widetilde{\mathfrak{U}}_j=\widetilde{\mathfrak{U}}_j^0\right]\sum_{\widetilde{\mathfrak{U}}_j^{1,0}\subset \widetilde{\mathfrak{U}}_j^0:|\widetilde{\mathfrak{U}}_j^{1,0}|\le a}\prod_{l=1}^{k}P\bigg[ \{\widetilde{\mathcal{C}}_j^{1,0}\xleftrightarrow[B_{x_l}(10\kappa L_{n_0})]{\mathcal{FI}^{u+\epsilon,T}}  \widetilde{\mathcal{C}}_j^{2,0}\}^c\bigg|\widetilde{\mathfrak{U}}_j=\widetilde{\mathfrak{U}}_j^0,\\
					&\ \ \ \ \ \ \ \ \ \ \ \ \ \ \ \ \ \ \ \ \ \ \ \ \ \ \ \ \ \ \ \ \ \ \ \ \ \ \ \ \ \ \ \ \  \bigcap_{1\le m<l}\{\widetilde{\mathcal{C}}_j^{1,0}\xleftrightarrow[B_{x_m}(10\kappa L_{n_0})]{\mathcal{FI}^{u+\epsilon,T}}  \widetilde{\mathcal{C}}_j^{2,0}\}^c\bigg]\\
					\le &\sum_{\widetilde{\mathfrak{U}}^0_j}P\left[ \widetilde{\mathfrak{U}}_j=\widetilde{\mathfrak{U}}_j^0\right]\cdot(CN^{d-1})^a\left[1-c\cdot e^{-C_4(\log(L_{n_0}))^2}\right]^k\\
					=&(CN^{d-1})^a\left[1-c\cdot e^{-C_4(\log(L_{n_0}))^2}\right]^k. 
				\end{split}
		\end{equation}

		Recalling the definition of $n_0$, we have $\mu(10N)\le L_{n_0}\le l_0\cdot \mu(10N)$. Hence, one has $k\ge c'\sqrt{N}/\mu(10N)$ and $1-e^{-C_4(\log(L_{n_0}))^2}\le 1-e^{-C'(\log(\mu(10N)))^2}$. 
		
		Since $a\le N^{1/4}$ and $\log(1-x)\le -x$ ($x>0$), we have: for each $c'(\epsilon)>0$, there exists integer $N_0(\epsilon)>0$ such that for all $N\ge N_0$,
			\begin{equation}\label{PE}
				\begin{split}
					P[E]\le \sum_{j=0}^{\lceil\sqrt{N} \rceil}P[E_j]\le &\lceil\sqrt{N} \rceil\cdot(CN^{d-1})^a\cdot\left[1-e^{-C_4(\log(L_{n_0}))^2}\right]^k\\
					\le &\lceil\sqrt{N} \rceil\cdot(CN^{d-1})^{ N^{1/4}}\cdot\left[1-e^{-C'(\log(\mu(10N)))^2}\right]^{c'\sqrt{N}/\mu(10N)}\\
					\le &\exp(\log(N)+N^{1/4}\log(CN^{d-1})-\frac{c'\sqrt{N}}{\mu(10N)}\cdot e^{-C''(\log(10N))^{2/3}})        \\
					\le &\exp(-c''N^{1/4}). 
				\end{split}
			\end{equation}
			From (\ref{PE}), we finally get (\ref{equationU}).
	\end{proof}

	\subsection{Renomalization}
	
	We define the following events: for $x\in \mathbb{L}_0$, let 
		\begin{equation}
			F^{(1)}_x=\left\lbrace B_x(2L_0+1)\xleftrightarrow[]{\mathcal{FI}_{L_0}^{u-3\epsilon,T}} \partial B_x(6(2L_0+1)) \right\rbrace,
		\end{equation}
		\begin{equation}
			\begin{split}
				F^{(2)}_x=\{ &all\ clusters\ in\ \mathcal{FI}_{L_0}^{u-\frac{5\epsilon}{3},T}\cap \mathcal{B}_x(4(2L_0+1))\ crossing\ B_x(4(2L_0+1))\setminus B_x(2(2L_0+1))\\
				&\ \ \ are\ connected\ by\ \mathcal{FI}_{L_0}^{u-\frac{4\epsilon}{3},T}\cap \mathcal{B}_x(4(2L_0+1)) \},
			\end{split}
	\end{equation}
	and
	\begin{equation}\label{454}
		F_{0,x}:=F^{(1)}_x\cap F^{(2)}_x.
	\end{equation}
	
	Fix a constant $\epsilon>0$. For $x\in \mathbb{L}_0$, we say that a box $B_x(L_0)$ is $(u,L_0)$-good if the event $F_{0,x}$ occurs. Furthermore, we also say a cluster $A\subset \mathbb{Z}^d$ is a $(u,L_0)$-good cluster if $A$ is a union of some $(u,L_0)$-good boxes.

	Similar to Proposition 4 in \cite{duminil2020equality}, we are to show that in a given box $B(N)$, with a high probability there exists a sufficiently large $(u,L_0)$-good cluster intersecting all big connected subsets of $B(N)$.

	\begin{proposition}\label{prop_J}
		For any $\epsilon>0$ and $u\ge \widetilde{u}+6\epsilon$, there exist $\rho(u,\epsilon)$, $L_0(u,\epsilon)$ and $c_4(u,\epsilon)>0$ such that for all integers $N\ge 1$, 
			\begin{equation}\label{prop_J1}
				\begin{split}
					P[&there\ exists\ (u,L_0)-good\ \ cluster\ J_N\subset B(2N)\ intersecting\ every\ cluster\ \\
					&\ \  \mathcal{S}\subset \mathcal{B}(N)\ with\ diameter\ge N^{\frac{1}{2}} ] \ge 1-\exp(-c_4N^\rho).
				\end{split}
		\end{equation}
	\end{proposition}

	\begin{proof}
		Recall the definition of the event $\xi(\cdot,\cdot,\cdot)$ in (\ref{eventxi}) and let $\xi_x:=x+\xi$. Moreover, one may note that 
		\begin{equation}\label{56}
			\xi_x(2L_0+1,u-4\epsilon,u-3\epsilon)\cap \left\lbrace \mathcal{B}_x(6(2L_0+1))\cap \left( \mathcal{FI}^{u-3\epsilon}-\mathcal{FI}_{L_0}^{u-3\epsilon}\right) =\emptyset  \right\rbrace   	    \subset F^{(1)}_x,
		\end{equation}
		\begin{equation}\label{old57}
			\xi_x(2L_0+1,u-\frac{14}{9}\epsilon,u-\frac{13}{9}\epsilon)\cap \left\lbrace \mathcal{B}_x(6(2L_0+1))\cap \left( \mathcal{FI}^{u-\frac{13}{9}\epsilon}-\mathcal{FI}_{L_0}^{u-\frac{13}{9}\epsilon}\right) =\emptyset  \right\rbrace  \subset F^{(2)}_x.
		\end{equation}
		
		By (\ref{PeFI}), for any $u'>0$, we have
			\begin{equation}\label{u'}
				\begin{split}
					&P\left[\mathcal{B}_x(6(2L_0+1))\cap \left( \mathcal{FI}^{u'}-\mathcal{FI}_{L_0}^{u'}\right) \neq\emptyset \right]\le C(u')e^{-c(u')L_0}. 
				\end{split}
		\end{equation}
	By (\ref{56}), (\ref{old57}) and (\ref{u'}), for any $u\ge \widetilde{u}+6\epsilon$ and $x\in \mathbb{L}_0$, we have 
		\begin{equation}\label{58}
			\begin{split}
				&P[F^c_{0,x}]\\
				\le	& P\left[\xi^c_x\left(2L_0+1,u-\frac{14}{9}\epsilon,u-\frac{13}{9}\epsilon\right)\right]+P\left[\xi^c_x\left(2L_0+1,u-\frac{14}{9}\epsilon,u-\frac{13}{9}\epsilon\right)\right]+C(u)e^{-c(u)L_0}.
			\end{split}
		\end{equation}
	Combine Proposition \ref{prop2} and (\ref{58}), 
		\begin{equation}\label{largel0}
			\limsup\limits_{L_0\to \infty}P[F_{0,x}]=1. 
		\end{equation}

		Note that it is sufficient to prove (\ref{prop_J1}) in the case when $N\ge (L_1)^2$. Let $n=\max\{k\in \mathbb{N}:L_{k+1}\le \sqrt{N}\}$. Then we have 
		\begin{equation}
			\frac{\sqrt{N}}{l_0^2}<L_n\le \frac{\sqrt{N}}{l_0}. 
		\end{equation}

		For any $m\ge 0$ and a subset $A\subset \mathbb{L}_m$, we say $A$ is $L_m$-connected if the set of vertices $\{\frac{z}{2L_m+1}:z\in A\}$ is connected in $\mathbb{Z}^d$. 
		
		For each $x\in \mathbb{L}_n$, let $Comp(x)$ be the largest $L_0$-connected component (in the sense of diameter) consisting of vertices $y\in \mathbb{L}_0\cap B_x(4\kappa L_n)$ such that the event $F_{0,y}$ happens (if there are more than one ``largest clusters'', let $Comp(x)$ be the lexicographically-smallest one of them). Define the following $(u,L_0)$-good cluster:
		\begin{equation}
			J_N=\bigcup_{x\in \mathbb{L}_n\cap B(N)} \bigcup_{y\in Comp(x)} B_y(L_0).
		\end{equation}
		Note that $J_N\subset B(N+4\kappa L_n+L_0)\subset B(2N)$.

		Recall the event $\mathscr{G}_n^0$ in (\ref{Gn0}) and define $\mathscr{G}_{n,x}^0=x+\mathscr{G}_n^0$. Let $G_N=\bigcap_{x\in \mathbb{L}_n\cap B(N)} \mathscr{G}_{n,x}^0$. Then we try to confirm that when $G_N$ happens, $J_N$ satisfies the condition in (\ref{prop_J1}):

		We first check that given the event $G_N$, $J_N$ is $L_0$-connected: for any $x,y\in \mathbb{L}_n\cap B(N)$ such that $|x-y|_1=2L_n+1$, since each pair of oppsite faces of $B_x(\kappa L_n)$ are both subsets of $B_x(10\kappa L_n)$ with diameters larger than $\kappa L_n$ (in other words, they are both $0$-admissible), there must exist a $0$-bridge $\mathfrak{B}^0$ with $(u,L_0)$-boxes connecting both of the two opposite faces. Since the diameter of $\{z: B_z(L_0)\in \mathfrak{B}^0\}$ must be larger than $\kappa L_n$ and $Comp(x)$ is the largest, we have that  $\mathcal{S}_x:=\bigcup_{z\in Comp(x)} \mathcal{B}_z(L_0)$ is $0$-admissible in both $\mathcal{B}_x(10\kappa L_n)$ and $\mathcal{B}_y(10\kappa L_n)$. In the same way, $\mathcal{S}_y$ is $0$-admissible in both $\mathcal{B}_x(10\kappa L_n)$ and $\mathcal{B}_y(10\kappa L_n)$ as well. By the definition of $G_N$, we have that
		\begin{equation}
			\mathcal{S}_x\cap \mathcal{S}_y\neq \emptyset,
		\end{equation}
		which implies that $J_N$ is connected.

		Meanwhile, for any connected subset $\mathcal{S}\subset \mathcal{B}(N)$ with diameter at least $\sqrt{N}$, we have $\text{diam}(S)\ge \sqrt{N}\ge 10\kappa L_n$. Arbitrarily choose a vertex $x\in \mathbb{L}_n\cap B(N)$ such that $\mathcal{S}\cap \mathcal{B}_x(L_n)\neq \emptyset$. Then $ \mathcal{S}\cap \mathcal{B}_x(4\kappa L_n)$ contains a cluster with diameter at least $\kappa L_n $. Sicne the event $G_N$ occurs, there must exist a $0$-bridge $\mathfrak{B}^0$ with $(u,L_0)$-boxes intersctig both $\mathcal{S}$ and $\mathcal{S}_x$. Therefore, we have that $\mathcal{S} \cap \mathcal{S}_x\neq \emptyset $. In conclusion, if we denote the event in (\ref{prop_J1}) by {\color{blue}$D$}, then 
		\begin{equation}\label{V}
			P\left[D\right]\ge P[G_N]. 
		\end{equation}
		
		Note that there exists a large enough $L_0$ such that events $\{F_{0,x}\}_{x\in \mathbb{L}_0}$ defined in (\ref{454}) satisfy all the conditions in Lemma \ref{0-bridgelemma}. In fact, since $F_{0,x}$ only depends on $\sum_{\eta\in \mathcal{FI}^{u,T}_{L_0}}\delta_{\eta}\cdot \mathbbm{1}_{\eta(0)\in B_{x}(20L_0)}$ and $\kappa\ge 100$, the first condition for $F_{0,x}$ in Lemma \ref{0-bridgelemma} holds; meanwhile, by (\ref{largel0}), we can choose a sufficient large $L_0$ such that $P[F_{0,x}]\le c_1$ for all $x\in \mathbb{L}_0$.	
		
		By (\ref{largel0}) and Lemma \ref{0-bridgelemma}, there exists a sufficient large $L_0$ such that   
		\begin{equation}\label{GN}
			P\left[D\right]\ge P[G_N]\ge 1-2^{-2^n}=1-\exp(-c'N^\rho). 
		\end{equation}
	\end{proof}

	\subsection{Proving SQ3}
	
	With Proposition \ref{prop_J}, we are ready to hereby finish the proof of \textbf{SQ3}.

	We first note that it suffices to prove that for any $u>\widetilde{u}$, there exists $c(u)>0$, $\rho(u)>0$ such that for all integer $N\ge 1$,
	\begin{equation}\label{ext1}
		P\left[\text{Exist}(N,\mathcal{FI}^{u,T})^c \right]\le \exp(-cN^\rho),
	\end{equation} 
	\begin{equation}\label{3.9}
		P\left[\text{Unique}(N,\mathcal{FI}^{u,T})^c \right]\le \exp(-cN^\rho). 
	\end{equation}
	
	For any $u>\widetilde{u}$, let $\epsilon:=\frac{u-\widetilde{u}}{10}$. By defintion we have $u\ge \widetilde{u}+6\epsilon$.

	For the first part, when the event in (\ref{prop_J1}) occurs, there exists at least one $(u,L_0)$-good cluster satisfying the condition of $J_N$ in (\ref{prop_J1}). We denote the lexicographically-smallest one of them by $J_N^0$. Note that $\text{diam}(J_N^0)\ge 2N-2\sqrt{N}$. Recalling the definition of ``$(u,L_0)$-good'' below (\ref{454}), there exists a connected cluster $\mathcal{J}_N^0\subset \mathcal{FI}^{u-\epsilon,T}$ intersecting every $(u,L_0)$-box in $J_N^0$. Therefore, (\ref{ext1}) holds by Proposition \ref{prop_J}.

	For the second part, we denote the event in (\ref{prop_J1}) by $A$. For any $x\in B(N)$, let $\mathcal{C}(x)$ be the connected cluster in $\mathcal{FI}^{u,T}\cap \mathcal{B}(N)$ containing $x$ (if such cluster does not exist, set $\mathcal{C}(x):=\emptyset$). 
	
	If $\text{Unique}(N,\mathcal{FI}^{u,T})$ does not happen, there must exist two clusters $\mathcal{L}_1$ and $\mathcal{L}_2$ with diameters $\ge \frac{R}{10} $ in $\mathcal{FI}^{u,T}\cap \mathcal{B}(N)$, which are not connected within $\mathcal{FI}^{u,T}\cap \mathcal{B}(2N)$. Since $\mathcal{J}_N^0\subset \mathcal{FI}^{u-\epsilon,T}$, there must exist one cluster in $\{\mathcal{L}_1,\mathcal{L}_2\}$ such that $V(\mathcal{L}_i)\cap V(\mathcal{J}_N^0)=\emptyset $. By Proposition \ref{prop_J},
	\begin{equation}\label{3.10}
		\begin{split}
			&P\left[\text{Unique}(N,\mathcal{FI}^{u,T})^c \right]\\
			\le &P\left[A^c \right]+P\left[\text{Unique}(N,\mathcal{FI}^{u,T})^c\cap A  \right]\\
			\le &\exp(-cN^\rho)+P\left[\text{Unique}(N,\mathcal{FI}^{u,T})^c\cap A  \right]\\
			\le &\exp(-cN^\rho)+\sum_{x\in B(N)}P\left[A,\text{diam}(\mathcal{C}(x))\ge N/10,V(\mathcal{C}(x))\cap V(\mathcal{J}_N^0)=\emptyset\right].  
		\end{split} 
	\end{equation}
	
	It is sufficient to prove: there exist $c'(u),\rho'(u)>0$ such that for any $N\ge 1$ and $x\in B(N)$, \begin{equation}\label{3.11}
		P\left[A,\text{diam}(\mathcal{C}(x))\ge N/10,V(\mathcal{C}(x))\cap V(\mathcal{J}_N^0)=\emptyset\right]\le \exp(-c'N^{\rho'}). 
	\end{equation}

	For any $L\in \mathbb{N}^+$, consider the following point measure \begin{equation}
			\underline{\mathcal{FI}}^{\epsilon,T}_L:=\sum_{(a_n,\eta_n)\in \mathcal{FI}^{T}}\delta_{\eta_n}\cdot \mathbbm{1}_{u-\epsilon<a_n\le u,|\eta_n|=L} 
		\end{equation}
		and the $\sigma$-filed $\mathcal{F}:=\sigma\left(\left( \mathcal{FI}^{u,T}-\underline{\mathcal{FI}}^{\epsilon,T}_1\right)\cap \mathcal{B}(2N) \right)$. Note that $\underline{\mathcal{FI}}^{\epsilon,T}_1$ has the same distribution as a Bernoulli bond percolation. We denote the law of $\underline{\mathcal{FI}}^{\epsilon,T}_1$ by $P_1$.

	Since the event $A$ is measurable w.r.t. $\mathcal{F}$, we have \begin{equation}\label{3.13}
		\begin{split}
			&P\left[A,\text{diam}(\mathcal{C}(x))\ge N/10,V(\mathcal{C}(x))\cap V(\mathcal{J}_N^0)=\emptyset \right]\\
			=&E\left[\mathbbm{1}_A\cdot P_1\left[\text{diam}(\mathcal{C}(x))\ge N/10,V(\mathcal{C}(x))\cap V(\mathcal{J}_N^0)=\emptyset \big| \mathcal{F}  \right]  \right]. 
		\end{split}
	\end{equation}

	For each $x\in B(N)$, denote the connected component in $(\mathcal{FI}^{u,T}- \underline{\mathcal{FI}}^{\epsilon,T}_1 )\cap \mathcal{B}(2N)$ containing $x$ by $\widetilde{\mathcal{C}}(x)$ (we also set $\widetilde{\mathcal{C}}(x)=\emptyset$ if there is no such a cluster). Note that conditioned on $\mathcal{F}$, the edge sets $J_N^0$, $\mathcal{J}^0_N$ and $\widetilde{\mathcal{C}}(x)$ are all deterministic.

	For $x\in B(N)$, similar to the proof of Proposition 1.5 in \cite{duminil2020equality}, we run a random algorithm $\mathfrak{T}_x$ under the following laws:
	\begin{enumerate}
		\item Initially, set $\widetilde{\mathcal{C}}_0(x):=\widetilde{\mathcal{C}}(x)$, $\text{Open}_0=\emptyset$, $\text{Close}_0:=\emptyset$ and $D_0:=J_N^0$.
		
		\item For step $i$, we already have $\widetilde{\mathcal{C}}_i(x)$, $\text{Open}_i$, $\text{Close}_i$ and $\mathcal{D}_i$. If $\partial_e^{out} \widetilde{\mathcal{C}}_i(x) \setminus \text{Close}_i \subset \partial_e^{out} \mathcal{B}(2N)$ or $\widetilde{\mathcal{C}}_i(x)$ intersects $V(\mathcal{J}_N^0)$, stop the process. Otherwise, we denote the lexicographically-smallest edge in $\partial_e^{out} \widetilde{\mathcal{C}}_i(x)$ by $e_{i+1}$.	
		%$e_{i+1}=\{x_{i+1},y_{i+1}\}$
		% (where $x_{i+1}\in \widetilde{\mathcal{C}}_i(x)$, $y_{i+1}\notin \widetilde{\mathcal{C}}_i(x) $) and sample the configuration at $e_{i+1}$ under $P_p$. 
		\begin{enumerate}
			\item If $e_{i+1}\in \underline{\mathcal{FI}}^{\epsilon,T}_1$ and $e_{i+1}$ does not interset $D_i$, let $\text{Open}_{i+1}=\text{Open}_{i}\cup \{e_{i+1}\}$, $\text{Close}_{i+1}=\text{Close}_{i}$, $D_{i+1}=D_i$ and let $\widetilde{\mathcal{C}}_{i+1}(x)$ the connected cluster in $[(\mathcal{FI}^{u,T}- \underline{\mathcal{FI}}^{\epsilon,T}_1 )\cap \mathcal{B}(2N)]\cup \text{Open}_{i+1}$ containing $x$;

			\item If $e_{i+1}\in \underline{\mathcal{FI}}^{\epsilon,T}_1$ and $e_{i+1}$ intersets $D_i$ at the boundary of a $0$-vertex-box $B_{x_{i+1}}(L_0)$, sample the configuration of $\underline{\mathcal{FI}}^{\epsilon,T}_1\cap \mathcal{B}_{x_{i+1}}(L_0)\setminus \left( \text{Open}_i \cup \text{Close}_i \right) $. Then denote the sets of open and closed edges in it by $\mathcal{OP}_{i+1}$ and $\mathcal{CL}_{i+1}$ respectively. Let $\text{Open}_{i+1}=\text{Open}_{i}\cup \mathcal{OP}_{i+1}$, $\text{Close}_{i+1}=\text{Close}_{i}\cup\mathcal{CL}_{i+1}$, $D_{i+1}=D_i\setminus B_{x_{i+1}}(L_0)$ and let $\widetilde{\mathcal{C}}_{i+1}(x)$ be the connected cluster in $[(\mathcal{FI}^{u,T}- \underline{\mathcal{FI}}^{\epsilon,T}_1 )\cap \mathcal{B}(2N)]\cup \text{Open}_{i+1}$ containing $x$;

			\item If $e_{i+1}\notin \underline{\mathcal{FI}}^{\epsilon,T}_1$, let $\widetilde{\mathcal{C}}_{i+1}(x)=\widetilde{\mathcal{C}}_{i}(x)$, $\text{Open}_{i+1}=\text{Open}_i$, $D_{i+1}=D_i$ and $\text{Close}_{i+1}=\text{Close}_{i}\cup \{e_{i+1}\}$. 
		\end{enumerate}
		
	\end{enumerate}
	
	Given configuration of $\mathcal{F}$, we denote by $N_1$ the number of steps that the condition in (b) holds before the random algorithm terminates. When event $A$ and $V(\mathcal{C}(x))\cap V(\mathcal{J}_N^0) =\emptyset$ happen, by property of $J_N^0$ (see (\ref{prop_J1})), we have \begin{equation}
		\text{diam}(\mathcal{C}(x))\le 10\sqrt{N}\cdot N_1.
	\end{equation} 
	Therefore, on the event $A$, we have
	\begin{equation}\label{3.14}
		\begin{split}
			&P_1\left[\text{diam}(\mathcal{C}(x))\ge N/10,V(\mathcal{C}(x))\cap V(\mathcal{J}_N^0)=\emptyset \big| \mathcal{F}  \right]\\
			\le &P_1\left[N_1\ge \lfloor\frac{\sqrt{N}}{100}\rfloor,V(\mathcal{C}(x))\cap V(\mathcal{J}_N^0)=\emptyset \bigg| \mathcal{F}  \right].
		\end{split}
	\end{equation}
	
Note that all the $0$-vertex-boxes $B_{x_{i+1}}(L_0)$ (recall step (b) in the random algorithm $\mathfrak{T}_x$) are disjoint to each other by definition of $D_i$. Each time when the condition in (b) holds, $\widetilde{\mathcal{C}}_{i+1}(x)$ will intersect $\mathcal{J}_N^0 $ with at least $c''(L_0,\epsilon)\in (0,1)$ probability (if $\mathcal{B}_{x_i}(2L_0+1)\subset \underline{\mathcal{FI}}^{\epsilon,T}_1$, then $\widetilde{\mathcal{C}}_{i+1}(x)\cap \mathcal{J}_N^0\neq \emptyset$). Hence, 
	\begin{equation}\label{3.15}
		P_1\left[N_1\ge \lfloor\frac{\sqrt{N}}{100}\rfloor,V(\mathcal{C}(x))\cap V(\mathcal{J}_N^0)=\emptyset \bigg| \mathcal{F}  \right]\le (1-c'')^{\lfloor\frac{\sqrt{N}}{100}\rfloor}.
	\end{equation}
	
	By (\ref{3.10}), (\ref{3.11}), (\ref{3.13}), (\ref{3.14}) and (\ref{3.15}), we get (\ref{3.9}) and finish the proof of \textbf{SQ3}.
	\qed

	\section{Proof of SQ4}\label{SQ4}
	
%	\note{Instead write what is the strategy: First prove for restricted model and then...}In this section, we prove \textbf{SQ4} by following the strategy in Section 5 and Section 6 of \cite{duminil2020equality}. 
	In this section, we will first prove the equality of critical values for restricted model $\mathcal{FI}^{u,T}_L$. And then we will prove an important inequality to estimate the influence of remaining part $\mathcal{FI}^{u,T}-\mathcal{FI}^{u,T}_L$ on the probability of an crossing event. Applying these two results, we finally conclude \textbf{SQ4} by contradiction.

	Like in Section \ref{SQ3}, we also fix $d\ge 3$ and $T>0$ in this section.

	\subsection{Equality for Critical Values of $\mathcal{FI}_L^{u,T}$}
	
	We hereby recall the definition of $\mathcal{FI}_L^{u,T}$ and define critical values of this restricted system,
%	Define several critical values of $\mathcal{FI}_L^{u,T}$: for $L\in \mathbb{N}^+$,  
	\begin{enumerate}
		\item $u_*^L=\sup\left\lbrace u>0:P\left[0\xleftrightarrow[]{\mathcal{FI}_{L}^{u,T}}\infty \right]=0  \right\rbrace $; 
		
		\item $u_{**}^L=\sup\left\lbrace u>0:\inf\limits_{R\in \mathbb{N}^+}P\left[B(R)\xleftrightarrow[]{\mathcal{FI}_L^{u,T}}\partial B(2R) \right]=0  \right\rbrace $;

		\item $\widetilde{u}^L=\inf\left\lbrace u>0 :\inf\limits_{R\in \mathbb{N}}R^dP\left[B(\mu(R))\stackrel{\mathcal{FI}_L^{u,T}}{\nleftrightarrow} \partial B(R) \right]=0 \right\rbrace $, where $\mu(R)=\lfloor e^{(\log(R))^{1/3}}\rfloor$. 
		
	\end{enumerate}
	
	Like Proposition 1.3 in \cite{duminil2020equality}, we show that these critical values are equivalent to each other. 
	
	\begin{proposition}\label{ul}
		For any $L\in \mathbb{N}^+$, $u_*^L=u_{**}^L=\widetilde{u}^L$.
	\end{proposition}
	
	Before proving Proposition \ref{ul}, we first state several basic properties of $\mathcal{FI}_L^{u,T}$.
	
	\begin{lemma}\label{properties}
		$\mathcal{FI}_L^{u,T}$ has the following properties:
		\begin{enumerate}
			\item[(a)] \textbf{Lattice Symmetry}: Assume that $\phi$ is either a shift, reflection w.r.t. hyperplanes spanned by the base vectors or rotation by $\frac{k\pi}{2}$ ($k\in \mathbb{Z}$) of $\mathbb{Z}^d$. Then $\phi(\mathcal{FI}_L^{u,T})$ has the same distribution as $\mathcal{FI}_L^{u,T}$. 
			
			\item[(b)] \textbf{Positive Association}: The FKG inequality holds. I.e., if $f$ and $g$ are both increasing functions (or both decreasing functions) on $\{0,1\}^{\mathbb{L}^d}$ such that $E[f^2],E[g^2]<\infty$, then \begin{equation}
					E\left[f(\mathcal{FI}_L^{u,T})g(\mathcal{FI}_L^{u,T}) \right] \ge E\left[f(\mathcal{FI}_L^{u,T})\right]E\left[g(\mathcal{FI}_L^{u,T})\right].
				\end{equation}

			\item[(c)] \textbf{Uniform Finite-energy}: For any $u>0$, there exists $c_{UF}(u)\in (0,1)$ such that for any edge set $\mathcal{K}$, event $E\in \sigma\left(\mathcal{K}\cap \mathcal{FI}_L^{u,T}\right)$ and edge $e\in \mathbb{L}^d\setminus \mathcal{K}$,  
				\begin{equation}\label{UF}
					1-c_{UF}\le 	P\left[e\notin \mathcal{FI}_L^{u,T}\big|E\right] \le c_{UF}. 
				\end{equation}

			\item[(d)] \textbf{Bounded-range i.i.d. Coding}: For $x\in \mathbb{Z}^d$ and $L\in \mathbb{N}^+$, denote
			\begin{equation}
				\mathcal{FI}_L^{u,T}(x)=\sum_{\eta \in \mathcal{FI}^{u,T}}\delta_{\eta}\cdot\mathbbm{1}_{\eta(0)=x,|\eta|\le L}.
			\end{equation}
			Then for any $e\in \mathbb{L}^d$, $\mathbbm{1}_{e\in \mathcal{FI}_L^{u,T}}$ only depends on $\left\lbrace \mathcal{FI}_L^{u,T}(x):d_1(x,e)\le L-1 \right\rbrace$.
			
			\item[(e)] \textbf{Sprikling Property}: For any $u>u'\ge 0$, there exists $\epsilon>0$ s.t. $\mathcal{FI}^{u,T}_L$ stochastically dominates $\mathcal{FI}^{u',T}_L\cup  \zeta_\epsilon$, where $\zeta_\epsilon$ is a Bernoulli bond percolation on $\mathbb{L}^d$ with parameter $\epsilon$. 
			%We denote it by $\mathcal{FI}^{u,T}_L\succ \mathcal{FI}^{u',T}_L\vee \zeta_\epsilon$. 
			
			\item[(f)] \textbf{Uniform Embeddability}: For any $u>0$, there exists $c_{UE}(u)>0$ such that for any $e\in \mathbb{L}^d$, subset $K\subsetneqq \{z\in \mathbb{Z}^d:d(e,\{z\})\le L-1\}$ and configuration $\omega$ of $\left\lbrace \mathcal{FI}_L^{u,T}(x):x\in K \right\rbrace$, 
			\begin{equation}
				P[e\in \mathcal{FI}_L^{u,T}|\omega]>c_{UE}.
			\end{equation}
		\end{enumerate}
		
	\end{lemma}
	
	\begin{proof}[Proof of Lemma \ref{properties}]
		Property (a) and (d) are immediate. For property (b), see Proposition \ref{FKG2}. Since $\mathcal{FI}_1^{u-u'}$ has the same distribution as a Bernoulli bond percolation, property (e) also holds. Hence, it is sufficient to check property (c) and (f).
		
		First consider the first inequality of (\ref{UF}). Note that events $E$ and $\{e=\{x,y\}\notin \mathcal{FI}_L^{u,T}\}$ are both  measurable w.r.t. $\mathscr{A}:=\sum_{(u_i,\eta_i)\in \mathcal{FI}}\delta_{\eta_i}\cdot \mathbbm{1}_{0<u_i\le u,\eta_i\cap (\mathcal{K}\cap \{e\})\neq \emptyset}$. For any configuration $\omega$ of $\mathscr{A}$ such that the event $E\cap \{e=\{x,y\}\notin \mathcal{FI}_L^{u,T}\}$ happens, there exists no path $\eta$ in $\mathcal{FI}_L^{u,T}$ such that $\eta(0)=x$ and $\eta(1)=y$. Thus, let $\eta_0=(x,y)$, then the configuration $\omega+\delta_{\eta_0} \in E\cap \{e\in \mathcal{FI}_L^{u,T}\}$. Note that for $N_{u,T}\sim\text{Pois}(\frac{2du}{T+1})$,
		
		\begin{equation}\label{Pw}
			\frac{P[\omega]}{P[\omega+\delta_{\eta_0}]}=\frac{P[N_{u,T}=0]}{P[N_{u,T}=1]\cdot P_x^{(T)}(\eta_0)}:=c\in (0,1).
		\end{equation}
		By (\ref{Pw}), we have 
		\begin{equation}\label{insert}
			P[E\cap \{e=\{x,y\}\notin \mathcal{FI}_L^{u,T}\}]\le c\cdot P[E\cap \{e=\{x,y\}\in \mathcal{FI}_L^{u,T}\}],
		\end{equation}
		which implies \begin{equation}\label{5.3}
			P[\{e=\{x,y\}\notin \mathcal{FI}_L^{u,T}\}|E]\le \frac{c}{1+c}<1.
		\end{equation}
		
		Now consider the second inequality of (\ref{UF}). For $\mathcal{K}\subset \mathbb{L}^d$ and $e=\{x,y\}$, denote that 
		\begin{equation}
			K_0=\left\lbrace z\in \mathbb{Z}^d:d_1(z,\mathcal{K})\le L-1, d_1(z,e)\le L-1 \right\rbrace.
		\end{equation}
		For any $z\in \mathbb{Z}^d$, since the number of paths with starting point $z$ and length $\le L$ is finite, the number of $\omega$, configurations of $\{\mathcal{FI}_L^{u,T}(z):z\in K_0\}$, such that for any $\eta_1,\eta_2\in \omega,\eta_1\neq \eta_2$ is also finite. We can define 
		\begin{equation}\label{c'Pw}
			c':=\min_{\substack{\omega:\text{configuration of}\ \{\mathcal{FI}_L^{u,T}(z):z\in K_0\}\\ \text{such that for any}\ \eta_1,\eta_2\in \omega,\ \eta_1\neq \eta_2}}P[\omega]\in (0,1). 
		\end{equation}

			We introduce four mappings, $\phi_1$, $\phi_2$, $\phi_3$ and $\phi$ as follows: 
			\begin{enumerate}
				\item[$\phi_1$:] for any path $\eta=(\eta(0),...,\eta(n))\in W^{\left[0,\infty \right) }$, if $\eta$ does not intersect $K_0$, define $\phi_1(\eta)=\emptyset$; otherwise, let $m_0:=\inf \left\lbrace 0\le i\le n: \eta(i)\in K_0 \right\rbrace $ and then define $\phi_1(\eta)=(\eta(m_0),...,\eta(n))$;
				
				\item[$\phi_2$:] for any path $\eta=(\eta(0),...,\eta(n))\in W^{\left[0,\infty \right) }$, if $\eta \cap e=\emptyset$, define $\phi_2(\eta)=\{\eta\}$; otherwise, we denote all integers $i$ such that $\{\eta(i),\eta(i+1)\}=e$ by $\{i_1,...,i_k\}$, where $0\le i_1\le ...\le i_k\le n-1$, and then define \begin{equation}
					\begin{split}
						\phi_2(\eta)=&\left\lbrace (\eta(i_j+1),...,\eta(i_{j+1})):1\le j\le k-1\right\rbrace\\
						 &\cup \left\lbrace (\eta(0),...,\eta(i_1)), (\eta(i_k+1),...,\eta(n))\right\rbrace.
					\end{split}
				\end{equation}

				\item[$\phi_3$:] for any point measure $\omega$ on $W^{\left[0,\infty \right) }$ containing finitely many paths, we denote all the diffierent paths in $\omega$ by $\{\eta_1,...,\eta_{n}\}$. Then we define $\phi_3(\omega)=\sum_{i=1}^{n}\delta_{\eta_i}$.

				\item[$\phi$:] for any point measure $\omega$ on $W^{\left[0,\infty \right) }$ containing finitely many paths, we also define that
				\begin{equation}
					\phi(\omega):=\phi_3\left( \sum_{\eta\in \omega}\sum_{\zeta\in \phi_2(\eta)}\delta_{\phi_1(\zeta)}\right). 
				\end{equation}
			\end{enumerate}
			We claim that for any configuration $\omega$ of $\{\mathcal{FI}_L^{u,T}(z):z\in K_0\}$, $\phi(\omega)$ must be one of the configurations mentioned in (\ref{c'Pw}) and satisfies $\phi(\omega)\cap\mathcal{K}=\omega\cap \mathcal{K}$. In fact, for each path in $\omega$, the mapping $\phi_2$ will cut off all steps traversing the edge $e$ while still visits exactly the same set of edges. Then for each sub-path, the mapping $\phi_1$ only keep the part of it after first intersecting $K_0$. Thus, compared with $\omega$, $\sum_{\eta\in \omega}\sum_{\zeta\in \phi_2(\eta)}\delta_{\phi_1(\zeta)}$ does not traverse $e$ and visits the same edges in $\mathcal{K}$ as $\omega$ does. Finally, the mapping $\phi_3$ eliminates those repetitive paths in $\sum_{\eta\in \omega}\sum_{\zeta\in \phi_2(\eta)}\delta_{\phi_1(\zeta)}$ and maps it to one of the configurations in (\ref{c'Pw}).

			Let $K_1:=\left\lbrace z\in \mathbb{Z}^d:d(z,\mathcal{K})\le L-1\right\rbrace $ and note that $K_0\subsetneqq K_1$. For any configuration $\omega^1$ of $\{\mathcal{FI}_L^{u,T}(z):z\in K_1\}$, let $\omega^1_{|K_0}=\sum_{\eta\in \omega^1}\delta_{\eta}\cdot\mathbbm{1}_{\eta(0)\in K_0}$ and $\widetilde{\omega}^1=\sum_{\eta\in \omega^1}\delta_{\eta}\cdot\mathbbm{1}_{\eta(0)\in K_1\setminus K_0}+\phi(\omega^1_{|K_0})$. By definition of $c'$ in (\ref{c'Pw}),  
			\begin{equation}\label{Pw1}
				\begin{split}
					\frac{P[\omega^1]}{P[\widetilde{\omega}^1]}=\frac{P[\omega^1_{|K_0}]}{P[\phi(\omega^1_{|K_0})]}\le \frac{1}{c'}.
				\end{split}
		\end{equation}
		Note that $\widetilde{\omega}^1\in E\cap \{e\notin \mathcal{FI}_L^{u,T}\}$. By (\ref{Pw1}) we have \begin{equation}
			P[E\cap \{e\in \mathcal{FI}_L^{u,T}\}]\le (c')^{-1} P[E\cap \{e\notin \mathcal{FI}_L^{u,T}\}], 
		\end{equation}
		which implies \begin{equation}\label{57}
			P[e\notin \mathcal{FI}_L^{u,T}|E]\ge \frac{c'}{c'+1}>0.
		\end{equation}
		By (\ref{5.3}) and (\ref{57}), we show that property (c) holds.
		
		For property (f), choose $z_0\in \{z\in \mathbb{Z}^d:d(e,z)\le L-1\}\setminus K$ and a path $\eta$ such that $\eta(0)=z_0$, $e\in\eta$ and $|\eta|\le L$. Since the event $\{\eta\in \mathcal{FI}^{u,T}_L\}$ is measurable w.r.t. $\mathcal{FI}^{u,T}_L(z_0)$ and $z_0\notin K$, we have \begin{equation}
			P\left[ e\in \mathcal{FI}_L^{u,T}\big|\omega\right] \ge P\left[ \eta\in \mathcal{FI}^{u,T}_L \big|\omega\right] =P\left[ \eta\in \mathcal{FI}^{u,T}_L\right] >c''\in(0,1). 
		\end{equation}
	\end{proof}

	We will prove Proposition \ref{ul} in the following two steps: 
	\begin{enumerate}
		\item[\textbf{Step 1}:] $\widetilde{u}^L\le u_*^L$;
		
		\item[\textbf{Step 2}:] $u_*^L=u_{**}^L\le \widetilde{u}^L$.
	\end{enumerate}

	For \textbf{Step 1}, we hereby cite a result (see (6.7) in \cite{duminil2020equality}), which shows that for a specific class of models ($\mathcal{FI}_L^{u,T}$ is one of them), there exists an infinite cluster inside a sufficiently thick slab in the supercritical phase. This is a version of Grimmett, Marstrand's theorem \cite{grimmett1990supercritical} for this class of models. 

	%{\color{blue}Similar to} (6.7) in \cite{duminil2020equality}, we have:

	\begin{lemma}\label{slab}
		Assume that $\mathcal{H}^u$ is a random subset of $\mathbb{L}^d$ which is increasing w.r.t. parameter $u$ and satisfies property (a)-(f) in Lemma \ref{properties}. Define $u_*(\mathcal{H}):=\sup\left\lbrace u>0:P\left[0\xleftrightarrow[]{\mathcal{H}^u} \infty \right]=0  \right\rbrace $. Then for any $u>u_*(\mathcal{H})$, there exists an integer $M>0$ such that \begin{equation}
			P\left[0\xleftrightarrow[\mathbb{S}(M)]{\mathcal{H}^u} \infty \right]>0, 
		\end{equation}
		where $\mathbb{S}(M):=\mathbb{Z}^2\times \{-M,...,-1,0,1,...,M\}^{d-2}$. 
	\end{lemma}

%The proof of Lemma \ref{slab} is parallel to Theorem A in \cite{grimmett1990supercritical} and we leave it in Appendix \ref{appendix_slab}. 	{\color{blue}We also want to specify that Lemma \ref{slab} is already stated in (6.7) of \cite{duminil2020equality} without proof.} 
	
	Note that $\mathcal{FI}_L^{u,T}$ satisfies all the requirements for $\mathcal{H}^u$. By Lemma \ref{slab}, for any $u>u^L_*$, there exists $M(u)\in \mathbb{N}^+$ such that 
	\begin{equation}\label{cPHu}
		c:=P\left[0\xleftrightarrow[\mathbb{S}(M)]{\mathcal{FI}_L^{u,T}} \infty \right]\in (0,1). 
	\end{equation}

	For any integer $R$ s.t. $\mu(R)>100(M+L)$ and integer $0\le i\le \left\lfloor\frac{\mu(R)-M}{5(M+L)}\right\rfloor$, let $x_i=(0,0,...,0,5(M+L)i)\in \mathbb{Z}^d$. Note that for any $0\le i<j\le \left\lfloor\frac{\mu(R)-M}{5(M+L)}\right\rfloor$, the distance between slabs $x_i+\mathbb{S}(M)$ and $x_j+\mathbb{S}(M)$ is greater than $2L$. Thus, the events $\left\lbrace x_i \xleftrightarrow[x_i+\mathbb{S}(M)]{\mathcal{FI}^{u,T}_L} \infty\right\rbrace^c $, for $0\le i\le \lfloor\frac{\mu(R)-M}{5(M+L)} \rfloor$ are independent. Therefore, by (\ref{cPHu}), for any $u>u^L_*$ and ingeter $R>0$ s.t. $\mu(R)>100(M+L)$, we have 
		\begin{equation}
			\begin{split}
				P\left[B(\mu(R)) \overset{\mathcal{FI}^{u,T}_L}{\not \leftrightarrow}B(R)  \right]\le & P\left[\bigcap_{i=0}^{\lfloor\frac{\mu(R)-M}{5(M+L)} \rfloor} \left\lbrace x_i \xleftrightarrow[x_i+\mathbb{S}(M)]{\mathcal{FI}^{u,T}_L} \infty \right\rbrace^c  \right]\\
				=&\prod_{i=1}^{\lfloor\frac{\mu(R)-M}{5(M+L)} \rfloor} P\left[\left\lbrace x_i \xleftrightarrow[x_i+\mathbb{S}(M)]{\mathcal{FI}^{u,T}_L} \infty \right\rbrace^c \right]=(1-c)^{\lfloor\frac{\mu(R)-M}{5(M+L)} \rfloor}
			\end{split}
		\end{equation}
		which implies that $u\ge \widetilde{u}^L$. In conclusion, $\widetilde{u}^L\le u_*^L$.

	For \textbf{Step 2}, it is sufficient to prove: for any $u<u_*^L$, there exists $c'(u)>0$ s.t. for any $R\in \mathbb{N}^+$, 
	\begin{equation}\label{5.4}
		P\left[0\xleftrightarrow[]{\mathcal{FI}^{u,T}_L} \partial B(R) \right]\le e^{-c'R}. 
	\end{equation}
In fact, by 
	$$P[B(R)\xleftrightarrow[]{\mathcal{FI}^{u,T}_L} \partial B(2R) ]\le \sum_{x\in \partial B(R)}P[x\xleftrightarrow[]{\mathcal{FI}^{u,T}_L} \partial B_x(R) ],$$  
	$$P[B(\mu(R))\xleftrightarrow[]{\mathcal{FI}^{u,T}_L} \partial B(R) ]\le \sum_{x\in \partial B(\mu(R))}P[x\xleftrightarrow[]{\mathcal{FI}^{u,T}_L} \partial B_x(0.5R) ],$$ and (\ref{5.4}), we have $u_*^L\le u_{**}^L$ and $u_*^L\le\widetilde{u}^L $. Using the same arguments as in the proof of \textbf{SQ1}, we also have $ u_*^L\ge u_{**}^L$. Hence, (\ref{5.4}) implies $u_*^L=u_{**}^L\le \widetilde{u}^L$.

	For any $u<u_*^L$, define event $\gamma_\epsilon(u)=\mathcal{FI}^{u,T}_L\cup \zeta_\epsilon$ (recalling that $\zeta_\epsilon$ is a Bernoulli bond percolation with parameter $\epsilon$). Let $\epsilon_c(u)\in [0,1]$ be the critical value of the model $\left\lbrace \gamma_\epsilon:\epsilon\in [0,1] \right\rbrace$. Precisely, 
	\begin{equation}\label{epsilonc}
		\epsilon_c(u):=\sup\left\lbrace \epsilon\in [0,1]:P\left[0\xleftrightarrow[]{\gamma_\epsilon}\infty \right]=0  \right\rbrace.
	\end{equation}
	For $u<u_*^L$, let $u'=\frac{u+u_*^L}{2}$. By property (e), there exists $\epsilon'>0$ such that $\mathcal{FI}_L^{u',T}$ stochastically dominates $\mathcal{FI}_L^{u,T}\cup \zeta_{\epsilon'} $. Since $\mathcal{FI}^{u',T}$ does not percolate, we have $\epsilon_c(u)\ge\epsilon'>0 $ for all $u<u_*^L$.

	In order to get (\ref{5.4}), it is sufficient to prove: for any $0\le \epsilon <\epsilon_c $, there exists $c''(u,\epsilon)>0$ such that for all integer $R\ge 1$,  
	\begin{equation}\label{thetaR}
		\theta_R(\epsilon):=P\left[0\xleftrightarrow[]{\gamma_\epsilon} \partial B(R) \right]\le e^{-c''R}. 
	\end{equation}
In fact, we prove here a moderately stronger result since (\ref{5.4}) is the special case of (\ref{thetaR}) when $\epsilon=0$.

	We fix $u<u_*^L$ and define $\xi_{R}:=\{0\xleftrightarrow[]{\gamma_\epsilon(u)}\partial B(R)\}$. We denote $\mathscr{V}:=\left\lbrace \mathcal{FI}_L^{u,T}(x):x\in B(R+L) \right\rbrace$ and   $\mathscr{E}:=\left\lbrace \zeta_\epsilon(e):e\in \mathcal{B}(R)\right\rbrace $. Note that $\mathbbm{1}_{\xi_{R}}$ is a function of $\mathscr{V}$ and $\mathscr{E}$.
	For each $x\in \mathbb{Z}^d$ and $e\in \mathbb{L}^d$, we also write $\mathscr{V}_x:=\mathcal{FI}_L^{u,T}(x)$ and $\mathscr{E}_e:=\zeta_\epsilon(e)$.

		Like in \cite{duminil2019sharp} and \cite{o2005every}, we define the following probabilities to discribe the influence when the  configurations of $\mathscr{V}$(or $\mathscr{E}$) at each single vertex (or edge) is resampled: 
		\begin{itemize}
			\item Assume that $\widetilde{\mathscr{V}}$ is an independent copy of $\mathscr{V}$. For any $x\in \mathbb{Z}^d$, we denote that $\widehat{\mathscr{V}}_{|x}:=\left\lbrace \mathscr{V}_y:y\in \mathbb{Z}^d\setminus \{x\}\right\rbrace\cup \{\widetilde{\mathscr{V}}_x\}$ and then define \begin{equation}
				\boldsymbol{Inf}_{\mathscr{V}_x}:=P\left[\mathbbm{1}_{\xi_{R}(\mathscr{V},\mathscr{E})}\neq \mathbbm{1}_{\xi_{R}(\widehat{\mathscr{V}}_{|x},\mathscr{E})}\right]. 
			\end{equation}

			\item Similarly, we write $\widetilde{\mathscr{E}}$ for an independent copy of $\mathscr{E}$. For each $e\in \mathbb{L}^d$, we also denote $\widehat{\mathscr{E}}_{|e}:=\left\lbrace \mathscr{E}_{e'}:e'\in \mathbb{L}^d\setminus \{e'\}\right\rbrace\cup \{\widetilde{\mathscr{E}}_e\}$ and define \begin{equation}
				\boldsymbol{Inf}_{\mathscr{E}(e)}:=P\left[\mathbbm{1}_{\xi_{R}(\mathscr{V},\mathscr{E})}\neq \mathbbm{1}_{\xi_{R}(\mathscr{V},\widehat{\mathscr{E}}_{|e})}\right]. 
			\end{equation}	
		\end{itemize}
	\begin{remark}
			Although the notation ``$\boldsymbol{Inf}$'' (abbreviation of the word ``influence'') may cause confusion with infimum ``$inf$'', we keep this notation in order to be consistent with previous works such as \cite{duminil2020equality,duminil2019sharp,o2005every}. 
	\end{remark}

	\begin{lemma}\label{thetaR'}
		For any $R\in \mathbb{N}^+$ and  $\epsilon\in (0,1)$,
		\begin{equation}\label{5.6}
			\frac{d \theta_R}{d \epsilon}\ge \alpha(\epsilon)\cdot\left(\sum_{x\in B(R+L)}\boldsymbol{Inf}_{\mathscr{V}(x)}+\sum_{e\in \mathcal{B}(R)}\boldsymbol{Inf}_{\mathscr{E}(e)}\right),
		\end{equation}
		where $\alpha(\epsilon)=0.5\cdot\min\left\{\left(\epsilon(1-\epsilon)(1-c_{UF})\right)^{c_5(d)L^d}(2L+2)^{-d},1\right\}$.	
	\end{lemma}
	
	\begin{proof}[Proof of Lemma \ref{thetaR'}]
		Using the standard Russo's formula (see Section 2.4 in \cite{grimmett1999percolation}) for the Bernoulli percolation, we have 
		\begin{equation}\label{5.7}
			\frac{d\theta_R}{d\epsilon}=\sum_{e\in \mathcal{B}(R)}P\left[\boldsymbol{Piv}_e(\xi_{R},\gamma_\epsilon),e\notin \mathcal{FI}^{u,T}_L \right],
		\end{equation}
		where $\boldsymbol{Piv}_e(\xi_{R},\gamma_\epsilon):=\{0\xleftrightarrow[]{\gamma_\epsilon\cup \{e\}} \partial B(R) , 0\xleftrightarrow[]{\gamma_\epsilon\setminus \{e\}} \partial B(R)\}  $. By definition, one may immediately have $P\left[\boldsymbol{Piv}_e(\xi_{R},\gamma_\epsilon),e\notin \mathcal{FI}^{u,T}_L \right]\ge \boldsymbol{Inf}_{\mathscr{E}(e)}$. Hence, \begin{equation}\label{5.8}
			\frac{d\theta_R}{d\epsilon}\ge \sum_{e\in \mathcal{B}(R)}\boldsymbol{Inf}_{\mathscr{E}(e)}. 
		\end{equation}
		
		Define that  $\boldsymbol{Piv}_{\mathcal{B}_x(L)}(\xi_{R},\gamma_\epsilon):=\{0\xleftrightarrow[]{\gamma_\epsilon\cup \mathcal{B}_x(L)} \partial B(R) , 0\xleftrightarrow[]{\gamma_\epsilon\setminus \mathcal{B}_x(L)} \partial B(R)\}$. Noting that $\mathcal{FI}^{u,T}_L(x)$ only influences the values of $\mathbbm{1}_{e\in \mathcal{FI}^{u,T}_L}$ for $e\in \mathcal{B}_x(L)$, we have: for any $x\in \mathbb{Z}^d$, 
		\begin{equation}\label{5.9}
			P\left[\boldsymbol{Piv}_{\mathcal{B}_x(L)}(\xi_{R},\gamma_\epsilon)\right] \ge \boldsymbol{Inf}_{\mathscr{V}_x}.  
		\end{equation}
		By property (c), we have 
			\begin{equation}\label{5.10}
				P\left[\boldsymbol{Piv}_{\mathcal{B}_x(L)}(\xi_{R},\gamma_\epsilon) \right]\le (1-c_{UF})^{-cL^d}P\left[\boldsymbol{Piv}_{\mathcal{B}_x(L)}(\xi_{R},\gamma_\epsilon), \mathcal{B}_x(L)\cap \mathcal{FI}^{u,T}_L=\emptyset\right].    
		\end{equation} 
		Note that the event on the RHS of (\ref{5.10}) is measurable w.r.t. $\sigma(\mathscr{V},\left\lbrace \mathscr{E}_e:e\in \mathcal{B}(R)\setminus \mathcal{B}_x(L) \right\rbrace )$. 
		
	For any $(\mathscr{V}^0,\hat{\mathscr{E}}^0)$, configuration of $(\mathscr{V},\left\lbrace \mathscr{E}_e:e\in \mathcal{B}(R)\setminus \mathcal{B}_x(L)\right\rbrace)$ such that the event in (\ref{5.10}), $\left\lbrace \boldsymbol{Piv}_{\mathcal{B}_x(L)}(\xi_{R},\gamma_\epsilon), \mathcal{B}_x(L)\cap \mathcal{FI}^{u,T}_L=\emptyset\right\rbrace $ happens, we can label edges in $\mathcal{B}_x(L)\cap \mathcal{B}(R)$ in an arbitrary but deterministic way, and open them one by one in this given order. Then the event $\xi_{R}$ will occur for the first time at a certain step. At the first time $\xi_R$ happens, suppose the edges opened are $(e_1,e_2,...,e_m)$ ($m\ge 1$). Let $\mathscr{E}^0(\mathscr{V}^0,\hat{\mathscr{E}}^0)$ be the configuration of $\left\lbrace \mathscr{E}_e:e\in \mathcal{B}(R)\right\rbrace$ such that for all $e\in \mathcal{B}(R)\setminus \mathcal{B}_x(L)$, $\mathscr{E}^0_e=\hat{\mathscr{E}}^0_e$ and for all $e\in \mathcal{B}(R)\cap \mathcal{B}_x(L)$, $\mathscr{E}^0_e=\mathbbm{1}_{e\in \{e_1,e_2,...,e_{m-1}\}}$. Then we have 

\begin{equation}\label{omega0}
				\left(\mathscr{V}^0,\mathscr{E}^0\right)\in \left\lbrace \boldsymbol{Piv}_{e_m}(\xi_{R},\gamma_\epsilon),e_m\notin \mathcal{FI}^{u,T}_L \right\rbrace. 
		\end{equation}
		Noting that $e_m\in \mathcal{B}_x(L)\cap \mathcal{B}(R)$, $m\le |\mathcal{B}_x(L)|\le cL^d$, (\ref{omega0}) and by property (c), we have 
		\begin{equation}\label{5.11}
			\begin{split}
				&\left(\epsilon(1-\epsilon)\right)^{cL^d}P\left[ \boldsymbol{Piv}_{\mathcal{B}_x(L)}(\xi_{R},\gamma_\epsilon), \mathcal{B}_x(L)\cap \mathcal{FI}^{u,T}_L=\emptyset\right] \\
				=&\sum_{(\mathscr{V}^0,\hat{\mathscr{E}}^0)\ s.t.\   \boldsymbol{Piv}_{\mathcal{B}_x(L)}(\xi_{R},\gamma_\epsilon), \mathcal{B}_x(L)\cap \mathcal{FI}^{u,T}_L=\emptyset }\left(\epsilon(1-\epsilon)\right)^{cL^d}P\left[\mathscr{V}=\mathscr{V}^0,\left\lbrace \mathscr{E}_e:e\in \mathcal{B}(R)\setminus \mathcal{B}_x(L)\right\rbrace=\hat{\mathscr{E}}^0\right]\\
				\le &\sum_{(\mathscr{V}^0,\hat{\mathscr{E}}^0)\ s.t.\   \boldsymbol{Piv}_{\mathcal{B}_x(L)}(\xi_{R},\gamma_\epsilon), \mathcal{B}_x(L)\cap \mathcal{FI}^{u,T}_L=\emptyset }P\left[\mathscr{V}=\mathscr{V}^0,\left\lbrace \mathscr{E}_e:e\in \mathcal{B}(R)\setminus \mathcal{B}_x(L)\right\rbrace=\hat{\mathscr{E}}^0,\mathscr{E}=\mathscr{E}^0\right]\\
				\le &\sum_{e\in \mathcal{B}_x(L)\cap \mathcal{B}(R)}P\left[\boldsymbol{Piv}_{e}(\xi_{R},\gamma_\epsilon),e\notin \mathcal{FI}^{u,T}_L \right],  
			\end{split}
		\end{equation}
		where the last inequality of (\ref{5.11}) is based on the observation that all events in the summation, $\left\lbrace \mathscr{V}=\mathscr{V}^0,\left\lbrace \mathscr{E}_e:e\in \mathcal{B}(R)\setminus \mathcal{B}_x(L)\right\rbrace=\hat{\mathscr{E}}^0,\mathscr{E}=\mathscr{E}^0\right\rbrace $ are contained in $\left\lbrace \boldsymbol{Piv}_{e_m}(\xi_{R},\gamma_\epsilon),e_m\notin \mathcal{FI}^{u,T}_L \right\rbrace$ and disjoint to each other.

		Combining (\ref{5.9}),(\ref{5.10}) and (\ref{5.11}),
			\begin{equation}\label{5.12}
				\sum_{x\in B(R+L)}\boldsymbol{Inf}_{\mathscr{V}_x}\le \left[\epsilon(1-\epsilon)(1-c_{UF}) \right]^{-cL^d} \sum_{x\in B(R+L)}\sum_{e\in \mathcal{B}_x(L)\cap \mathcal{B}(R)}P\left[\boldsymbol{Piv}_{e}(\xi_{R},\gamma_\epsilon),e\notin \mathcal{FI}^{u,T}_L \right].
		\end{equation}
		For any $e\in \mathcal{B}(R)$, there exist at most $(2L+2)^d$ points $x$ such that $e\in \mathcal{B}_x(L)$. Thus for each edge $e$ in the inner summation of (\ref{5.12}), it can only be summed for at most $(2L+2)^d$ times. Hence, by (\ref{5.7}) and (\ref{5.12}) we have  
		\begin{equation}\label{cef}
			\begin{split}
				\sum_{x\in B(R+L)}\boldsymbol{Inf}_{\mathscr{V}_x}\le &\left[\epsilon(1-\epsilon)(1-c_{UF}) \right]^{-cL^d}(2L+2)^d\cdot\sum_{e\in \mathcal{B}(R)}P\left[\boldsymbol{Piv}_{e}(\xi_{R},\gamma_\epsilon),e\notin \mathcal{FI}^{u,T}_L \right]\\
				=&\left[\epsilon(1-\epsilon)(1-c_{UF}) \right]^{-cL^d}(2L+2)^d\cdot\frac{d \theta_R}{d \epsilon}. 
			\end{split}
		\end{equation}
		
		By (\ref{5.8}) and (\ref{cef}), (\ref{5.6}) follows.
	\end{proof}
	
	Inspired by Definition 6.3 in \cite{duminil2020equality} and Section 3 in \cite{duminil2019sharp}, we want to introduce the following randomized algorithm \textbf{T}. This algorithm provides an approach to sample the random variable $\mathbbm{1}_{\xi_j} $, where $j\sim U\{1,R\}$, $R\in \mathbb{N}^+$ (i.e. for any $1\le i\le R$, $P[j=i]=\frac{1}{R}$).
	
	\begin{definition}[Algorithm \textbf{T}]\label{AlgorithmT}
		First, uniformly choose $j\in \{1,2,...,R\}$. Initially, set $\mathcal{E}_0=\mathcal{B}(j)$, $P_0=\emptyset$ and $\mathcal{C}_0=\emptyset$. Then	we construct the algorithm \textbf{T} by induction. 
		
		For $t\ge 1$, assume that we already have $\mathcal{E}_{t-1}$, $P_{t-1}$ and $\mathcal{C}_{t-1}$,

		\begin{itemize}
			\item If $\partial_e^{out}\mathcal{C}_{t-1} \cap \mathcal{E}_{t-1}\neq \emptyset$ (we set $\partial_e^{out}\emptyset=\mathbb{L}^d$ and denote the lexicographically-smallest edge in $\partial_e^{out}\mathcal{C}_{t-1}\cap \mathcal{E}_{t-1}$ by $e_t$) and $\mathscr{E}_{e_t}$ has not been sampled, there are two different cases:
			\begin{itemize}
				\item If $e_t\in \cup_{y\in  P_{t-1}}\mathscr{V}_{y}$, let $\mathcal{E}_{t}=\mathcal{E}_{t-1}\setminus \{e_t\}$, $P_t=P_{t-1}$ and $\mathcal{C}_{t}=\mathcal{C}_{t-1}\cup \{e_t\}$.
				
				\item If $e_t\notin \cup_{y\in  P_{t-1}}\mathscr{V}_{y}$, sample $\mathscr{E}_{e_t}$. Then if $\mathscr{E}_{e_t}=1$, let $\mathcal{E}_{t}=\mathcal{E}_{t-1}\setminus \{e_t\}$, $P_t=P_{t-1}$ and $\mathcal{C}_{t}=\mathcal{C}_{t-1}\cup \{e_t\}$; if $\mathscr{E}_{e_t}=0$ and $\left\lbrace y:d(y,e_t)\le L-1  \right\rbrace\setminus P_{t-1}=\emptyset$, let $\mathcal{E}_{t}=\mathcal{E}_{t-1}\setminus \{e_t\}$, $P_t=P_{t-1}$ and $\mathcal{C}_{t}=\mathcal{C}_{t-1}$; in the other cases, let $\mathcal{E}_{t}=\mathcal{E}_{t-1}$, $P_t=P_{t-1}$ and $\mathcal{C}_{t}=\mathcal{C}_{t-1}$.
				
			\end{itemize}
			
			\item If $\partial_e^{out}\mathcal{C}_{t-1}\cap \mathcal{E}_{t-1}\neq \emptyset$ and $\mathscr{E}_{e_t}$ has been sampled, there are also two different cases:
			\begin{itemize}
				\item If $\left\lbrace y:d(y,e_t)\le L-1  \right\rbrace\setminus P_{t-1}$ contains more than one elements, we denote the lexicographically-smallest point in it by $x_t$ and sample $\mathscr{V}_{x_t}$. If $e_t\in \mathscr{V}_{x_t}$, let $\mathcal{E}_{t}=\mathcal{E}_{t-1}\setminus \{e_t\}$, $P_t=P_{t-1}\cup \{x_t\}$ and $\mathcal{C}_{t}=\mathcal{C}_{t-1}\cup \{e_t\}$; otherwise, let $\mathcal{E}_{t}=\mathcal{E}_{t-1}$, $P_t=P_{t-1}\cup \{x_t\}$ and $\mathcal{C}_{t}=\mathcal{C}_{t-1}$.
				
				\item If $\left\lbrace y:d(y,e_t)\le L-1  \right\rbrace\setminus P_{t-1}$ contains only one point $x_t$, sample $\mathscr{V}_{x_t}$. If $e_t\in \mathscr{V}_{x_t}$, $\mathcal{E}_{t}=\mathcal{E}_{t-1}\setminus \{e_t\}$, $P_t=P_{t-1}\cup \{x_t\}$ and $\mathcal{C}_{t}=\mathcal{C}_{t-1}\cup \{e_t\}$; otherwise, let $\mathcal{E}_{t}=\mathcal{E}_{t-1}\setminus \{e_t\}$, $P_t=P_{t-1}\cup \{x_t\}$ and $\mathcal{C}_{t}=\mathcal{C}_{t-1}$.
			\end{itemize}
			
			\item If $\partial_e^{out}\mathcal{C}_{t-1}\cap \mathcal{E}_{t-1}= \emptyset$, stop the algorithm. 
			
		\end{itemize}
		
	\end{definition}
	
	Let $\rho_{\mathscr{V}}(x):=P\left[\mathscr{V}_x\ is\ sampled\ in\ \textbf{T}\right] $ and $\rho_{\mathscr{E}}(e):=P\left[\mathscr{E}_e\ is\ sampled\ in\ \textbf{T}\right] $. By OSSS inequality (see Theorem 3.1 in \cite{o2005every}), we have: for any $R\in \mathbb{N}^+$,
	\begin{equation}\label{5.13}
		Var[\mathbbm{1}_{\xi_{R}}]=\theta_R(1-\theta_R)\le \sum_{x\in B(R+L)}\rho_{\mathscr{V}}(x)\cdot \boldsymbol{Inf}_{\mathscr{V}_x}+\sum_{e\in \mathcal{B}(R)}\rho_{\mathscr{E}}(e)\cdot\boldsymbol{Inf}_{\mathscr{E}_e}.  
	\end{equation}

	\begin{lemma}\label{lemmathetaR}
		For an integer $R\ge 1$, we denote $\Sigma_R=\sum_{r=0}^{R-1}\theta_r$ (for convenience, we set $\theta_0=1$) and let $\beta(\epsilon)=\alpha(\epsilon)\cdot\left[2(2L+1)^d\right] ^{-1}$ (recalling the function $\alpha(\epsilon)$ in Lemma \ref{thetaR'}). Then for any $R\in \mathbb{N}^+$ and $\epsilon\in (0,1)$, 
		\begin{equation}\label{5.14}
			\frac{d \theta_R}{d \epsilon}\ge \beta\cdot\frac{R}{\Sigma_R}\cdot\theta_R (1-\theta_R).
		\end{equation}
	\end{lemma}
	
	\begin{proof}[Proof of Lemma \ref{lemmathetaR}]
			We first claim that if $\mathscr{V}_x$ is sampled in the Algorithm \textbf{T}, then the event $\{\partial B(j_0) \xleftrightarrow[]{\gamma_\epsilon} B_x(L)\cap B(R)\}$ must happen. In fact, if $j=j_0$ and $\mathscr{V}_x$ is sampled in the $k$-th step for some $k\in \mathbb{N}^+$, by Definition \ref{AlgorithmT} we have $d_1(e_k,\{x\})\le L-1$ (recalling that $e_k$ is the lexicographically-smallest edge in $\partial_e^{out}\mathcal{C}_{k-1}\cap \mathcal{E}_{k-1}$). If $k=1$, recalling that $\mathcal{C}_{k-1}=\emptyset$ and $\mathcal{E}_{k-1}=\mathcal{B}(j_0)$, we have $B_x(L)\cap B(j_0)\neq \emptyset $ and thus the event $\{\partial B(j_0) \xleftrightarrow[]{\gamma_\epsilon} B_x(L)\cap B(R)\}$ happens. If $k\ge 2$, since the edge set $\mathcal{C}_{k-1}$ is contained in $\gamma_\epsilon\cap \mathcal{B}(j_0)$ and intersects $\partial B(j_0)$, we have that $d_1(\gamma_\epsilon\cap \mathcal{B}(j_0),\{x\})\le L$ and thus $\{\partial B(j_0) \xleftrightarrow[]{\gamma_\epsilon} B_x(L)\cap B(R)\}$ also happens.  Hence, we have \begin{equation}\label{5.15}
			\rho_{\mathscr{V}}(x)\le \frac{1}{R}\sum_{j=1}^{R}P\left[ \partial B(j) \xleftrightarrow[]{\gamma_\epsilon} B_x(L)\cap B(R)\right]\le \frac{1}{R}\sum_{j=1}^{R}\sum_{y\in B_x(L)\cap B(R)}P\left[y \xleftrightarrow[]{\gamma_\epsilon}\partial B(j)\right].  
		\end{equation}
		For any $y\in B(R)$, if $y\in \partial B(l)$, $l\in [0,R]$, then for all $1\le j \le R$, $P\left[y \xleftrightarrow[]{\gamma_\epsilon}\partial B(j)\right]\le \theta_{|j-l|}$. Therefore, we have
		\begin{equation}\label{Rj1}
			\begin{split}
				&\sum_{j=1}^{R}\sum_{y\in B_x(L)\cap B(R)}P\left[y \xleftrightarrow[]{\gamma_\epsilon}\partial B(j)\right]\\
				=& \sum_{y\in B_x(L)\cap B(R)}\left(\sum_{j=1}^{l}P\left[y \xleftrightarrow[]{\gamma_\epsilon}\partial B(j)\right]+\sum_{j=l+1}^{R}P\left[y \xleftrightarrow[]{\gamma_\epsilon}\partial B(j)\right]\right) \\
				\le &\sum_{y\in B_x(L)\cap B(R)}\left(\sum_{j=1}^{l}\theta_{|j-l|}+\sum_{j=l+1}^{R}\theta_{|j-l|}\right) \\
				\le  &\sum_{y\in B_x(L)\cap B(R)} 2\Sigma_R\le 2(2L+1)^d\cdot \Sigma_R.
			\end{split}
		\end{equation}
		Combine (\ref{5.15}) and (\ref{Rj1}), 
		\begin{equation}\label{5.17}
			\rho_{\mathscr{V}}(x)\le  2(2L+1)^d\cdot\frac{\Sigma_R}{R}.
		\end{equation}

		Similarly, for any $e=\{x,y\}\in \mathcal{B}(R)$ we have 
		\begin{equation}\label{5.18}
			\rho_{\mathscr{E}}(e)\le \frac{1}{R}\sum_{j=1}^{R}\left( P\left[x\xleftrightarrow[]{\gamma_\epsilon}\partial B(j) \right]+P\left[y\xleftrightarrow[]{\gamma_\epsilon}\partial B(j) \right]\right)\le 4\cdot\frac{\Sigma_R}{R}\le 2(2L+1)^d\cdot\frac{\Sigma_R}{R}.
		\end{equation}
		
		By (\ref{5.6}), (\ref{5.13}), (\ref{5.15}), (\ref{5.17}) and (\ref{5.18}), we have 
		\begin{equation}\label{5.19}
			\begin{split}
				\theta_R (1-\theta_R)\le 2(2L+1)^d\cdot\frac{\Sigma_R}{R}\cdot\left[\sum_{x\in B(R+L)}\boldsymbol{Inf}_{\mathscr{V}_x}+\sum_{e\in \mathcal{B}(R)}\boldsymbol{Inf}_{\mathscr{E}_e}\right]\le  \beta^{-1}\cdot\frac{\Sigma_R}{R}\cdot\frac{d \theta_R}{d \epsilon}. 
			\end{split}
		\end{equation} 
		From (\ref{5.19}), we finally get (\ref{5.14}). 
	\end{proof}

	Now we are able to prove (\ref{thetaR}) by adapting the proof of Lemma 3.1 in \cite{duminil2019sharp}.  
	
	\begin{proof}[Proof of (\ref{thetaR})]
		Let $\epsilon_1=\inf\left\lbrace \epsilon\in [0,1]:\limsup\limits_{R\to \infty}\frac{\log(\Sigma_R)}{\log(R)}\ge 1\right\rbrace $. Noting that when $\epsilon>p_c(d)$ ($p_c(d)\in (0,1)$ is the critial parameter of Bernoulli bond percolation on $\mathbb{L}^d$), for any integer $n\ge 0$, we have 
			\begin{equation}\label{limtheta}
				\theta_n \ge P\left[0\xleftrightarrow[]{\zeta_\epsilon}\partial B(n) \right]\ge P\left[0\xleftrightarrow[]{\zeta_\epsilon}\infty \right]:=c(\epsilon)>0. 
			\end{equation}
			By (\ref{limtheta}),
			\begin{equation}
				\limsup\limits_{R\to \infty}\frac{\log(\Sigma_R)}{\log(R)}\ge\limsup\limits_{R\to \infty} \frac{\log(c(\epsilon)R)}{\log(R)}=1, 
			\end{equation}
			which implies that $\epsilon_1\le p_c(d)<1$.

		We then show that for any $\epsilon\in (\epsilon_1,1)$, \begin{equation}\label{5.25}
			\theta_\infty(\epsilon):=\lim\limits_{R\to \infty}\theta_R(\epsilon)>0.
		\end{equation}
		
		In order to get (\ref{5.25}), we first note the following fact: for $T_n(\epsilon)=\frac{1}{\log(n)}\sum_{i=1}^{n}\frac{\theta_i}{i}$, then \begin{equation}\label{5.26}
			\lim\limits_{n\to \infty} T_n(\epsilon)= \theta_\infty(\epsilon). 
		\end{equation} 
		To verify (\ref{5.26}), consider $\widetilde{T}_n=\frac{\sum_{i=1}^{n}\frac{\theta_i}{i}}{\sum_{i=1}^{n}\frac{1}{i}} $. Since $\lim\limits_{i\to \infty}\theta_i=\theta_\infty$, for any $\delta>0$, there exists some $M>0$ such that for all $i\ge M$, $|\theta_i-\theta_\infty|<0.5\delta$. Hence, 
		\begin{equation}
			\begin{split}
				|\widetilde{T}_n-\theta_\infty|\le& \left| \frac{\sum_{i=1}^{M-1}\frac{1}{i}(\theta_i-\theta_\infty)}{\sum_{i=1}^{n}\frac{1}{i}} \right|+\left|\frac{\sum_{i=M}^{n}\frac{1}{i}(\theta_i-\theta_\infty)}{\sum_{i=1}^{n}\frac{1}{i}} \right|\\
				\le & \frac{\sum_{i=1}^{M-1}\frac{1}{i}}{\sum_{i=1}^{n}\frac{1}{i}}+0.5\delta.
			\end{split}
		\end{equation}
	Note that $\frac{\sum_{i=1}^{M-1}\frac{1}{i}}{\sum_{i=1}^{n}\frac{1}{i}}<0.5\delta$ holds for all sufficiently large $n$. So that we have $\lim\limits_{n\to \infty}\widetilde{T}_n=\theta_\infty$ and then (\ref{5.26}) follows because of $\lim\limits_{n\to \infty} \frac{\sum_{i=1}^{n}\frac{1}{i}}{\log(n)}=1$.

	For convenience, we write $\frac{d \theta_i}{d \epsilon}$ as $\theta_i'$. By Lemma \ref{lemmathetaR} and $\theta_i\le \theta_1$,
		\begin{equation}\label{thetaR1}
			\frac{\theta_i'}{\theta_i}\ge \beta (1-\theta_i)\cdot\frac{i}{\Sigma_i}\ge \beta (1-\theta_1)\cdot\frac{i}{\Sigma_i}:=\widetilde{\beta}\cdot\frac{i}{\Sigma_i}. 
		\end{equation}
		By (\ref{thetaR1}), we have 
		\begin{equation}\label{5.29}
			T_n'(\epsilon)=\frac{1}{\log(n)}\sum_{i=1}^{n}\frac{\theta_i'}{i}\ge \frac{\widetilde{\beta}}{\log(n)}\sum_{i=1}^{n}\frac{\theta_i}{\Sigma_i}.
		\end{equation}  
		Since $\theta_i=\Sigma_{i+1}- \Sigma_{i}\ge 0$ for all $i\in \mathbb{N}^+$, \begin{equation}\label{5.30}
			\frac{\theta_i}{\Sigma_i}\ge \int_{\Sigma_{i}}^{\Sigma_{i+1}}\frac{1}{t}dt=\log(\Sigma_{i+1})-\log(\Sigma_{i}). 
		\end{equation}
		Combine (\ref{5.29}), (\ref{5.30}) and $\Sigma_{1}=\theta_0=1$, \begin{equation}\label{Tn'}
			T_n'\ge\frac{\widetilde{\beta}}{\log(n)} \sum_{i=1}^{n} \left(\log(\Sigma_{i+1})-\log(\Sigma_{i})\right)=\frac{\widetilde{\beta}}{\log(n)}*\log(\Sigma_{n+1}). 
		\end{equation}
		For any $1>\epsilon>\epsilon'>\epsilon_1$, by (\ref{Tn'}) we have 
		\begin{equation}\label{5.32}
			\begin{split}
				T_n(\epsilon)-T_n(\epsilon')=\int_{\epsilon'}^{\epsilon}T_n'(s)ds\ge &\frac{1}{\log(n)}\int_{\epsilon'}^{\epsilon} \widetilde{\beta}(s)\log(\Sigma_{n}(s))ds \\
				\ge &\frac{\log(\Sigma_{n+1}(\epsilon'))}{\log(n)}c'(\epsilon,\epsilon')\left(\epsilon-\epsilon'\right),
			\end{split}
		\end{equation}
		where $c'(\epsilon,\epsilon'):=\inf\{\widetilde{\beta}(s):s\in [\epsilon',\epsilon]\}>0$. Recall that $\lim\limits_{n\to \infty}T_n=\theta_\infty$. Taking upper limits in both two sides of (\ref{5.32}), since $\limsup\limits_{n\to \infty}\frac{\log(\Sigma_{n}(\epsilon'))}{\log(n)}\ge 1$, we have (\ref{5.25}) since that 
			\begin{equation}\label{5.33}
				\theta_\infty(\epsilon)\ge c(\epsilon,\epsilon')\left(\epsilon-\epsilon'\right)+\theta_\infty(\epsilon')>0.  
		\end{equation} 
		Recalling the definition of $\epsilon_c$ in (\ref{epsilonc}), by (\ref{5.33}) we have: \begin{equation}\label{ep1geepc}
			\epsilon_1\ge \epsilon_c>0. 
		\end{equation}
		
		On the other hand, for any $0<\epsilon<\epsilon_1$, there exist $\nu(\epsilon)>0$ and $M(\epsilon)\in \mathbb{N}^+$ such that for any $R\ge M$, $\Sigma_R\le R^{1-\nu}$. By (\ref{thetaR1}), 
		\begin{equation}\label{5.20}
			\frac{\theta_R'}{\theta_R}\ge \widetilde{\beta}(\epsilon)R^{\nu}. 
		\end{equation}	
		For any $0<\epsilon'< \epsilon$, by (\ref{5.20}) we have, \begin{equation}\label{intep}
			\ln(\theta_R(\epsilon))-\ln(\theta_R(\epsilon'))=\int_{\epsilon'}^{\epsilon}\frac{\theta_R'}{\theta_R}ds\ge \int_{\epsilon'}^{\epsilon}\widetilde{\beta}(s)R^{\nu} ds\ge c''(\epsilon',\epsilon)R^{\nu}(\epsilon-\epsilon'), 
		\end{equation}	
		where $c''(\epsilon',\epsilon):=\inf\{\widetilde{\beta}(s):\epsilon' \le s\le\epsilon \}>0$.	From (\ref{intep}), we have: for all $R\ge M$, 
		\begin{equation}
			\theta_R(\epsilon')\le \exp(\ln(\theta_R(\epsilon))-c''(\epsilon',\epsilon)R^\alpha(\epsilon-\epsilon'))\le e^{-c''(\epsilon',\epsilon)(\epsilon-\epsilon')R^{\nu}}
		\end{equation}
		and thus $\Sigma:=\sum_{R=0}^{\infty}\theta_R(\epsilon')\le M+ \sum_{R=M}^{\infty}e^{-c(\epsilon',\epsilon)(\epsilon-\epsilon')R^{\nu} }<\infty$. Hence, for any $0<\epsilon''<\epsilon'$, $s\in [\epsilon'',\epsilon']$ and $R\ge 1$, \begin{equation}\label{SigmaRs}
			\Sigma_R(s)\le \Sigma_R(\epsilon')\le \Sigma<\infty.
		\end{equation}
		Combining (\ref{thetaR1}) and (\ref{SigmaRs}), we have: for $s\in [\epsilon'',\epsilon']$,  
		\begin{equation}\label{fracthetaR}
			\frac{\theta_R'}{\theta_R}(s)\ge \frac{\widetilde{\beta}(s) R}{\Sigma}. 
		\end{equation}
		Taking integration from $\epsilon''$ to $\epsilon'$ on both sides of (\ref{fracthetaR}), 
		\begin{equation}
			\ln(\theta_R(\epsilon'))-\ln(\theta_R(\epsilon''))=\int_{\epsilon''}^{\epsilon'}\frac{\theta_R'}{\theta_R}ds\ge \frac{R}{\Sigma}\int_{\epsilon''}^{\epsilon'}\widetilde{\beta}(s)ds\ge c'''(\epsilon'',\epsilon')\frac{R}{\Sigma}(\epsilon'-\epsilon''),
		\end{equation}
		where $c'''(\epsilon'',\epsilon'):=\inf\{\widetilde{\beta}(s):\epsilon'' \le s\le\epsilon' \}>0$. Thus $\theta_R(\epsilon'')\le e^{-c'''(\epsilon'',\epsilon')(\epsilon'-\epsilon'')\Sigma^{-1}R}$.

		To summarize, we have: for any $0\le \epsilon<\epsilon_1$, there exists $\widetilde{c}(\epsilon)>0$ such that for any integer $R\ge 1$, \begin{equation}\label{5.24}
			\theta_R(\epsilon)\le e^{-\widetilde{c}R},
		\end{equation}
		which implies that $\epsilon_1\le \epsilon_c$. Recalling (\ref{ep1geepc}), we get $\epsilon_1=\epsilon_c$ and then the proof of (\ref{thetaR}) is completed. 
	\end{proof}

	In conclusion, we have finished the proof of Proposition \ref{ul}. \qed

	\subsection{From $\mathcal{FI}_L^{u,T}$ to $\mathcal{FI}^{u,T}$}\label{estimate_truncation}
	
	Recall the notation $L_n$ at the beginning of Section \ref{bridginglemmas}. Similar to (5.4) in \cite{duminil2020equality}, we define a function $\psi(t,u,R)$: for $R\in \mathbb{N}^+$, if $t\in \mathbb{N}^+$, 
	\begin{equation}
		\psi(t,u,R):=P\left(B(R) \xleftrightarrow[]{\mathcal{FI}_{L_n}^{u,T}} \partial B(2R)\right);
	\end{equation}
	otherwise, let $n=\lfloor t \rfloor$ and $\chi_t:=\mathcal{FI}_{L_n}^{u,T}+ \left(\mathcal{FI}_{L_{n+1}}^{u(t-n),T}-\mathcal{FI}_{L_{n}}^{u(t-n),T} \right)$,
	\begin{equation}
		\psi(t,u,R)=P\left(B(R) \xleftrightarrow[]{\chi_t} \partial B(2R)\right).
	\end{equation}

	We first note that for any $n\in \mathbb{N}^+$ and $t\in (n,n+1)$, $\psi(t,u,R)$ is an analytic function w.r.t. $(t,u)$. To confirm this fact, we consider an equivalence ``$\sim$'' between point measures on $W^{\left[ 0,\infty \right) }$: for point measures $\omega_1,\omega_2$ on $W^{\left[ 0,\infty \right) }$, say $\omega_1\sim\omega_2$ if $\left\lbrace \eta\in W^{\left[ 0,\infty \right) }:\omega_1(\eta)>0 \right\rbrace= \left\lbrace \eta\in W^{\left[ 0,\infty \right) }:\omega_2(\eta)>0 \right\rbrace$. In fact, the event $\left\lbrace B(R) \xleftrightarrow[]{\chi_t} \partial B(2R)\right\rbrace $ is measurable w.r.t. $\mathcal{F}:=\sigma\left(\sum_{(a_i,\eta_i)\in \mathcal{FI}^{T}}\delta_{\eta}\cdot \left(\mathbbm{1}_{|\eta|\le L_n,\eta(0)\in B(2R+L_n)}+\mathbbm{1}_{L_n<|\eta|\le L_{n+1},\eta(0)\in B(2R+L_{n+1})} \right) \right) $. Under the equivalence ``$\sim$'', there are only finitely many classes of configurations in $\mathcal{F}$. Then it is sufficient to show that the summed probability of any class is analytic w.r.t. $(u,t)$. For any given paths $\{\eta_1,...,\eta_m\}\subset W^{\left[ 0,\infty \right) }$ with lengths $\le L_{n+1}$ and equivalent class $\{\omega:\text{for any }1\le i\le m,\omega(\eta_i)>0 \}$, the summed probability of this class under the law of $\chi_t$ is $\prod_{i=1}^{m}P[\omega(\eta_i)\ge 1]$, where each term is analytic because $\omega(\eta_i)$ is a Poisson random variable with parameter $\frac{2du}{T+1}P^{(T)}_{\eta_i(0)}(\eta_i)$ or $\frac{2du(t-n)}{T+1}P^{(T)}_{\eta_i(0)}(\eta_i)$. In conclusion, $\psi(t,u,R)$ is an analytic function w.r.t. $(t,u)$.

	Inspired by \cite{de2015russo}, we can calculate partial derivatives of $\psi(t,u,R)$. 
	
	\begin{proposition}(Russo's formula)\label{russo}
		For any $n\in \mathbb{N}^+$, $t\in (n,n+1)$, $u>0$ and $R\in \mathbb{N}^+$,  
		\begin{equation}\label{russo1}
			\begin{split}
				\frac{\partial }{\partial u}\psi(t,u,R)=&\frac{2d}{T+1}\sum_{\eta\in W^{\left[0,\infty \right)}: |\eta|\le L_n }Q^{(T)}(\eta)\cdot P\left(  \boldsymbol{Piv}(\chi_t,\eta)\right)\\
				&+ \frac{2d(t-n)}{T+1}\sum_{\eta\in W^{\left[0,\infty \right)}: L_n< |\eta|\le L_{n+1} }Q^{(T)}(\eta)\cdot P\left( \boldsymbol{Piv}(\chi_t,\eta)\right)
			\end{split}
		\end{equation}
		and \begin{equation}\label{russo2}
			\begin{split}
				\frac{\partial }{\partial t}\psi(t,u,R)=\frac{2du}{T+1}\sum_{\eta\in W^{\left[0,\infty \right)}: L_n< |\eta|\le L_{n+1} }Q^{(T)}(\eta)\cdot P\left( \boldsymbol{Piv}(\chi_t,\eta)\right),
			\end{split}
		\end{equation}
		where $Q^{(T)}:=\sum_{x\in \mathbb{Z}^d}P^{(T)}_x$ and $\boldsymbol{Piv}(\chi_t,\eta):=\left\lbrace B(R) \xleftrightarrow[]{\chi_t\cup \eta} \partial B(2R) \right\rbrace \cap \left\lbrace B(R)  \overset{\chi_t}{\nleftrightarrow} \partial  B(2R) \right\rbrace $.
	\end{proposition}

	\begin{proof}[Proof]
		Recalling that $\psi(t,u,R)$ is an analytic function of $(t,u)$, we have \begin{equation}\label{562}
				\frac{\partial }{\partial u}\psi(t,u,R)=\lim\limits_{\delta\to 0+}\frac{1}{\delta}\left[\psi(t,u+\delta,R)-\psi(t,u,R)\right]. 
		\end{equation}
		
		Therefore, it is sufficient to calculate the limit in (\ref{562}). For any $\delta>0$, we denote by $\hat{N}^{(1)}_{x}(\delta)$ the number of paths in $\mathcal{FI}^{u+\delta,T}- \mathcal{FI}^{u,T}$ with starting point $x$ and length $\in \left[0,L_n\right]$, and by $\hat{N}^{(2)}_{x}(\delta)$ the number of paths in $\mathcal{FI}^{(u+\delta)(t-n),T}- \mathcal{FI}^{u(t-n),T}$ with starting point $x$ and length $\in \left(L_n,L_{n+1} \right]$. By definition of FRI , we have that $\left\lbrace \hat{N}^{(1)}_x\right\rbrace_{x\in \mathbb{Z}^d} $ and $\left\lbrace \hat{N}^{(2)}_x\right\rbrace_{x\in \mathbb{Z}^d} $ are sequences of Poisson random variables with parameters $\frac{2d\delta}{T+1}P_0^{(T)}\left(|\eta|\le L_n\right) $ and $\frac{2d\delta(t-n)}{T+1}P_0^{(T)}\left(L_n<|\eta|\le  L_{n+1}\right) $ respectively. Meanwhile, by property of Poisson point process (see Section 2.9.1 in \cite{streit2010poisson}), we also know that $\left\lbrace \hat{N}^{(1)}_x\right\rbrace_{x\in \mathbb{Z}^d}$, $\left\lbrace \hat{N}^{(2)}_x\right\rbrace_{x\in \mathbb{Z}^d} $ and $\chi_t$ are all independent to each other. Thus, 
		\begin{equation}\label{6.53}
			\begin{split}
				&\psi(t,u+\delta,R)-\psi(t,u,R)\\
				=&P\left[B(R) \overset{\chi_t(t,u)}{\nleftrightarrow} \partial B(2R),B(R) \xleftrightarrow[]{\chi_t(t,u+\delta)} \partial B(2R) \right]\\
				\ge &P\left[B(R) \overset{\chi_t(t,u)}{\nleftrightarrow} \partial B(2R),B(R) \xleftrightarrow[]{\chi_t(t,u+\delta)} \partial B(2R),\sum_{x\in B(2R+L_{n})}\hat{N}^{(1)}_x=1,\sum_{x\in B(2R+L_{n+1})}\hat{N}^{(2)}_x=0 \right]\\
				&+P\left[B(R) \overset{\chi_t(t,u)}{\nleftrightarrow} \partial B(2R),B(R) \xleftrightarrow[]{\chi_t(t,u+\delta)} \partial B(2R),\sum_{x\in B(2R+L_{n})}\hat{N}^{(1)}_x=0,\sum_{x\in B(2R+L_{n+1})}\hat{N}^{(2)}_x=1 \right]\\
				=&P\left[ \sum_{x\in B(2R+L_{n+1})}\hat{N}^{(2)}_x=0 \right] \frac{P\left[ \sum\limits_{x\in B(2R+L_{n})}\hat{N}^{(1)}_x=1\right] }{\sum\limits_{x\in B(2R+L_{n})}P_x^{(T)}(|\eta|\le L_n)}\sum_{\eta\in W^{\left[0,\infty \right)}:|\eta|\le L_n }Q^{(T)}(\eta) P\left[\boldsymbol{Piv}(\chi_t,\eta)\right] \\
				+&\left( P\left[ \sum_{x\in B(2R+L_{n})}\hat{N}^{(1)}_x=0]\right]\cdot\frac{P\left[ \sum\limits_{x\in B(2R+L_{n+1})}\hat{N}^{(2)}_x=1 \right] }{\sum\limits_{x\in B(2R+L_{n+1})}P_x^{(T)}(L_n<|\eta|\le L_{n+1})}\right) \\
				&\cdot\left( \sum_{\eta\in W^{\left[0,\infty \right)}:L_n<|\eta|\le L_{n+1} }Q^{(T)}(\eta)P\left[\boldsymbol{Piv}(\chi_t,\eta)\right]\right) .
			\end{split}
		\end{equation}
		Recalling that $\hat{N}^{(1)}_x\sim Pois(\frac{2d\delta}{T+1}P_0^{(T)}\left(|\eta|\le L_n\right) )$ and $\hat{N}^{(2)}_x\sim Pois(\frac{2d\delta(t-n)}{T+1}P_0^{(T)}\left(L_n<|\eta|\le L_{n+1}\right) )$ for all $x\in \mathbb{Z}^d$, we have 
		\begin{equation}\label{6.54}
			\lim\limits_{\delta\to 0+}\frac{P[\sum_{x\in B(2R+L_{n+1})}\hat{N}^{(2)}_x=1 ]}{\delta*\sum_{x\in B(2R+L_{n+1})}P_x^{(T)}(L_n<|\eta|\le L_{n+1})}=\frac{2d}{T+1},
		\end{equation}
		\begin{equation}\label{6.55}
			\lim\limits_{\delta\to 0+}\frac{P[\sum_{x\in B(2R+L_{n+1})}\hat{N}^{(2)}_x=1 ]}{\delta*\sum_{x\in B(2R+L_{n+1})}P_x^{(T)}(L_n<|\eta|\le L_{n+1})}=\frac{2d(t-n)}{T+1},
		\end{equation}
		and \begin{equation}\label{6.55.5}
				\lim\limits_{\delta\to 0+}P[\sum_{x\in B(2R+L_{n})}\hat{N}^{(1)}_x=0 ]=\lim\limits_{\delta\to 0+}P[\sum_{x\in B(2R+L_{n+1})}\hat{N}^{(2)}_x=0 ]=1.
			\end{equation}
			Combine (\ref{6.53})-(\ref{6.55.5}), 
		\begin{equation}\label{inf1}
			\begin{split}
				&\liminf\limits_{\delta\to 0+}\frac{1}{\delta}\left[\psi(t,u+\delta,R)-\psi(t,u,R)\right]\\
				\ge &\frac{2d}{T+1}\sum_{\eta\in W^{\left[0,\infty \right)}:|\eta|\le L_n}Q^{(T)}(\eta)P[\boldsymbol{Piv}(\chi_t,\eta)]+\frac{2d(t-n)}{T+1}\sum_{\eta\in W^{\left[0,\infty \right)}:L_n<|\eta|\le L_{n+1}}Q^{(T)}(\eta)P[\boldsymbol{Piv}(\chi_t,\eta)].
			\end{split}
		\end{equation}

		On the other hand, we also have 
		\begin{equation}
			\begin{split}
				&\theta(t,u+\delta,R)-\theta(t,u,R)\\
				\le &P\left[B(R) \overset{\chi_t(t,u)}{\nleftrightarrow} \partial B(2R),B(R) \xleftrightarrow[]{\chi_t(t,u+\delta)} \partial B(2R),\sum_{x\in B(2R+L_{n})}\hat{N}^{(1)}_x=1,\sum_{x\in B(2R+L_{n+1})}\hat{N}^{(2)}_x=0 \right]\\
				+&P\left[B(R) \overset{\chi_t(t,u)}{\nleftrightarrow} \partial B(2R),B(R) \xleftrightarrow[]{\chi_t(t,u+\delta)} \partial B(2R),\sum_{x\in B(2R+L_{n})}\hat{N}^{(1)}_x=0,\sum_{x\in B(2R+L_{n+1})}\hat{N}^{(2)}_x=1 \right]\\
				&+P\left[\sum_{x\in B(2R+L_{n})}\hat{N}^{(1)}_{x}=1 \right]\cdot P\left[\sum_{x\in B(2R+L_{n+1})}\hat{N}^{(2)}_{x}=1 \right]+P\left[\sum_{x\in B(2R+L_{n})}\hat{N}^{(1)}_{x}\ge 2 \right]\\
				&+ P\left[\sum_{x\in B(2R+L_{n+1})}\hat{N}^{(2)}_{x}\ge 2 \right].
			\end{split}
		\end{equation}
		Note that we have $P\left[\sum_{x\in B(2R+L_{n})}\hat{N}^{(1)}_{x}\ge 2 \right]=o(\delta)$, $P\left[\sum_{x\in B(2R+L_{n+1})}\hat{N}^{(2)}_{x}\ge 2 \right]=o(\delta)$, $P\left[\sum_{x\in B(2R+L_{n})}\hat{N}^{(1)}_{x}=1 \right]=O(\delta)$ and $P\left[\sum_{x\in B(2R+L_{n+1})}\hat{N}^{(2)}_{x}=1 \right]=O(\delta)$. Hence,	
		\begin{equation}\label{sup1}
			\begin{split}
				&\limsup\limits_{\delta\to 0+}\frac{1}{\delta}\left[\theta(t,u+\delta,R)-\theta(t,u,R)\right]\\
				\le  &\frac{2d}{T+1}\sum_{\eta\in W^{\left[0,\infty \right)}:|\eta|\le L_n}Q^{(T)}(\eta)P[\boldsymbol{Piv}(\chi_t,\eta)]+\frac{2d(t-n)}{T+1}\sum_{\eta\in W^{\left[0,\infty \right)}:L_n<|\eta|\le L_{n+1}}Q^{(T)}(\eta)P[\boldsymbol{Piv}(\chi_t,\eta)].
			\end{split}
		\end{equation}

		Combining (\ref{inf1}) and (\ref{sup1}), we get (\ref{russo1}). The proof of (\ref{russo2}) is parallel to (\ref{russo1}), which is therefore omitted.	
	\end{proof}

	\begin{lemma}\label{partial}
		For any $\epsilon>0$ and $U_0>u_{**}+\epsilon$, there exist $c_5(U_0,\epsilon),c_6(U_0),C_6(U_0,\epsilon),C_7(U_0)>0$ such that for all $n\in \mathbb{N}^+$, $n< t<n+1$, $u_{**}+\epsilon\le u\le U_0$, $L_0\ge C_6$ and $R\ge C_7$, 
		\begin{equation}\label{536}
			\frac{\partial}{\partial t}\psi(t,u,R)\le e^{-c_5L_{n}}\cdot\frac{\partial}{\partial u}\psi(t,u,R)+e^{-t}\cdot e^{-c_6R}. 
		\end{equation}
	\end{lemma}

	\begin{proof}[Proof]
		Let $n_1=\inf\{m:L_{m+1}\ge 0.5R\}$.
		
		\textbf{Case 1:} If $n\ge n_1$, by (\ref{russo2}), there exists $R_0(U_0)>0$ such that for any $R\ge R_0(U_0)$,  	
		\begin{equation}\label{537}
			\begin{split}
				\frac{\partial }{\partial t}\psi(t,u,R)=&\frac{2du}{T+1}\sum_{\eta\in W^{\left[0,\infty \right)}: L_n< |\eta|\le L_{n+1} }Q^{(T)}(\eta)P\left(\boldsymbol{Piv}(\chi_t,\eta)\right)	\\
				\le &\frac{2du}{T+1}\sum_{\eta\in W^{\left[0,\infty \right)}:\eta(0)\in B(2R+L_{n+1}) ,L_n< |\eta|\le L_{n+1} }Q^{(T)}(\eta)\\
				=&\frac{2du}{T+1}\sum_{x\in B(2R+L_{n+1})} P_x^{(T)}(L_n<|\eta|\le L_{n+1})\\
				\le &\frac{2dU_0}{T+1}\cdot\left(4R+2L_{n+1}+1\right)^d\cdot\left(\frac{T}{T+1}\right)^{L_n+1}\\
				\le &e^{-2cL_n}=e^{-cL_n}\cdot e^{-cL_n}\le e^{-\frac{cR}{2l_0}}\cdot e^{-t}. 
			\end{split}
		\end{equation} 
		
		\textbf{Case 2:} If $0\le n< n_1$, for any cluster $\mathcal{A}\subset \mathcal{B}(2R)$, denote by $\mathfrak{C}(\mathcal{A})$ be the collection of all clusters in $\chi_t\cap \left(\mathcal{B}(2R)\setminus \mathcal{B}(R)\right)$ intersecting $\mathcal{A}$. Then we define $\mathcal{C}_\mathcal{A}:=\mathcal{A}\cup \bigcup\limits_{\mathcal{C}\in \mathfrak{C}(\mathcal{A})}\mathcal{C}$. Note that $\mathcal{C}_\mathcal{A}$ is always connected. We also define $\mathcal{E}(R):=(\mathcal{B}(2R)\setminus \mathcal{B}(R)) \cup \partial_e \mathcal{B}(R)$. 
		
		For any disjoint clusters $\mathcal{C}_1,\mathcal{C}_2\subset \mathcal{E}(R) $ such that $\partial_e \mathcal{B}(R)\subset \mathcal{C}_1$ and $\partial_e \mathcal{B}(2R)\subset \mathcal{C}_2$, define the event $C(\mathcal{C}_1,\mathcal{C}_2)=\{\mathcal{C}_{\mathcal{B}(R)}=\mathcal{C}_1,\mathcal{C}_{\partial_e \mathcal{B}(2R)}=\mathcal{C}_2\}$. Note that the event $C(\mathcal{C}_1,\mathcal{C}_2)$ is measurable with respect to $\chi_t\cap (\bar{\mathcal{C}}_1\cup \bar{\mathcal{C}}_2)$, where $\bar{\mathcal{C}}:=\left(\mathcal{C}\cup \partial_e^{out}\mathcal{C}\right)\cap  \mathcal{E}(R)$.

		Arbitrarily fix a path $\eta\in W^{\left[0,\infty \right) }$ such that $\eta\cap \mathcal{E}(R)\neq \emptyset$ and that $L_n< |\eta|\le L_{n+1}$. Then denote $\Lambda_{n+1}^\eta:=\mathcal{B}_{\eta(0)}(10\kappa L_{n+1})$. For each $\mathcal{C}_1,\mathcal{C}_2$ s.t.  $C(\mathcal{C}_1,\mathcal{C}_2)\subset Piv(\chi_t,\eta)$, since $u\ge u_{**}+\epsilon$ and $C(\mathcal{C}_1,\mathcal{C}_2)\in \sigma\left(\sum_{(a_i,\eta_i)}\delta_{(a_i,\eta_i)}\cdot \mathbbm{1}_{0<u_i\le u,\eta\cap (\bar{\mathcal{C}}_1\cup\bar{\mathcal{C}}_2\cup (\Lambda_{n+1}^\eta)^c)\neq \emptyset} \right) $, we can use the same approach in Lemma \ref{bridgelemma} to prove 
		\begin{equation}\label{5.38}
			P\left[\bar{\mathcal{C}}_1  \xleftrightarrow[\Lambda_{n+1}^\eta]{\chi_t}     \bar{\mathcal{C}}_2 \bigg| C(\mathcal{C}_1,\mathcal{C}_2) \right] \ge e^{-C(\epsilon,L_0)(\log(L_{n+1}))^2}.
		\end{equation}
		
		If $(\bar{\mathcal{C}}_1\cap \bar{\mathcal{C}}_2)\cap \Lambda_{n+1}^\eta\neq \emptyset$, then there exists an edge $e_0$ in $\Lambda_{n+1}^\eta$ such that $\mathcal{C}_1$ and $\mathcal{C}_2$ are connected by $e_0$. Hence, we have 
		\begin{equation}\label{539}
			P[\bar{\mathcal{C}}_1  \xleftrightarrow[\Lambda_{n+1}^\eta]{\chi_t}     \bar{\mathcal{C}}_2,C(\mathcal{C}_1,\mathcal{C}_2) ]\le \sum_{e_0\in \Lambda_{n+1}^\eta}P\left[\boldsymbol{Piv}(\chi_t,e_0),C(\mathcal{C}_1,\mathcal{C}_2)\right].
		\end{equation}
		
		If $(\bar{\mathcal{C}}_1\cap \bar{\mathcal{C}}_2)\cap \Lambda_{n+1}^\eta=\emptyset$, and when event $\bar{\mathcal{C}}_1  \xleftrightarrow[\Lambda_{n+1}^\eta]{\chi_t} \bar{\mathcal{C}}_2 $ happens, there exist two different edges $e_i=\{y_i,z_i\}\in \Lambda_{n+1}^\eta$ ($i=1,2$) such that $z_1\xleftrightarrow[\Lambda_{n+1}^\eta]{\chi_t}z_2$, $y_1\xleftrightarrow[]{\chi_t} \partial B(R)$, $y_2\xleftrightarrow[]{\chi_t} \partial B(2R)$ and $\{e_1,e_2\}\cap \chi_t=\emptyset$. Recalling (\ref{insert}), we have 
			\begin{equation}\label{PC1C2}
				\begin{split}
					&P\left[\bar{\mathcal{C}}_1  \xleftrightarrow[\Lambda_{n+1}^\eta]{\chi_t}     \bar{\mathcal{C}}_2 ,C(\mathcal{C}_1,\mathcal{C}_2) \right]\\
					\le &\sum_{e_2\in\Lambda_{n+1}^\eta}\sum_{e_1\in \Lambda_{n+1}^\eta\setminus \{e_2\} }P\left[z_1\xleftrightarrow[\Lambda_{n+1}^\eta]{\chi_t}z_2,y_1\xleftrightarrow[]{\chi_t} \partial B(R),y_2\xleftrightarrow[]{\chi_t} \partial B(2R),C(\mathcal{C}_1,\mathcal{C}_2),e_1\notin \chi_t,e_2\notin \chi_t \right]\\
					\le &C'(\epsilon)\sum_{e_2\in\Lambda_{n+1}^\eta}\sum_{e_1\in \Lambda_{n+1}^\eta\setminus \{e_2\} }P\left[z_1\xleftrightarrow[\Lambda_{n+1}^\eta]{\chi_t}z_2,y_1\xleftrightarrow[]{\chi_t} \partial B(R),y_2\xleftrightarrow[]{\chi_t} \partial B(2R),C(\mathcal{C}_1,\mathcal{C}_2),e_1\in \chi_t,e_2\notin \chi_t \right].
				\end{split}
		\end{equation}
		Recall the event $\boldsymbol{Piv}(\chi_t,\cdot)$ in Proposition \ref{russo}. For each $e_2\in \Lambda_{n+1}^\eta$, noting that each event on the RHS of (\ref{PC1C2}) implies $\boldsymbol{Piv}(\chi_t,e_2)$, we have 
			\begin{equation}\label{e_1e_2}
				\begin{split}
					&\sum_{e_1\in \Lambda_{n+1}^\eta\setminus \{e_2\}}P\left[z_1\xleftrightarrow[\Lambda_{n+1}^\eta]{\chi_t}z_2,y_1\xleftrightarrow[]{\chi_t} \partial B(R),y_2\xleftrightarrow[]{\chi_t} \partial B(2R),C(\mathcal{C}_1,\mathcal{C}_2),e_1\in \chi_t,e_2\notin \chi_t \right]\\
					\le &\sum_{e_1\in \Lambda_{n+1}^\eta\setminus \{e_2\}}P\left[\boldsymbol{Piv}(\chi_t,e_2),C(\mathcal{C}_1,\mathcal{C}_2) \right]\\
					\le &C''(d)(L_{n+1})^{d}\cdot \sum_{e_1\in \Lambda_{n+1}^\eta\setminus \{e_2\}}P\left[\boldsymbol{Piv}(\chi_t,e_2),C(\mathcal{C}_1,\mathcal{C}_2) \right]. 
				\end{split}
		\end{equation}
		Combining (\ref{5.38})-(\ref{e_1e_2}), we have: there exists $M_1(\epsilon)>0$ such that for all $L_0\ge M_1$, 
		\begin{equation}\label{541}
			\begin{split}
				P[\boldsymbol{Piv}(\chi_t,\eta)]\le& e^{C(\log(L_{n+1}))^2}\sum_{\mathcal{C}_1,\mathcal{C}_2:C(\mathcal{C}_1,\mathcal{C}_2)\subset Piv(\chi_t,\eta)}P\left[ \bar{\mathcal{C}}_1  \xleftrightarrow[\Lambda_{n+1}^\eta]{\chi_t}     \bar{\mathcal{C}}_2,C(\mathcal{C}_1,\mathcal{C}_2)\right] \\
				\le&\max\{C'C''(L_{n+1})^{d},1\}e^{C(\log(L_{n+1}))^2}\cdot\sum_{e\in \Lambda_{n+1}^\eta}P\left[\boldsymbol{Piv}(\chi_t,e)\right]\\
				\le &e^{2C(\log(L_{n+1}))^2}\cdot\sum_{e\in \Lambda_{n+1}^\eta}P\left[\boldsymbol{Piv}(\chi_t,e)\right].
			\end{split}
		\end{equation}
		By Proposition \ref{russo} and (\ref{541}), there exists $M_2(\epsilon,U_0)>M_1$ such that for any $L_0\ge M_2$, 
		\begin{equation}\label{542}
				\begin{split}
					\frac{\partial }{\partial t}\psi(t,u,R)=&\frac{2du}{T+1}\sum_{\eta\in W^{\left[0,\infty \right)}: L_n< |\eta|\le L_{n+1} }Q^{(T)}(\eta)P\left(\boldsymbol{Piv}(\chi_t,\eta)\right)\\
					\le &\frac{2dU_0}{T+1}e^{2C(\log(L_{n+1}))^2}\sum_{\eta\in W^{\left[0,\infty \right)}: L_n< |\eta|\le L_{n+1} }Q^{(T)}(\eta)\sum_{e\in \Lambda_{n+1}^\eta}P\left[\boldsymbol{Piv}(\chi_t,e)\right]\\
					\le &\frac{2dU_0}{T+1}e^{2C(\log(L_{n+1}))^2}\sum_{e\in \mathcal{B}(R)}P\left[\boldsymbol{Piv}(\chi_t,e)\right]\sum_{\eta\in W^{\left[0,\infty \right)}:L_n< |\eta|\le L_{n+1},d(e,\eta(0))\le 10\kappa L_{n+1} }Q^{(T)}(\eta)\\
					\le &\frac{2dU_0}{T+1}e^{2C(\log(L_{n+1}))^2}\sum_{e\in \mathcal{B}(R)}P\left[\boldsymbol{Piv}(\chi_t,e)\right]\sum_{z\in \mathbb{Z}^d:d(e,z)\le 10\kappa L_{n+1} }P_z^{(T)}(L_n<|\eta|\le L_{n+1})\\
					\le &\frac{2dU_0}{T+1}e^{2C(\log(L_{n+1}))^2}(20\kappa L_{n+1}+2)^{d} \left(\frac{T}{T+1}\right)^{L_n+1} \sum_{e\in \mathcal{B}(R)}P\left[\boldsymbol{Piv}(\chi_t,e)\right]\\
					\le &e^{-c'L_n}\frac{2d}{T+1}\sum_{e=\{x,y\}\in \mathcal{B}(R)}\left( Q^{(T)}[\eta=(x,y)]+Q^{(T)}[\eta=(y,x)]\right) P\left[\boldsymbol{Piv}(\chi_t,e)\right]\\
					\le &e^{-c'L_n}\frac{\partial }{\partial u}\psi(t,u,R).
				\end{split}
		\end{equation}

		Combining (\ref{537}) and (\ref{542}), we get (\ref{536}).
	\end{proof}

	\subsection{Proving SQ4}
	
	Before concluding the proof of \textbf{SQ4}, we prove the following corollary of Lemma \ref{partial}:
	\begin{lemma}\label{lemma9}
		For any $\epsilon>0$ and $u\ge u_{**}+\epsilon$, there exist integers $L_0(u,\epsilon)$ and $m(u,\epsilon)>0$ such that for any $R\ge C_8(u)$, 
		\begin{equation}\label{543}
			P\left[B(R) \xleftrightarrow[]{\mathcal{FI}_{L_m}^{u,T}} B(2R) \right]	\le 	P\left[B(R) \xleftrightarrow[]{\mathcal{FI}^{u,T}} B(2R) \right] \le P\left[B(R) \xleftrightarrow[]{\mathcal{FI}_{L_m}^{u+\epsilon,T}} B(2R) \right]+e^{-c_6R}.
		\end{equation} 
	\end{lemma}
	
	\begin{proof}[Proof]
		The left inequality in (\ref{543}) is a direct corollary of the fact that $\mathcal{FI}_{L_m}^{u,T}\subset \mathcal{FI}^{u,T}$.

		Consider a function $f(t,u,R):=\psi(t,u+e^{-t},R)$. Note that for any $t>0$, $u+e^{-t}< u+1$. By Lemma \ref{partial}, for any $n\in \mathbb{N}^+$, $t\in (n,n+1)$, $L_0\ge C_6$ and $R\ge C_7$, 
		\begin{equation}
			\begin{split}
				\frac{\partial}{\partial t}f(t,u,R)=&\frac{\partial}{\partial t}\psi(t,u+e^{-t},R)-e^{-t}\frac{\partial}{\partial u}\psi(t,u+e^{-t},R)\\
				\le  &(e^{-c_5L_n}-e^{-n})\frac{\partial}{\partial u}\psi(t,u+e^{-t},r,t)+e^{-c_6R}\cdot e^{-t}.
			\end{split}
		\end{equation}
		Recalling that $L_n=L_n\cdot l_0^n$, there exists $m_0$ such that for all $n=\lfloor t\rfloor \ge m_0$, $e^{-c_5L_n}-e^{-n}\le 0$. Then for all $t\ge m_0$, 
		\begin{equation}
			\begin{split}
				\frac{\partial}{\partial t}f(t,u,R)\le e^{-c_6R}\cdot e^{-t}.
			\end{split}
		\end{equation}
		Therefore, for all $n\ge m_0$, \begin{equation}\label{5.39}
			\begin{split}
				&\psi(n+1,u+e^{-(n+1)},R)-\psi(n,u+e^{-n},R)\\
				=&\int_{n}^{n+1}\frac{\partial}{\partial t}f(t,u,R)dt\\
				\le &\int_{n}^{n+1}e^{-c_6\sqrt{R}}\cdot e^{-t}dt=e^{-c_6\sqrt{R}}(e^{-n}-e^{-(n+1)}). 
			\end{split}
		\end{equation}
		By (\ref{5.39}) and $\lim\limits_{t\to \infty}\theta(t,u+e^{-t},R)=P\left[B(R) \xleftrightarrow[]{\mathcal{FI}^{u,T}} \partial B(2R) \right]$, for any $m\ge m_0$, we have  
		\begin{equation}
			\begin{split}
				&P\left[B(R) \xleftrightarrow[]{\mathcal{FI}^{u,T}} \partial B(2R) \right]-\psi(m,u+e^{-m},R)\\
				=&\sum_{n=m}^{\infty}\left[ \psi(n+1,u+e^{-(n+1)},R)-\psi(n,u+e^{-n},R)\right] \\
				\le &e^{-c_6\sqrt{R}}\sum_{n=m}^{\infty}\left( e^{-n}-e^{-(n+1)}\right) =e^{-c_6\sqrt{R}}\cdot e^{-m}. 
			\end{split}
		\end{equation}
		Hence, for all $m\ge m_0$,
		\begin{equation}\label{549}
			\begin{split}
				P\left[B(R) \xleftrightarrow[]{\mathcal{FI}^{u,T}} \partial B(2R) \right] \le P\left[B(R) \xleftrightarrow[]{\mathcal{FI}_{L_m}^{u+e^{-m},T}} \partial B(2R) \right]+e^{-c_6\sqrt{R}}.
			\end{split}
		\end{equation}
		
		Choose integer $m\ge m_0$ such that $e^{-m}<\epsilon$ in (\ref{549}), then we get (\ref{543}).	
	\end{proof}

With Lemma \ref{lemma9}, we are ready to conclude the proof of \textbf{SQ4}.
	
	\begin{proof}[Proof of \textbf{SQ4}]
		By contradiction, suppose that $u_{**}< \widetilde{u}$. Take $u_0>0$ and $\epsilon>0$ such that $u_{**}<u_0-2\epsilon<u_0<u_0+2\epsilon< \widetilde{u}$. And note that for any $L\in \mathbb{N}^+$, $\widetilde{u}\le \widetilde{u}^L=u_{**}^L$.

		By Lemma \ref{lemma9} and the fact that  $u_{**}<u_0<u_0+\epsilon<u_{**}^{L_m}$, there exist integers $L_0$ and $m$ such that 
			\begin{equation}\label{0lim}
				0<\limsup\limits_{R\to \infty}P\left[B(R) \xleftrightarrow[]{\mathcal{FI}^{u_0,T}} \partial B(2R) \right]\le \limsup\limits_{R\to \infty}\left\lbrace P\left[B(R) \xleftrightarrow[]{\mathcal{FI}^{u_0+\epsilon,T}_{L_m}} \partial B(2R) \right]+e^{-c_6R}\right\rbrace=0, 
			\end{equation}
			which is a contradiction. In conclusion, we finish the proof of \textbf{SQ4}.
	\end{proof}

	\section{Proof of Corollaries}\label{corollaries}
	
	In Section \ref{corollaries}, we give the proofs for Corollary \ref{sharpness}-\ref{coro3}. 
	
	\subsection{Proof of Corollary \ref{sharpness}}

	To prove Corollary \ref{sharpness}, we need a more detailed version of the renormalization arguement in \cite{rodriguez2013phase}. Here we need some notations introduced in \cite{popov2015soft}:
	
	\begin{itemize}
		\item Fix a constant $b\in \left(1,2\right]   $ and a positive integer $J_1\ge 100$. 
		
		\item For any $k\ge 1$, let $J_{k+1}=2\left( 1+\frac{1}{(k+5)^b} \right)  J_{k} $ and $\mathbb{J}_k=J_k\cdot \mathbb{Z}^d$;	by (7.3) in \cite{popov2015soft}, we have: for any $k\in \mathbb{N}^+$, 
		\begin{equation}\label{6.20}
			J_12^{k-1}\le J_k\le e^{\zeta(b)}J_12^{k-1}, 
		\end{equation}
		where $\zeta(b):=\sum_{j=1}^{\infty}\frac{1}{(j+5)^b}$.
		
		\item For $x\in \mathbb{Z}^d$ and $k\in \mathbb{N}^+$, let $C_x^k=\left[0,J_k \right)^d\cap \mathbb{Z}^d+x$ and $D_x^k=\left[-J_k,2J_k \right)^d\cap \mathbb{Z}^d+x$; we also write that $\mathcal{C}_x^k:=\left\lbrace e=\{x,y\}\in \mathbb{L}^d:x,y\in C_x^k\right\rbrace $ and $\mathcal{D}_x^k:=\left\lbrace e=\{x,y\}\in \mathbb{L}^d:x,y\in D_x^k\right\rbrace $. 
		
		\item For $x\in \mathbb{Z}^d$, $k\in \mathbb{N}^+$ and $u>0$, define the event $A_x^k(u):=\left\lbrace C_x^k \xleftrightarrow[]{\mathcal{FI}^{u,T}} \mathbb{Z}^d\setminus D_x^k\right\rbrace $.

		\item For $k\in \mathbb{N}^+$, consider two collections of vertices $\{x_i^k\}_{i=1}^{3^d}$ and $\{y_i^k\}_{j=1}^{2d*7^{d-1}}$ as follows:
		\begin{itemize}
			\item $\{x_i^k\}_{i=1}^{3^d}\subset \mathbb{J}_k$ satisfies that $C_0^{k+1}=\bigcup_{i=1}^{3^d}C_{x_i^k}^{k}$;
			
			\item $\{y_i^k\}_{j=1}^{2d*7^{d-1}}$ satisfies that $ \bigcup_{j=1}^{2d*7^{d-1}}C_{y_j^k}^k$ is disjoint from $D_0^{k+1}$ and contains $\partial \left(\mathbb{Z}^d\setminus D_0^{k+1}\right)  $.
		\end{itemize}
		By (7.8) of \cite{popov2015soft}, we have 
		\begin{equation}\label{6.21}
			A_0^{k+1}\subset \bigcup_{i\le 3^d,j\le 2d*7^{d-1}} \left(A^k_{x_i^k}\cap A^k_{y_j^k} \right). 
		\end{equation}
		See Figure \ref{figure5} for an illustration of this renormalization scheme.
		
		\begin{figure}[h]
			\centering
			\includegraphics[width=0.5\textwidth]{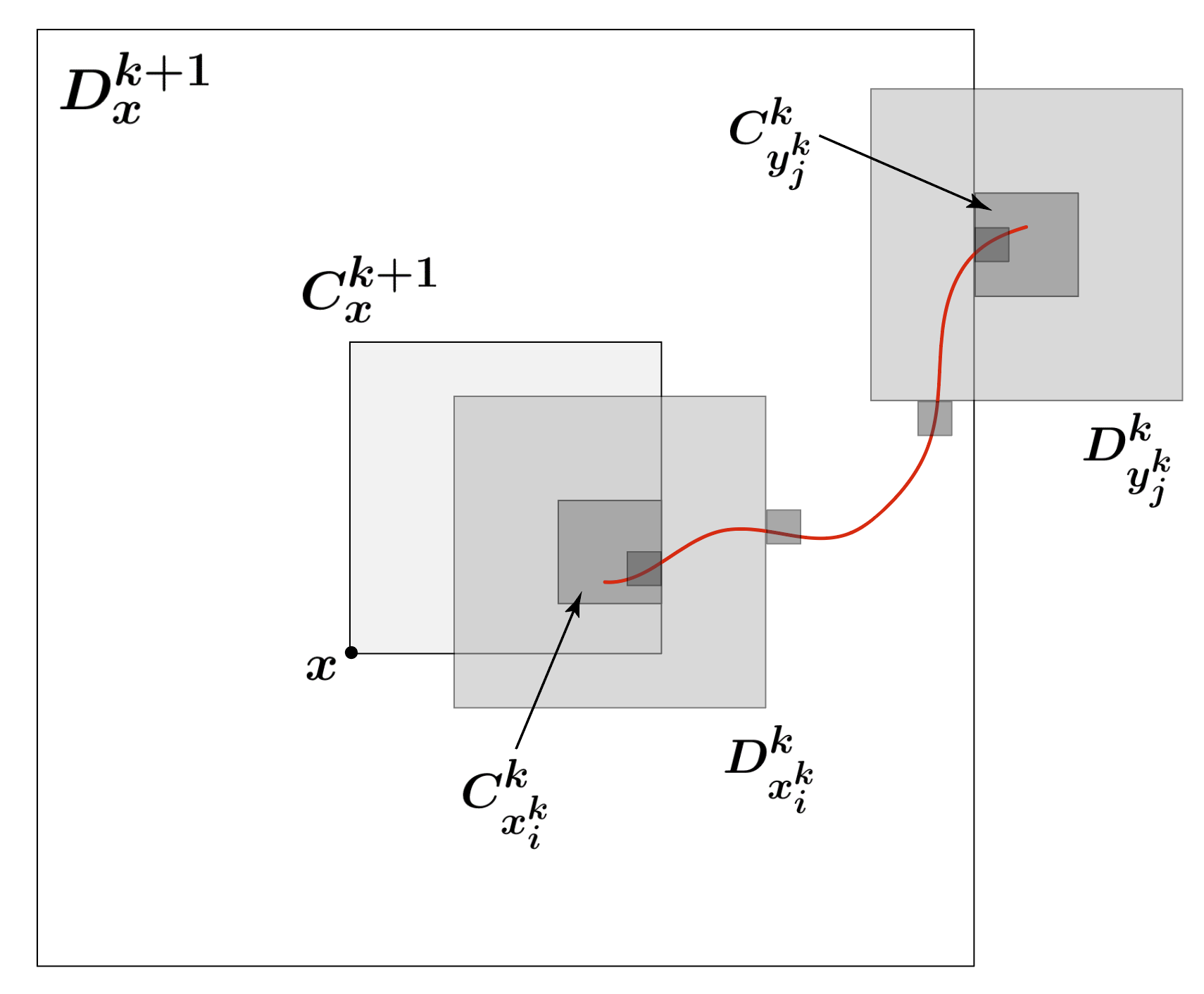}
			\caption{An
				illustration of the renormalization scheme.}
			\label{figure5}
		\end{figure}

		\item For $x\in \mathbb{Z}^d$ and $k\in \mathbb{N}^+$, let $E_x^k:=\{z\in \mathbb{Z}^d:d(z,D_x^k)\le \frac{J_k}{(k+5)^b}\}$. Then we define the restricted FRI $\widehat{\mathcal{FI}}^{u,T}_{x,k}:=\sum_{\eta_i\in \mathcal{FI}^{u,T}}\delta_{\eta_i}\cdot\mathbbm{1}_{\eta_i(0)\in E_x^k} $ and consider the following restriction of the event $A_x^{k}$: 
			\begin{equation}
				\widehat{A}_x^k(u):=\left\lbrace C_x^k \xleftrightarrow[]{\widehat{\mathcal{FI}}^{u,T}_{x,k}} \mathbb{Z}^d\setminus D_x^k\right\rbrace.
		\end{equation}
		
	\end{itemize}

		Before proving Corollary \ref{sharpness}, we first introduce the following estimate about the diameter of geometrically killed random walks, similar to Lemma 5.1 in \cite{cai2021some}. 
		
		\begin{lemma}\label{lemma12}
			For $d\ge 3$ and $T>0$, assume that $ \left\lbrace X_i^{(T)} \right\rbrace $ is a geometrically killed random walks with law $P_x^{(T)}$. Then there exist constants $C_9(d,T)$, $c_7(d,T)>0$ such that for all $n\in \mathbb{N}^+$,
			\begin{equation}\label{lemma12.0}
				P_0^{(T)}\left[\max\left\lbrace \left| X_i^{(T)}\right| \right\rbrace \ge n\right] \le C_9e^{-c_7n^{\frac{4}{3}}}.
			\end{equation}
		\end{lemma}
		\begin{proof}
			We denote the length of $ \left\lbrace X_i^{(T)} \right\rbrace $ by $\hat{N}^{(T)}$. Note that $\hat{N}^{(T)}\sim Geo(\frac{T}{T+1})$. Since $\left\lbrace \max\limits_{0\le i\le \hat{N}^{(T)}}\left\lbrace \left| X_i^{(T)}\right| \right\rbrace\ge n \right\rbrace\subset \left\lbrace \hat{N}^{(T)}\ge n \right\rbrace   $, we have: for any $\delta>0$, \begin{equation}\label{lemma12.1}
				\begin{split}
					&P_0^{(T)}\left[\max\limits_{0\le i\le \hat{N}^{(T)}}\left\lbrace \left| X_i^{(T)}\right| \right\rbrace\ge n \right]\\
					\le &P_0^{(T)}\left[\hat{N}^{(T)}\ge n^{2-\delta} \right]+ P_0^{(T)}\left[n\le \hat{N}^{(T)}<n^{2-\delta},\max\limits_{0\le i\le \hat{N}^{(T)}}\left\lbrace \left| X_i^{(T)}\right| \right\rbrace\ge n \right]\\
					\le &\left(\frac{T}{T+1}\right)^{n^{2-\delta}}+ P_0^{(T)}\left[n\le \hat{N}^{(T)}<n^{2-\delta}\right]\cdot P\left[\max\limits_{0\le i\le n^{2-\delta}}\left\lbrace \left| X_i\right| \right\rbrace\ge n \right],
				\end{split}
			\end{equation}
			where $\left\lbrace X_i \right\rbrace_{i=0}^{\infty}  $ is a simple random walk on $\mathbb{Z}^d$ with starting point $0$. 
			
			By Theorem 1.5.1 in \cite{lawler2013intersections}, there exists $c(d)>0$ such that for all $n\in \mathbb{N}^+$, \begin{equation}\label{lemma12.2}
				P\left[\max\limits_{0\le i\le n^{2-\delta}}\left\lbrace \left| X_i\right| \right\rbrace\ge n \right]\le c\cdot e^{-n^{\frac{\delta}{2}}}.
			\end{equation}
			Combine (\ref{lemma12.1}) and (\ref{lemma12.2}), \begin{equation}\label{lemma12.3}
				\begin{split}
					P_0^{(T)}\left[\max\limits_{0\le i\le \hat{N}^{(T)}}\left\lbrace \left| X_i^{(T)}\right| \right\rbrace\ge n \right]\le e^{-c'(T)n^{2-\delta}}+c\cdot e^{-c'(T)n}\cdot  e^{-n^{\frac{\delta}{2}}}.
				\end{split}
			\end{equation}
			Taking $\delta=\frac{2}{3}$ in (\ref{lemma12.3}),  then we get (\ref{lemma12.0}).
	\end{proof}

	Now we are ready to give the proof of Corollary \ref{sharpness}:
	
	\begin{proof}[Proof of Corollary \ref{sharpness}]
			For any $k\ge 1$, by (\ref{6.21}) we have 
			\begin{equation}\label{6.25}
				P\left( A_0^{k+1}(u)\right)\le \sum_{1\le i\le 3^d,1\le j\le 2d*7^{d-1}}P\left(A^k_{x_i^k}(u)\cap A^k_{y_j^k}(u) \right). 
			\end{equation}
			
			For any $x\in \mathbb{Z}^d$ and integer $k\ge 0$, we denote by $F_x^k$ the event that all paths in $\mathcal{FI}^{u,T}-\widehat{\mathcal{FI}}^{u,T}_{x,k}$ do not intersect $D_x^k$. By Lemma \ref{lemma12}, we have the following estimate for probability of $F_x^k$: 
			\begin{equation}\label{625}
				\begin{split}
					P\left[\left( F_x^k \right) ^c\right]=&1-\exp(-\frac{2du}{T+1}\sum_{z\in (E_x^k)^c}P_z^{(T)}\left[ \left\lbrace X_i^{(T)} \right\rbrace \ \text{intersects}\ D_x^k\right]  )\\
					\le &\frac{2du}{T+1}\sum_{z\in (E_x^k)^c}P_z^{(T)}\left[  \left\lbrace X_i^{(T)} \right\rbrace \ \text{intersects}\ D_x^k\right]\\
					\le &\frac{2du}{T+1}\sum_{m=\lceil \frac{J_k}{(k+5)^b} \rceil}^{\infty}C(d)\left(m+\frac{3}{2}J_k \right)^{d-1}P_z^{(T)}\left[  \max\limits_{0\le i\le \hat{N}^{(T)}}\left\lbrace \left| X_i^{(T)}\right| \right\rbrace\ge m\right]\\
					\le &\frac{2du}{T+1}\sum_{m=\lceil \frac{J_k}{(k+5)^b} \rceil}^{\infty}C(d)\left(m+\frac{3}{2}J_k \right)^{d-1}C_9e^{-c_7m^{\frac{4}{3}} }\\
					\le &C'(d,u,T)\sum_{m=\lceil \frac{J_k}{(k+5)^b} \rceil}^{2J_k}(4J_k)^{d-1}e^{-c_7m^{\frac{4}{3}} }+C'(d,u,T)\sum_{m=2J_k}^{\infty}(2m)^{d-1}e^{-c_7m^{\frac{4}{3}} }.
				\end{split}
			\end{equation} 
			For the first term in the RHS of (\ref{625}), by (\ref{6.20}) there exists integer $M_1$ such that for all $J_1\ge M_1$ and $k\ge 1$,
			\begin{equation}\label{626}
				\begin{split}
					C'\sum_{m=\lceil \frac{J_k}{(k+5)^b} \rceil}^{2J_k}(4J_k)^{d-1}e^{-c_7m^{\frac{4}{3}} }\le C'(4J_k)^{d-1}e^{-c_7\left( \lceil \frac{J_k}{(k+5)^b} \rceil\right)^{\frac{4}{3}} }\sum_{m=0}^{2J_k-\lceil \frac{J_k}{(k+5)^b} \rceil}e^{-c_7m^{\frac{4}{3}}}\le \frac{1}{2}e^{-J_12^{k}}.
				\end{split}
			\end{equation}
			For the second term, by (\ref{6.20}) there exists integer $M_2$ such that for all $J_1\ge M_2$ and $k\ge 1$,
			\begin{equation}\label{627}
				\begin{split}
					C'\sum_{m=2J_k}^{\infty}(2m)^{d-1}e^{-c_7m^{\frac{4}{3}} }=&C'e^{-c_7(2J_k)^{\frac{4}{3}}}\sum_{m=0}^{\infty}(2m+4J_k)^{d-1}e^{-c_7m^{\frac{4}{3}}}\\
					\le &C'(4J_k)^{d-1}e^{-c_7(2J_k)^{\frac{4}{3}}}\sum_{m=0}^{\infty}(m+1)^{d-1}e^{-c_7m^{\frac{4}{3}}}\le \frac{1}{2}e^{-J_12^{k}}.
				\end{split}
			\end{equation}
			Combining (\ref{625}), (\ref{626}) and (\ref{627}), we have: for $J_1\ge M_3:=\max\{M_1,M_2\}$ and $k\ge 1$, 
			\begin{equation}\label{628}
				P\left[\left( F_x^k \right) ^c\right]\le e^{-\frac{c2^{k}J_1}{4(k+5)^b}}.
			\end{equation}

			For any $1\le i\le 3^d$ and $1\le j\le 2d*7^{d-1}$, since $E^k_{x_i^k}\cap E^k_{y_j^k}=\emptyset$, $\widehat{\mathcal{FI}}^{u,T}_{x_i,k}$ and $\widehat{\mathcal{FI}}^{u,T}_{y_j,k}$ are independent. By (\ref{628}), $\widehat{A}_x^k\subset A_x^k$ and $ A_x^k\cap F_x^k\subset \widehat{A}_x^k$, for each term in the RHS of (\ref{6.25}),
			\begin{equation}\label{629}
				\begin{split}
					&P\left(A^k_{x_i^k}(u)\cap A^k_{y_j^k}(u) \right)\\
					\le &P\left(A^k_{x_i^k}(u)\cap A^k_{y_j^k}(u)\cap F_{x_i^k}^k\cap F_{y_j^k}^k \right)+P\left[ \left(F_{x_i^k}^k\right)^c \right] +P\left[ \left(F_{y_j^k}^k\right)^c \right]\\
					\le &P\left(\widehat{A}^k_{x_i^k}(u)\cap \widehat{A}^k_{y_j^k}(u) \right)+2e^{-J_12^{k}}\\
					= &P\left(\widehat{A}^k_{x_i^k}(u)\right) P\left(  \widehat{A}^k_{y_j^k}(u) \right)+2 e^{-J_12^{k}}\\
					\le &P\left(A^k_{x_i^k}(u)\right) P\left(  A^k_{y_j^k}(u) \right)+2 e^{-J_12^{k}}.
				\end{split}
			\end{equation}
			Combine (\ref{6.25}) and (\ref{629}),
			\begin{equation}\label{630}
				\begin{split}
					&P\left( A_0^{k+1}(u)\right)+2 e^{-J_12^{k}}\\
					\le&
					\sum_{1\le i\le 3^d,1\le j\le 2d*7^{d-1}}\left[ P\left(A^k_{x_i^k}(u)\right) P\left(  A^k_{y_j^k}(u) \right)+4e^{-J_12^{k}}\right] \\
					\le &\sum_{1\le i\le 3^d,1\le j\le 2d*7^{d-1}} \left[P\left(A^k_{x_i^k}(u)\right)+2 e^{-J_12^{k-1}} \right]\cdot \left[P\left(  A^k_{y_j^k}(u) \right)+2e^{-J_12^{k-1}} \right].
				\end{split}
			\end{equation}
			By (\ref{630}) and induction, we have: for all $k\in \mathbb{N}^+$, 
			\begin{equation}\label{6.27}
				P\left( A_0^{k+1}(u)\right)\le P\left( A_0^{k+1}(u)\right)+2 e^{-J_12^{k}}\le (3^d\cdot2d\cdot7^{d-1})^{2^{k+1}}\left( P\left(A^1_{0}(u)\right)+2e^{-J_1} \right)^{2^k}.  
			\end{equation}
			
		Let $\widetilde{J}_1:=\lfloor \frac{J_1}{2}\rfloor $ and $z_0=(\widetilde{J}_1,\widetilde{J}_1,...,\widetilde{J}_1)\in \mathbb{Z}^d$. Since $C_0^1\subset B_{z_0}(\widetilde{J}_1+1)$ and $B_{z_0}(2\widetilde{J}_1+2)\subset D_0^1$, we have $P\left(A^1_{0}(u)\right)\le P\left(B_{z_0}(\widetilde{J}_1+1)\xleftrightarrow[]{\mathcal{FI}^{u,T}}\partial B_{z_0}(2\widetilde{J}_1+2)\right) $. For any $u<u_{*}=u_{**}$, recalling that $\inf_{R\in \mathbb{N}^+}P\left(B(R)\xleftrightarrow[]{\mathcal{FI}^{u,T}}\partial B(2R) \right)=0 $, there must exist a certain $J_1\ge M_3$ such that 
			\begin{equation}\label{6.28}
				\begin{split}
					&	\left( 3^d\cdot2d\cdot7^{d-1}\right)^2\cdot\left( P\left(A^1_{0}(u)\right)+2^{-J_1}\right)\\
						\le &\left( 3^d\cdot2d\cdot7^{d-1}\right)^2\cdot\left( P\left(B_{z_0}(\widetilde{J}_1+1)\xleftrightarrow[]{\mathcal{FI}^{u,T}}\partial B_{z_0}(2\widetilde{J}_1+2)\right)+2^{-J_1}\right)<\frac{1}{2}. 
				\end{split}
			\end{equation}
		We arbitrarily take a $J_1$ such that (\ref{6.28}) holds. By (\ref{6.27}) and (\ref{6.28}), we have: for all $k\in \mathbb{N}^+$, 
			\begin{equation}\label{6.29}
				P\left(A_0^{k+1}(u)\right)
				\le 2^{-2^k}. 
			\end{equation}

			For each integer $N>J_2$, assume that $N\in \left[J_{k_0+1},J_{k_0+2}\right)$, $k_0\in \mathbb{N}^+$. Noting that for all $z\in \{0,-1\}^d$, $D_{J_{k_0}z}^{k_0}\subset B(N)$, by (\ref{6.20}), (\ref{6.29}) and $N<J_{k_0+2}\le 9J_{k_0}$, we have \begin{equation}\label{637}
				\begin{split}
					P\left(0\xleftrightarrow[]{\mathcal{FI}^{u,T}} B(N) \right)\le \sum_{z\in \{0,-1\}^d}P\left(C_{J_{k_0}x}^{k_0}\xleftrightarrow[]{\mathcal{FI}^{u,T}}\mathbb{Z}^d\setminus D_{J_{k_0}x}^{k_0} \right)
					\le 2^d\cdot2^{-2^{k_0}}\le 2^d\cdot2^{-\frac{2N}{9J_1e^{\zeta(b)}}}.
				\end{split}
			\end{equation}
		From (\ref{637}), we finish the proof of Corollary \ref{sharpness}.  
		\end{proof}

%	 an increasing sequence of positive integers $\{n_j\}_{j=1}^{\infty}$ such that $\lim\limits_{j\to \infty }P\left(B(n_j)\xleftrightarrow[]{\mathcal{FI}^{u,T}}\partial B(2n_j) \right)=0$ 
%			
%			
%			
%			
%			
%			
%			
%			
%			
%			
%			 since  $P\left(A^1_{0}(u)\right)\le \sum_{z\in \partial C_0^1}P\left(z\xleftrightarrow[]{\mathcal{FI}^{u,T}}\partial B_z(J_1-1)\right)$, 
%			
%			
%			
%			
%			
%			we have that $\lim\limits_{J_1\to \infty}P\left(A^1_{0}(u)\right)=0$. Thus, there exists integer $M_4\ge M_3$ such that for all $J_1\ge M_4$,

	\subsection{Proof of Corollary \ref{coro3}}
	
	\begin{proof}[Proof of Corollary \ref{coro3}]
		Recalling the definition of $T_*^+(u)$ in (\ref{T*+}), by Theorem \ref{equality} we know that $\mathcal{FI}^{u,T}$ strongly percolates for all $T>T_*^+(u)$. Hence, it is sufficient to show that for any $d\ge3$, $u>0$, $T_*(u)^+\in (0,\infty)$.

		By Theorem 2 of \cite{procaccia2021percolation}, there exists $T_0>0$ such that for all $T<T_0$, $\mathcal{FI}^{u,T}$ does not percolate. Thus we have $T_*^+(u)\ge T_0>0$. 
		
		On the other hand, if $T_*^+(u_0)=+\infty$ for a certain $u_0>0$, there must exist an increasing sequence $\{T_n\}_{n=1}^{\infty}$ such that $T_n\to \infty$ and that for all $n\ge 1$, $u_0= u_*(T_n)$. By (v) of Theorem 3 in \cite{cai2021some}, we have: there exists $C,C',c>0$ such that for all $T\ge C$, $u_*(T)\le C'T^{-c}$. Therefore, for all large enough $n$, $u_*(T_n)\le C'T_n^{-c}$, which is contradictory to $u_0= u_*(T_n)$. 
		
		In conclusion, $0<T_*^+(u)<\infty$ for all $u>0$. 
	\end{proof}

	\appendix
	
	\section{FKG Inequality for $\mathcal{FI}^{u,T}$ and $\mathcal{FI}_L^{u,T}$}\label{appendixFKG}
	
	In this section, we prove FKG Inequality for both $\mathcal{FI}^{u,T}$ and $\mathcal{FI}_L^{u,T}$ by using the approach in Section 2.2, \cite{grimmett1999percolation}.

	\begin{proposition}[FKG Inequality for $\mathcal{FI}^{u,T}$]\label{FKG}
		If $f$ and $g$ are both increasing functions (or both decreasing functions) on $\{0,1\}^{\mathbb{L}^d}$ such that $Ef^2,Eg^2<\infty$, then \begin{equation}\label{A.1}
			E\left[f(\mathcal{FI}^{u,T})g(\mathcal{FI}^{u,T}) \right] \ge E\left[f(\mathcal{FI}^{u,T})\right]E\left[g(\mathcal{FI}^{u,T})\right].
		\end{equation}
	\end{proposition}
\begin{remark}
		Especially, for two increasing events (or two decreasing events) $A$ and $B$, $\mathbbm{1}_{A}$ and $\mathbbm{1}_{B}$ are both increasing events (or decreasing events). Hence, by Proposition \ref{FKG} we have 
		\begin{equation}
			P[A\cap B]\ge P[A]P[B].
		\end{equation}
\end{remark}
	
To prove Lemma \ref{FKG}, we need the following lemma as preparation:
	
	\begin{lemma}\label{lemmaW}
		Enumerate all elements in $W^{\left[0,\infty \right) }$ by $\{\eta_i\}_{i=1}^{\infty}$. For $\omega\in \mathscr{W}$, define countable 0-1 valued random variable $\{\psi_i\}_{i=1}^{\infty}$, where $\psi_i(\omega)=1$ if and only if $ \omega(\eta_i)\ge1$. Then for any $n\ge 1$, \begin{equation}\label{a4}
			P^{u,T}[\psi_1=...=\psi_n=0]=\prod_{i=1}^{n}P^{u,T}[\psi_i=0].
		\end{equation}
		As a result, $\{\psi_i\}_{i=1}^{\infty}$ are independent.
	\end{lemma}
	
	\begin{proof}[Proof of Lemma \ref{lemmaW}]
		Let $(i_0,...,i_m)$ be the unique integer array satisfying $1=i_0<i_1<...<i_m=n$ and following properties: $i_{1}:=\max\{j_0\in [1,n]:for\ all\ i_0\le j\le j_0, \eta_j(0)=\eta_1(0)\}$; for any $k\ge 1$, $i_{k+1}:=\max\{j_0\in [i_k+1,n]:for\ all\ i_k+1\le j\le j_0, \eta_j(0)=\eta_{i_k+1}(0)\}$. 
		
		Since paths in $\mathcal{FI}^{u,T}$ with different starting points are independent, we have \begin{equation}\label{a5}
			P^{u,T}[\psi_1=...=\psi_n=0]=P^{u,T}[\psi_1=...=\psi_{i_1}=0]\cdot\prod_{k=1}^{m-1}P^{u,T}[\psi_{i_k+1}=...=\psi_{i_{k+1}}=0].
		\end{equation}
		
		In each term in the RHS of (\ref{a5}), all related paths have the same starting point, so it is sufficient to prove (\ref{a4}) in the case $\eta_1(0)=...=\eta_n(0)$. Since for any $x\in \mathbb{Z}^d$, the number of paths in $\mathcal{FI}^{u,T}$ starting from $x$ is a Poisson random variable with parameter $\frac{2du}{T+1}$, we have 
		\begin{equation}\label{a6}
			\begin{split}
				&P^{u,T}[\psi_1=...=\psi_n=0]\\
				=&\sum_{q=0}^{\infty}\exp(-\frac{2du}{T+1})\cdot\frac{\left(\frac{2du}{T+1}\right)^q }{q!}\cdot\left(P^{(T)}\left(\text{for all}\ 1\le i\le n,\eta\neq \eta_i \right)  \right)^{q}\\
				=&\sum_{q=0}^{\infty}\exp(-\frac{2du}{T+1})\cdot\frac{\left(\frac{2du}{T+1}\right)^q }{q!}\cdot\left(1-\sum_{i=1}^{n}P^{(T)}\left(\eta= \eta_i \right)  \right)^{q}\\
				=&\exp(-\frac{2du}{T+1}\cdot\sum_{i=1}^{n}P^{(T)}\left(\eta= \eta_i \right))\\
				=&\prod_{i=1}^{n}\exp(-\frac{2du}{T+1}\cdot P^{(T)}\left(\eta= \eta_i \right))=\prod_{i=1}^{n}P^{u,T}[\psi_i=0]. 
			\end{split}
		\end{equation}  
		
		Combining (\ref{a5}) and (\ref{a6}), we get Lemma \ref{lemmaW}. 
	\end{proof}

Now we are ready to show the proof of Proposition \ref{FKG}:
	\begin{proof}
		Without loss of generality, we only prove Proposition \ref{FKG} in the case when $f$ and $g$ are both increasing functions.
		
		First, note that $f$ and $g$ are both measurable w.r.t. $\sigma\left(\psi_i,i\ge 1\right) $. For any integer $n\ge 1$, define that $f_n:=E[f|\psi_1,...,\psi_n]$ and $g_n:=E[g|\psi_1,...,\psi_n]$. By Theorem 4.6.8 in \cite{durrett2019probability}, we have $f_n$ and $g_n$ almost surely converge to $f$ and $g$ respectively, as $n\to \infty$.
		
		We claim that for any $n\ge1$, $f_n$ and $g_n$ are both increasing functions on $\{0,1\}^n$ (i.e. for any sequences $(\psi_1',...,\psi_n'),(\psi_1'',...,\psi_n'')\in \{0,1\}^{n}$ such that $\psi_i'\ge \psi_i''$ for all $1\le i\le n$, one has $f_n(\psi_1',...,\psi_n')\ge f_n(\psi_1'',...,\psi_n'')$ and $g_n(\psi_1',...,\psi_n')\ge g_n(\psi_1'',...,\psi_n'')$). Without loss of generality, we only check it for $f$: since that $f$ is increasing and that $\sigma(\psi_1,...,\psi_n)$ is independent with $\sigma(\psi_{n+1},\psi_{n+2},...)$, we have 
		\begin{equation}
			\begin{split}
				f_n\left(\psi_1',...,\psi_n'\right)-  f_n\left(\psi_1'',...,\psi_n''\right)                            
				= E\left[f\left(\psi_1',...,\psi_n',\psi_{n+1},...\right)-f\left(\psi_1'',...,\psi_n'',\psi_{n+1},...\right)\right] 
				\ge 0.
			\end{split}
		\end{equation}
	
		Now we are going to prove that for any $n\ge1$ and increasing functions $h_1,h_2$ on $\{0,1\}^n$, 
		\begin{equation}\label{A.5}
			E\left[h_1(\psi_1,...,\psi_n)h_2(\psi_1,...,\psi_n) \right] \ge E\left[h_1(\psi_1,...,\psi_n)\right]E\left[h_2(\psi_1,...,\psi_n)\right].
		\end{equation}

We prove (\ref{A.5}) by induction. When $n=1$, denote that $p_m=P[\psi_m=1]$ for all $m\ge 1$. Then we have 
		\begin{equation}
			\begin{split}
				&E\left[ h_1(\psi_1)h_2(\psi_1)\right] -E\left[ h_1(\psi_1)\right] E\left[ h_2(\psi_1)\right] \\
				=&p_1h_1(1)h_2(1)+(1-p_1)h_1(0)h_2(0)-\left[p_1h_1(1)+(1-p_1)h_1(0)\right]\cdot\left[p_1h_2(1)+(1-p_1)h_2(0)\right]\\
				=&p_1(1-p_1)\left(h_1(1)-h_1(0) \right) \left(h_2(1)-h_2(0) \right)\ge 0.
			\end{split}
		\end{equation}
		For $n\ge2$, using inductive hypothesis and Lemma \ref{lemmaW}, 
		\begin{equation}
			\begin{split}
				&E\left[ h_1(\psi_1,...,1)h_2(\psi_1,...,\psi_n)\right] -E\left[ h_1(\psi_1,...,\psi_n)\right] E\left[h_2(\psi_1,...,\psi_n)\right] \\
				=&p_nE\left[ h_1(\psi_1,...,1)h_2(\psi_1,...,1)\right] +(1-p_n)E\left[ h_1(\psi_1,...,0)h_2(\psi_1,...,0)\right] \\
				&-\left\lbrace p_nE\left[h_1(\psi_1,...,1)\right] +(1-p_n)E[h_1(\psi_1,...,0)] \right\rbrace\\
				&\ \ \ \ \cdot\left\lbrace p_nE\left[ h_2(\psi_1,...,1)\right] +(1-p_n)E\left[ h_2(\psi_1,...,0)\right]  \right\rbrace\\
				\ge &p_nE\left[ h_1(\psi_1,...,1)\right] E\left[ h_2(\psi_1,...,1)\right] +(1-p_n)E\left[ h_1(\psi_1,...,0)\right] E\left[ h_2(\psi_1,...,0)\right] \\
				&-\left\lbrace p_nE\left[ h_1(\psi_1,...,1)\right] +(1-p_n)E\left[ h_1(\psi_1,...,0)\right]  \right\rbrace\\
				&\ \ \ \ \cdot\left\lbrace p_nE\left[ h_2(\psi_1,...,1)\right] +(1-p_n)E\left[ h_2(\psi_1,...,0)\right]  \right\rbrace\\
				=&p_n(1-p_n)\left\lbrace E\left[ h_1(\psi_1,...,1)\right] -E\left[ h_1(\psi_1,...,0)\right]  \right\rbrace \left\lbrace E\left[ h_2(\psi_1,...,1)\right] -E\left[ h_2(\psi_1,...,0)\right]  \right\rbrace \ge 0.
			\end{split}
		\end{equation}
		In conclusion, we finish the induction and get (\ref{A.5}). 
		
		For any $n\ge 1$, let $h_1=f_n$ and $h_2=g_n$ in (\ref{A.5}). Then we have 
		\begin{equation}\label{A.10}
			E\left[f_n(\mathcal{FI}^{u,T})g_n(\mathcal{FI}^{u,T}) \right] \ge E\left[f_n(\mathcal{FI}^{u,T})\right]E\left[g_n(\mathcal{FI}^{u,T})\right].
		\end{equation}
		Taking limits in both sides of (\ref{A.10}), then we complete the proof of Proposition \ref{FKG}.	
\end{proof}

	For $\mathcal{FI}^{u,T}_L$, proof of FKG inequality is parallel. So we just leave it here and omit its proof.
	
	\begin{proposition}[FKG Inequality for $\mathcal{FI}_L^{u,T}$]\label{FKG2}
		Assume that $f$ and $g$ are both increasing functions (or both decreasing functions) on $\{0,1\}^{\mathbb{L}^d}$ such that $Ef^2,Eg^2<\infty$, then for any $L\in \mathbb{N}^+$,  \begin{equation}
			E\left[f(\mathcal{FI}_L^{u,T})g(\mathcal{FI}_L^{u,T}) \right] \ge E\left[f(\mathcal{FI}_L^{u,T})\right]E\left[g(\mathcal{FI}_L^{u,T})\right].
		\end{equation}
		
	\end{proposition}

	\bibliographystyle{plain}
	\bibliography{ref}
	
\end{document}